\documentclass[reqno,11pt]{amsart}
\usepackage[utf8]{inputenc}
\usepackage{amsmath, latexsym, amsfonts, amssymb, amsthm, amscd, dsfont}
\usepackage{graphics, epsf, psfrag}
\usepackage{mathrsfs, stmaryrd}
\usepackage{MnSymbol}
\usepackage[dvipsnames]{xcolor}
\usepackage{url}

\numberwithin{equation}{section}

\newtheorem{theorem}{Theorem}[section]
\newtheorem{lemma}[theorem]{Lemma}
\newtheorem{proposition}[theorem]{Proposition}

\newtheorem{rem}[theorem]{Remark}
\newtheorem{definition}[theorem]{Definition}

\newtheorem{example}[theorem]{Example}
\newtheorem{assumption}[theorem]{Assumption}

\newcommand{\R}{\mathbb{R}}
\newcommand{\Z}{\mathbb{Z}}

\renewcommand{\tilde}{\widetilde}
\xdefinecolor{red}{named}{Red}

\newcommand{\norm}[1]{\left\lVert#1\right\rVert}

\newcommand{\cA}{{\ensuremath{\mathcal A}} }

\newcommand{\cF}{{\ensuremath{\mathcal F}} }
\newcommand{\cP}{{\ensuremath{\mathcal P}} }

\newcommand{\cC}{{\ensuremath{\mathcal C}} }

\newcommand{\cL}{{\ensuremath{\mathcal L}} }
\newcommand{\cT}{{\ensuremath{\mathcal T}} }

\newcommand{\cU}{{\ensuremath{\mathcal U}} }
\newcommand{\cV}{{\ensuremath{\mathcal V}} }

\newcommand{\cG}{{\ensuremath{\mathcal G}} }

\newcommand{\bE}{{\ensuremath{\mathbf E}} }

\DeclareMathSymbol{\leqslant}{\mathalpha}{AMSa}{"36} 
\DeclareMathSymbol{\geqslant}{\mathalpha}{AMSa}{"3E} 
\DeclareMathSymbol{\eset}{\mathalpha}{AMSb}{"3F}     
\renewcommand{\leq}{\;\leqslant\;}                   
\renewcommand{\geq}{\;\geqslant\;}                   
\newcommand{\dd}{\,\text{\rm d}}             

\newcommand{\bbC}{{\ensuremath{\mathbb C}} }

\newcommand{\bbE}{{\ensuremath{\mathbb E}} }

\newcommand{\bbP}{{\ensuremath{\mathbb P}} }

\newcommand{\bbT}{{\ensuremath{\mathbb T}} }

\newcommand{\bbZ}{{\ensuremath{\mathbb Z}} }

\newcommand{\ga}{\alpha}

\newcommand{\gd}{\delta}
\newcommand{\gep}{\varepsilon}       

\newcommand{\gs}{\sigma}

\makeatletter
\def\captionfont@{\footnotesize}
\def\captionheadfont@{\scshape}

\long\def\@makecaption#1#2{%
\vspace{2mm}
\setbox\@tempboxa\vbox{\color@setgroup
\advance\hsize-6pc\noindent
\captionfont@\captionheadfont@#1\@xp\@ifnotempty\@xp
{\@cdr#2\@nil}{.\captionfont@\upshape\enspace#2}%
\unskip\kern-6pc\par
\global\setbox\@ne\lastbox\color@endgroup}%
\ifhbox\@ne 
\setbox\@ne\hbox{\unhbox\@ne\unskip\unskip\unpenalty\unkern}%
\fi
\ifdim\wd\@tempboxa=\z@ 
\setbox\@ne\hbox to\columnwidth{\hss\kern-6pc\box\@ne\hss}%
\else 
\setbox\@ne\vbox{\unvbox\@tempboxa\parskip\z@skip
\noindent\unhbox\@ne\advance\hsize-6pc\par}%
\fi
\ifnum\@tempcnta<64 
\addvspace\abovecaptionskip
\moveright 3pc\box\@ne
\else 
\moveright 3pc\box\@ne
\nobreak
\vskip\belowcaptionskip
\fi
\relax
}
\makeatother
\def\writefig#1 #2 #3 {\rlap{\kern #1 truecm
\raise #2 truecm \hbox{#3}}}

\usepackage{tikz}

\makeatletter
\newcommand\RSloop{\@ifnextchar\bgroup\RSloopa\RSloopb}
\makeatother
\newcommand\RSloopa[1]{\bgroup\RSloop#1\relax\egroup\RSloop}
\newcommand\RSloopb[1]%
{\ifx\relax#1%
\else
\ifcsname RS:#1\endcsname
\csname RS:#1\endcsname
\else
\GenericError{(RS)}{RS Error: operator #1 undefined}{}{}%
\fi
\expandafter\RSloop
\fi
}
\newcommand\X{0}
\newcommand\RS[1]%
{\begin{tikzpicture}
[every node/.style=
{circle,draw,fill,minimum size=1.5pt,inner sep=0pt,outer sep=0pt},
line cap=round
]
\coordinate(\X) at (0,0);
\RSloop{#1}\relax
\end{tikzpicture}
}
\newcommand\RSdef[1]{\expandafter\def\csname RS:#1\endcsname}
\newlength\RSu
\newlength\RSv
\RSdef{n}{\draw (\X) node{};}
\RSdef{l}{\draw (\X) -- +(-35:\RSu) node{};}
\RSdef{r}{\draw (\X) -- +(35:\RSu) node{};}
\RSdef{m}{\draw (\X) node{} -- +(-35:\RSu);}
\RSdef{u}{\draw (\X)  +(90:\RSu) coordinate(\X I);\edef\X{\X I}}
\RSdef{v}{\draw (\X)  +(90:\RSv) coordinate(\X I);\edef\X{\X I}}
\RSdef{Z}{\draw (\X) -- +(35:\RSu) coordinate(\X Z) node{};\edef\X{\X Z}}
\RSdef{T}{\draw (\X) -- +(135:\RSu) coordinate(\X T);\edef\X{\X T}}
\RSdef{D}{\draw (\X) -- +(0:\RSu) coordinate(\X D) node{};\edef\X{\X D}}

\renewcommand{\ij}{\RS{uD}}
\newcommand{\ijik}{\RS{ulr}}
\newcommand{\ijki}{\RS{umr}}
\newcommand{\ijjk}{\RS{uDD}}
\newcommand{\ikjk}{\RS{vZT}}
\newcommand{\lijik}{\RS{uDlr}}
\newcommand{\lijk}{\RS{uDDD}}

\def\<#1, #2>{\mathinner{\langle#1, \, #2\rangle}}

\title[CLT for empirical measures of diffusions on random graphs]{Central Limit Theorems for global and local empirical measures of diffusions on Erd\H{o}s-R\'enyi graphs}

\author{Fabio Coppini}
\address{Dipartimento di Matematica Ulisse Dini, Università degli Studi di Firenze, Viale Giovanni Battista Morgagni, 67/a, Florence 50134, Italy, \url{fabio.coppini@unifi.it}}

\author{Eric Lu\c{c}on}
\address{Université Paris Cité, CNRS, MAP5, F-75006 Paris \& FP2M, CNRS FR 2036, \url{eric.lucon@u-paris.fr}.
}

\author{Christophe Poquet}
\address{Université Claude Bernard Lyon 1, CNRS UMR 5208, Institut Camille Jordan, F-69622 Villeurbanne, France, \url{poquet@math.univ-lyon1.fr}}

\keywords{Mean-field systems, Erd\H{o}s-Rényi random graphs, interacting diffusions, central limit theorem, empirical measures, Grothendieck inequalities}
\subjclass[2010]{60F05,60H15,60K37,82C20}

\date{\today}

\begin{document}

\begin{abstract} 
We address the issue of the Central Limit Theorem for (both local and global) empirical measures of diffusions interacting on a possibly diluted Erd\H{o}s-R\'enyi graph. Special attention is given to the influence of initial condition (not necessarily i.i.d.) on the nature of the limiting fluctuations. We prove in particular that the fluctuations remain the same as in the mean-field framework when the initial condition is chosen independently from the graph. We give an example of non-universal fluctuations for carefully chosen initial data that depends on the graph. A crucial tool for the proof is the use of extensions of Grothendieck inequality.
\end{abstract}

\maketitle


\section{Introduction}
Fix $n=2, 3, \dots$ and let $\xi^{(n)}=\left( \xi^{(n)}_{ij} \right)_{ij=1, \dots, n}$ be a collection of independent and identically distributed Bernoulli random variables, with parameter $p_n \in (0,1]$ (i.e. $\xi^{(n)}$ defines an asymmetric Erd\H{o}s-Rényi random graph of parameter $p_n$). Let $T>0$ be a finite time horizon. We are interested in the empirical measure of a weakly interacting particle system where each particle is represented by a function on the 1-dimensional torus $\bbT:= \mathbb{ R}/2 \pi \mathbb{ Z}$. The population dynamics is defined by the following system of stochastic differential equations
\begin{equation}
\label{eq:wips}
\dd \theta^{i,n}_t = \frac{1}{n p_n}\sum_{j=1}^n \xi_{ij}^{(n)}\Gamma \left(\theta^{i,n}_t, \theta^{j,n}_t \right)\dd t+\dd B^i_t, \quad 0 < t \leq T, \quad i= 1, \dots, n,
\end{equation}
endowed with some initial condition $\left(\theta^{ i, n}_0\right)_{i=1,\ldots, n}$ and where $(B^i)_{i\geq 1}$ are independent and identically distributed standard Brownian motions. The dynamics of $ \theta^{ i,n}$ is influenced by the $\theta^{j,n}$ with $i$ and $j$ neighbors in the graph, through some regular function $\Gamma: \bbT\times\bbT\to \R$. The interaction in \eqref{eq:wips} is renormalized by the (uniform) expected degree $np_n$ of each vertex so that the interaction remains of order $1$ as $n\to\infty$.  We denote by $\mathbb{P}$ the joint law of the graph and the initial condition, and by $\mathbf{P}$ the law of the Brownian motions, so we will be working with $\bbP\otimes \mathbf{P}$. Moreover, we will denote by $\bbP_g$ and $\bbP_0$ the marginals of $ \mathbb{ P}$ of the graph and initial condition respectively.
\begin{rem}
One could very well consider the more general dynamics
\begin{equation}
\label{eq:wips_gen}
\dd \theta^{i,n}_t = F(\theta_{ t}^{ i, n}) {\rm d}t + \frac{1}{n p_n}\sum_{j=1}^n \xi_{ij}^{(n)}\Gamma \left(\theta^{i,n}_t, \theta^{j,n}_t \right)\dd t+\dd B^i_t, \quad 0 < t \leq T, \quad i= 1, \dots, n,
\end{equation}
where $F$ is a (bounded) smooth function on $ \mathbb{ T}$ modelling intrinsic dynamics for each particle. The result of the present paper remain obviously valid, up to the notational cost of adding e.g. some drift term $- \partial_{ \theta} \left[ \mu_{ t}(\theta) F(\theta)\right] $ in the nonlinear Fokker-Planck equation \eqref{eq:limit PDE}. We chose to restrict to $F\equiv 0$ for simplicity of exposition.
\end{rem}

An easy instance of \eqref{eq:wips} corresponds to the mean-field case, i.e. when $p_n \equiv 1$ so that $\xi^{(n)}_{ij} \equiv 1$ for all $i$ and $j$ (hence the graph of interaction is the complete graph). In such a case, the interaction in \eqref{eq:wips} is a functional of the empirical measure of $(\theta^{i,n})_{i=1,\dots,n}$ defined as
\begin{equation}
\label{eq:emp_meas}
\mu^n_t = \frac 1n \sum_{i=1}^n \delta_{\theta^{i,n}_t},\ t\in [0, T].
\end{equation}
The empirical measure is a (random) probability measure on $\bbT$. The behavior of $ \mu^{ n}$ as $ n\to\infty$ in the mean-field case is standard and particularly well-covered in the literature (see e.g. \cite{Sznitman1991,Gartner}): provided that $\mu^n_0$ converges to $\mu_0$, one can show that $ \mu^{ n}$ converges to the unique solution to the following non-linear Fokker-Planck equation
\begin{equation}\label{eq:limit PDE}
\partial_t \mu_t(\theta) = \frac12 \partial^2_{\theta}\mu_t(\theta)-\partial_\theta\left[\mu_t(\theta) (\Gamma*\mu_t)(\theta)\right], \quad 0 < t \leq T.
\end{equation}
In \eqref{eq:limit PDE}, the $*$ denotes the integration with respect to the second variable, i.e., $(\Gamma * \mu) (\cdot) = \int_\bbT \Gamma(\cdot, \theta) \mu (\dd \theta)$ for $\mu$ a probability measure on $\bbT$. The convergence to \eqref{eq:limit PDE} can be for instance considered in the space of probability measure on continuous functions on the torus, i.e., in $\cP(\cC([0,T], \bbT))$. In general, it depends on the topology where the sequence $(\mu^n)_{n\geq 1}$ is studied.

\medskip

A recent interest has been shown in the literature concerning extensions as \eqref{eq:wips} to generic graphs of interactions. Observe that when $(\xi_{ij}^{ (n)})$ is no longer constantly equal to $1$, the interaction is not a linear functional of the empirical measure \eqref{eq:emp_meas}, but rather of a collection of the local empirical measures defined by
\begin{equation}
\label{def:mu_nl}
\mu_{ t}^{ n,l}:= \frac{ 1}{ np_{ n}} \sum_{ i=1}^{ n} \xi_{li}^{ (n)} \delta_{ \theta_{t}^{ i, n}},\ t\in [0, T],\ l=1,\ldots, n.
\end{equation}
Contrary to the mean-field case where one easily obtains by Ito's formula a closed semimartingale decomposition of $ \mu^{ n}$ (e.g., \cite{neunzert81,Sznitman1991}), applying the same calculation in the general case (see, e.g., Lemma~\ref{lem:mu^n_t}) shows that $ \mu^{ n,l}$ depends itself on empirical measures involving higher order expansions within the graph structure: a whole hierarchy of empirical measures (indexed by local patterns in the graph) appear and the difficulty is to find a way to properly close this decomposition. However, if the graph $( \xi_{ij}^{ (n)})$ is sufficiently close to the complete graph (in a way to be precised), the behavior of the system \eqref{eq:wips} as $n\to\infty$ should be described as well by the same macroscopic limit \eqref{eq:limit PDE}. A rather informal question would be to understand how universal the mean-field framework is: how much can we perturb the complete graph of interaction and still conserve the same macroscopic properties as $n\to\infty$? This question has been addressed in details at the level of the law of large numbers: a series of recent papers \cite{Delattre2016,coppini2018law,Lucon2020,bayraktar2020graphon,BBW19} have shown that $ \mu^{ n}$ converges to $ \mu$ for a large class of graphs that includes the Erd\H{o}s-R\'enyi case. We refer to Section~\ref{ss:literature} for a more detailed discussion on these results.

Note that one may be interested in two distinct graph regimes: the \emph{dense case} when $p_{ n} \to p\in(0,1]$ as $n\to\infty$, i.e., the mean degree of each vertex remains proportional to the size of the population, and the \emph{diluted case} (or \emph{vanishing density degree}) when $ p_{ n}\to0$ as $n\to\infty$ and up to the sparse threshold $np_n \to c > 0$, where other phenomena are known to be in play \cite{Oliveira2020,Lacker2019}.
\subsection{Aim of this work}
The purpose of the present work is to address the universality of the mean-field framework at the level of fluctuations. We are interested in studying the limit of two objects: first of all, the \emph{global fluctuation process} $\eta^n$, given by
\begin{equation}
\label{def:fluctuation_process}
\eta^n_t := \sqrt{n} (\mu^n_t - \mu_t), \quad \text{for } 0 \leq t \leq T,
\end{equation}
i.e., the standard fluctuations of $\mu^n$ around the limit $\mu$. Note that $ \eta^{ n}$ not only depends on the randomness coming from the noise in \eqref{eq:wips}, but also on the graph $\xi^{(n)}$. We will consider the behavior of $ \eta^{ n}$ under the law $ \mathbf{ P}$ only, i.e., as a quenched object with respect to the graph sequence realisation. In addition, we will include the challenging case when the initial condition depends on the underlying graph. To the authors' knowledge, this issue has never been tackled. We refer to Section~\ref{ss:literature} for an overview on the existing literature.

\medskip

The second aim of this paper concerns the \emph{fluctuations of local empirical measures} around their limit. Recall the definition of the local empirical measure $ \mu^{ n, l}$ in \eqref{def:mu_nl}. This process is the empirical measure of particles at distance $1$ of vertex $l$ within the graph. Note that $ \mu^{ n,l}$ is not necessarily of mass $1$ (one only has $ \left\langle \mu^{ n,l}\, ,\, 1\right\rangle \xrightarrow[ n\to\infty]{}1$, $ \mathbb{ P}_{ g}$-a.s.), as \eqref{def:mu_nl} is not renormalised by the degree $d_{ n,l}:= \sum_{ j=1}^{ n} \xi_{lj}^{ (n)}$ of vertex $l$, but rather by its expectation $\mathbb{ E} \left[d_{ n,l}\right]=np_{ n}$. This choice of renormalisation turns out to be convenient for our fluctuation results, as it permits to focus only on the fluctuations that come from the dynamics and not to bother with intrinsic fluctuations in the vertex degrees in the graph. The influence of this last choice of renormalising constant on our results is discussed in details in the Appendix~\ref{sec:renorm}. It turns out that the natural limit for the local empirical measures \eqref{def:mu_nl} is also given by $ \mu$ solution to \eqref{eq:limit PDE}. Hence, the second aim of the paper will be to address the behavior of the joint fluctuation process
\begin{equation}
\label{def:zetas}
\zeta_{ t}^{ n}:= \left(\zeta_{ t}^{ n, 1}, \zeta_{ t}^{ n, 2}\right):= \left( \sqrt{ n p_{ n}} \left( \mu_{ t}^{ n, 1} - \mu_{ t}\right), \sqrt{ n p_{ n}} \left( \mu_{ t}^{ n, 2} - \mu_{ t}\right)\right),
\end{equation}
that is the joint fluctuation process for the local empirical measures around vertices $1$ and $2$. One point will be to understand how the limits of both $ \eta^{ n}$ and $ \zeta^{ n}$ may or may not depend on a specific realisation of the graph and, secondly, on the graph structure itself (in particular the fact that the graph may be dense of diluted).

\subsection{Strategy of proof}
To prove the convergence of fluctuations, we follow the classical trilogy of arguments: tightness of $\eta^n$, identification of the limiting values of  $\eta^n$ as solution to a suitable partial differential equation and uniqueness of the limit solution. The main tool to prove the first two steps is a semimartingale decomposition in a suitable Hilbert space of distributions (see, e.g., \cite{mitoma85, hitsudaMitoma86, Fernandez1997}). From this point of view, we closely follow the strategy proposed in \cite{Fernandez1997}. As already said, there is no possible way to obtain a closed semimartingale decomposition for the empirical processes $\mu^n$ and $\mu^{n,l}$ (and thus for $\eta^n$): we have supplementary terms that depend on expansions of higher order within the graph. Our strategy, and the notable exception compared to \cite{Fernandez1997}, is to pursue the expansion and to write the semimartingale decomposition of these higher order terms, until the remaining errors in the expansion of $\mu^n$ become lower than $1/\sqrt{n}$.  Although written in a separate way for clarity of exposition, the treatment of both global and local fluctuations (note that the local fluctuations are considered under a stricter set of hypotheses concerning the initial condition) follow the same main lines. The main argument to close this expansion (and the main contribution of the paper) concerns the use of generalised \emph{Grothendieck inequalities}, which enable us to decorrelate the dynamics from the proximity estimates between the graph $( \xi_{ij}^{ (n)})$ and the complete graph. More details on the use of Grothendieck inequalities are given in Section~\ref{sec:grothendieck}. 

\subsection{Notation}
\label{ss:notation}
The space of probability measures on a metric space $X$ is denoted by $\cP(X)$. As usual, we denote by $\cC([0, T], X)$ the space of continuous functions from $[0, T]$ into $X$ endowed with the supremum norm, this last one being denoted by $\norm{\cdot}_\infty$. For two probability measures $\mu$ and $\nu$ on a given metric space $X$, we denote by $d_{ BL}(\mu, \nu)$ their bounded Lipschitz distance, i.e.,
\begin{equation}
	\label{eq:dBL}
	d_{ BL}(\mu, \nu) = \sup \left\{ \int_X f \dd \mu - \int_X f \dd \nu\;,\; f\in BL(X) \right\},
\end{equation}
where $BL(X) :=\left\lbrace f: X \to \R: \; \left\Vert f \right\Vert_{ \infty}\leq 1,\ \left\Vert f \right\Vert_{ Lip}\leq 1\right\rbrace $ and $\norm{f}_{Lip}$ is the Liptschitz constant of $f$. In order to study the limit of both \eqref{def:fluctuation_process} and \eqref{def:zetas}, we will need to introduce several other weighted empirical processes, that all depend on the graph sequence $ \left(\xi_{ij}^{ (n)}\right)$. The first one is a weighted empirical measure $\hat{\nu}^n_t$ on $\bbT^2$, that is associated to the fluctuations arising from the graph sequence, i.e., 
\begin{equation}
\label{def:centered_graph_emp}
\hat{\nu}^n_t = \frac 1{n^2} \sum_{i, j=1}^n \hat{\xi}^{(n)}_{ij}  \delta_{\theta^{i,n}_t} \otimes \delta_{\theta^{j,n}_t}, \qquad 0\leq t \leq T,
\end{equation}
as well as its rescaled counterpart
\begin{equation}
\hat{\eta}^n := \sqrt{n} \, \hat{\nu}^n.
\label{eq:hat_etan}
\end{equation}
In \eqref{def:centered_graph_emp} and \eqref{eq:hat_etan}, we used the notation
\begin{equation}
\label{eq:hatxi}
\hat{\xi}^{(n)}_{ij}:= \frac{ \xi^{(n)}_{ij}}{ p_{ n}} - 1,\ i,j = 1, \dots, n.
\end{equation}
A measure that will arise naturally in the study of local fluctuations is the following:
\begin{equation}
\label{def:mu_n12}
\mu_{ t}^{ n, 1,2}:= \frac{ 1}{ np_{ n}^2} \sum_{ i=1}^{ n} \xi_{1i}^{ n} \xi_{2 i}^{ n} \delta_{ \theta_{ t}^{ i, n}},
\end{equation}
whose role is to account for the presence/absence of the pattern $1 \leftarrow i \to 2$ in the graph, capturing a notion of connectivity between vertices $1$ and $2$.  Let $\< \cdot, \cdot >$ denote the duality bracket between function spaces and their dual spaces, observe that the mass $ \left\langle \mu_{ t}^{ n,1,2}\, ,\, 1\right\rangle$ tends also almost surely to $1$ as $n\to\infty$.

For studying the sequence $(\eta^n)_{n\geq 1}$ defined in \eqref{def:fluctuation_process}, we need a suitable space of distributions. In fact, quantities such as $\eta^n$ or $\hat{\nu}^n$ are not probability measures, having total measure close to 0, and require to be studied in a larger space. As already done in the literature \cite{Fernandez1997,hitsudaMitoma86,Bertini:2013aa, bechtold_coppini_2021}, we choose to work in a class of Hilbert spaces that include (as a continuous embedding) the probability measures. The canonical choice is given by the usual Sobolev Hilbert spaces $H^{ -r}(\mathbb{ T}^{ d}):= W^{-r, 2}(\mathbb{ T}^{ d})$ (with $r>0$), dual of $H^{ r}(\mathbb{ T}^{ d})$ space of test functions with derivatives up to order $r$ with finite moments of order $2$ \cite{adams_fournier_2003}. We define $\Vert h\Vert_{H^{-r}(\bbT)}$ (resp. $\Vert h\Vert_{H^{-r}(\bbT^{ 2})}$) as the norm on $ H^{ -r} ( \mathbb{ T})$ (resp. $ H^{ -r}(\mathbb{ T}^{ 2})$). To keep the notation concise, we will often write $\Vert h\Vert_{-r}$ for both these norms, whenever the context leaves no ambiguity on the fact that this notation concerns functional acting on test functions on $ \mathbb{ T}$ or $ \mathbb{ T}^{ 2}$. Sobolev inequalities, e.g.,  \cite{adams_fournier_2003}, state that
\begin{equation}
\label{eq:sobolev_ineq}
W^{r, 2} (\bbT^d) =H^r(\bbT^d) \subset \cC^{0, \alpha}(\bbT^d)
\end{equation}
for $\alpha\in(0,1]$ and $r = \frac d2 + \alpha$, where $\cC^{0, \alpha}(\bbT^d)$ is the space of continuous functions with $\alpha$-H\"{o}lder regularity. By duality, $\cP(\bbT^d)$ is continuously embedded into $H^{-r}(\bbT^d)$ for any $r > \frac d2$. This implies that any probability measure on $\bbT$ belongs to $H^{-r}(\bbT)$ for any $r> \frac 12$. For a given $t \in [0,T]$, the distribution $\hat{\nu}^n_t$ is an element of the Hilbert space $H^{-r}(\bbT^2)$ for $r > 1$. Recall also that one has the following Hilbert-Schmidt embeddings (see \cite{Fernandez1997} and \cite[\S 6]{adams_fournier_2003})
\begin{equation*}
	H^{ -j}(\mathbb{ T}^{ d}) \subset H^{ -(m+j)}(\mathbb{ T}^{ d}), \quad m> \frac{ d}{ 2} \text{ and } j\geq0.
\end{equation*}

In the Hilbert space $H^r(\mathbb{ T}^{ d})$, one can define the semigroup operator $S=(S_t)_{t\geq0}$ associated to the Laplacian operator, this last one being denoted by $\Delta$ (or $\partial^2_\theta$ in the one-dimensional case). It is well-known, e.g., \cite{henry}, that $S_\cdot$ is an analytic semigroup. For a given operator $\cL$ on some Hilbert space, we denote by $\cL^*$ its dual operator in the corresponding dual Hilbert space.

\section{Hypotheses and main results}
\subsection{General hypotheses}
\label{sec:hypotheses}
\begin{assumption}[Regularity of $ \Gamma$]
\label{ass:Gamma}
We assume that $ \Gamma$ is infinitely differentiable on $ \mathbb{ T}^{ 2}$ (and hence bounded with bounded derivatives). A careful reading of the proofs below shows that $ \Gamma$ being $ \mathcal{ C}^{ k}$ for a sufficiently large $k$ would be actually sufficient.
\end{assumption}
\begin{assumption}[Initial condition]
\label{ass:initial}
We suppose that the initial condition, that is the random variables $ \left(\theta_{ 0}^{ 1, n}, \ldots, \theta_{ 0}^{ n,n}\right)$, are chosen independently from the Brownian motions $(B^{ 1}, \ldots, B^{ n})$ (but not necessarily i.i.d. and they may depend on the graph), such that their empirical measure $ \mu_{ 0}^{ n}$ converges weakly to some $ \mu_{ 0}$ in the following way:
\begin{equation}
\label{eq:conv_mu0}
d_{ BL} \left(\mu_{ 0}^{ n}, \mu_{ 0}\right) \xrightarrow[ n\to\infty]{}0, \ \mathbb{ P}_{ 0}-a.s.
\end{equation}
\begin{rem}
\label{rem:conv_empmeas}
Note that since $ d_{ BL} \left(\mu_{ 0}^{ n}, \mu_{ 0}\right)\leq 2$, the convergence \eqref{eq:conv_mu0} actually implies 
\begin{equation}
\label{eq:conv_mu0_E}
\mathbb{ E}_{ 0}\left[ d_{ BL} \left(\mu_{ 0}^{ n}, \mu_{ 0}\right)^{ q}\right] \xrightarrow[ n\to\infty]{}0
\end{equation} for any $q\geq1$. It is possible to assume only \eqref{eq:conv_mu0_E} for some $q\geq1$ and in such a case \eqref{eq:conv_empmeas} below remains true up to a supplementary integration w.r.t. $ \mathbb{ E}_{ 0}$ and the results of the paper remain valid.
\end{rem}
\end{assumption}
\subsection{Laws of large numbers for global and local empirical measures}
Before presenting our main fluctuation results, let us state the following result (which may have an interest of its own) on the convergence of the empirical measures $\mu^n$, $\mu^{n,l}$ and $\mu^{n,1,2}$ defined respectively in \eqref{eq:emp_meas}, \eqref{def:mu_nl} and \eqref{def:mu_n12}. 
\begin{theorem}
\label{th:conv_empmeas}
Suppose here that Assumptions~\ref{ass:Gamma} and~\ref{ass:initial} are satisfied.
\begin{enumerate}
\item Convergence of the global empirical measure: under the dilution condition
\begin{equation}
\label{eq:dilution_cond_opt}
np_{ n}\to\infty
\end{equation}
the global empirical measure $ \mu^{ n}$ given in \eqref{eq:emp_meas} verifies, for all $q\geq 1$, 
\begin{equation}
\label{eq:conv_empmeas}
\mathbf{ E} \left[ \sup_{ s\leq T} d_{ BL} \left(\mu_{s}^{ n}, \mu_{ s}\right)^{ q}\right] \xrightarrow[ n\to\infty]{}0,\ \mathbb{ P}-a.s.
\end{equation}
\item Convergence of local empirical measures: if one supposes further that the initial condition is independent of the graph: 
\begin{enumerate}
\item if $np_{ n}^{ 3}\to \infty$ as $n\to\infty$, then for any fixed $l\geq 1$, the local empirical measure $ \mu^{ n,l}$ defined in \eqref{def:mu_nl} verifies, for all $q\geq 1$,
\begin{equation}
\label{eq:conv_munl}
\mathbf{ E} \left[ \sup_{ s\leq T} d_{ BL } \left( \mu_{ s}^{ n, l}, \mu_{ s}\right)^{ q}\right] \xrightarrow[ n\to\infty]{} 0, \ \mathbb{ P}-a.s.
\end{equation}
\item if $n p_n^5 \to \infty$ as $n\to\infty$, then the local empirical measure $ \mu^{ n,1,2}$ defined in \eqref{def:mu_n12} verifies for all $q\geq 1$,
\begin{equation}
\label{eq:conv_mun12}
\bE \left[ \sup_{s\leq T} d_{BL} (\mu_{ s}^{ n, 1, 2}, \mu_s)^q \right] \xrightarrow[ n\to\infty]{} 0, \quad \bbP-\text{a.s.}
\end{equation}
\end{enumerate}
\end{enumerate}
\end{theorem}
The proof of Theorem~\ref{th:conv_empmeas} can be found in Appendix~\ref{sec:conv_empmeas}. In comparison with the existing literature, the convergence result \eqref{eq:conv_empmeas} generalises the previous results in two ways: we obtain the optimal dilution condition \eqref{eq:dilution_cond_opt} under the quenched set-up and more importantly, we allow for initial condition that possibly depend on the graph (not to mention that they need not be necessarily i.i.d.). We refer to Section~\ref{ss:literature} for a more detailed discussion on this matter. It is however likely that the dilution conditions required for the convergence of the local empirical measures may not be optimal.

\subsection{Global fluctuations}
\label{sec:global_fluct}

We now proceed with the main results concerning fluctuations. We address two issues: first, \emph{global fluctuations} (Section~\ref{sec:global_fluct}) that is the behavior as $n\to\infty$ of the global fluctuation process $ \eta^{ n}$ given in \eqref{def:fluctuation_process} as $n\to\infty$; second, \emph{local fluctuations} (Section~\ref{sec:local_fluct}) that is the joint convergence of the local fluctuation processes $ \left( \zeta^{ n, 1}, \zeta^{ 2, n}\right)$ given in \eqref{def:zetas}. In the rest of the paper, the following indices are fixed:
\begin{equation}
\label{eq:indices}
r_{ 0}> 3, \ r_{ 1}:= r_{ 0}+2,\  r_{ 2}:= r_{ 1}+2.
\end{equation}
Recall the definitions of $ \eta^{ n}$ in \eqref{def:fluctuation_process} and $ \hat{ \eta}^{ n}$ in \eqref{eq:hat_etan}. We first state our main hypotheses concerning the initial condition. We suppose in the following that either Assumption~\ref{ass:jointconv_0W_qu} or Assumption~\ref{ass:jointconv_0W_ann} is true.
\begin{assumption}[Quenched initial fluctuations]
\label{ass:jointconv_0W_qu}
We suppose that there exists $ \alpha\in(0, 1)$ such that the following estimates are true
\begin{align}
\sup_{ n} \mathbb{ E}_{ 0} \left( \left\Vert \eta^{ n}_{ 0} \right\Vert_{-r_{ 0}}^{ 1+ \alpha}\right)&< +\infty,\  \mathbb{ P}_{ g}\text{-a.s.}\label{eq:control_eta0_qu}\\
\sup_{ n} \mathbb{ E}_{ 0} \left( \left\Vert \hat{ \eta}^{ n}_{ 0} \right\Vert_{ -r_{ 0}}^{ 1+ \alpha}\right)&< +\infty,\  \mathbb{ P}_{ g}\text{-a.s.}\label{eq:control_hateta0_qu}
\end{align} 
Note that under \eqref{eq:control_eta0_qu} and \eqref{eq:control_hateta0_qu}, $\mathbb{ P}_{ g}$-a.s., $ \left(\eta_{ 0}^{ n}\right)$ and $ \left( \hat{ \eta}_{ 0}^{ n}\right)$ are tight in $ H^{ -r_{ 1}} \left(\mathbb{ T}\right)$ and $H^{ - r_{ 1}} \left(\mathbb{ T}^{ 2}\right)$ respectively. In addition, we require that we have, $\mathbb{ P}_{ g}$-a.s., the joint convergence in law (w.r.t. $ \mathbb{ P}_{ 0}$) of $ \left( \eta_{ 0}^{ n}, \hat{ \eta}_{ 0}^{ n}\right)$ in $ H^{ - r_{ 1}}\left(\mathbb{ T}\right) \otimes H^{ -r_{ 1}}(\mathbb{ T}^{ 2})$ towards some $ \left( \eta_{ 0}, \hat{ \eta}_{ 0}\right)\in  H^{ - r_{ 1}}\left(\mathbb{ T}\right) \otimes H^{ - r_{ 1}}(\mathbb{ T}^{ 2})$ as $n\to\infty$.
\end{assumption}
\begin{assumption}[Annealed initial fluctuations]
\label{ass:jointconv_0W_ann}
We suppose the same hypotheses as for Assumption~\ref{ass:jointconv_0W_qu}, with \eqref{eq:control_eta0_qu} and \eqref{eq:control_hateta0_qu} replaced by
\begin{align}
\sup_{ n}\mathbb{ E}_{ g} \mathbb{ E}_{ 0} \left( \left\Vert \eta^{ n}_{ 0} \right\Vert_{-r_{ 0}}^{ 1+ \alpha}\right)&< +\infty,\label{eq:control_eta0_ann}\\
\sup_{ n} \mathbb{ E}_{ g}\mathbb{ E}_{ 0} \left( \left\Vert \hat{ \eta}^{ n}_{ 0} \right\Vert_{ -r_{ 0}}^{ 1+ \alpha}\right)&< +\infty,\label{eq:control_hateta0_ann}
\end{align} 
and the joint convergence of $ \left( \eta_{ 0}^{ n}, \hat{ \eta}_{ 0}^{ n}\right)$ in law w.r.t. the joint law $ \mathbb{ P}$ of the initial condition and graph.
\end{assumption}

Let us now state the main result of this paper on global fluctuations, concerning the weak limit of the fluctuation processes $\eta^n$ and $\hat \eta^n$. Define first the following linear differential operators: for all test functions $f$ and $g$, $s\in [0, T]$ and $\nu \in H^{-r}$ with $r>1/2$,
\begin{align}
\Theta f(\theta_{ 1}, \theta_{ 2})&:= \Gamma \left( \theta_{ 1}, \theta_{ 2}\right) \partial_{ \theta}f( \theta_{ 1}),\label{eq:Theta}\\
\mathcal{ L}_{ \nu}^{(1)}f(\theta)&:= \frac{ 1}{ 2} \partial_{ \theta}^{ 2} f(\theta) + \left\langle \nu({\rm d} \theta^{ \prime})\, ,\, \Theta f(\theta, \theta') \right\rangle + \left\langle \nu({\rm d}\theta^{ \prime})\, ,\, \Theta f(\theta', \theta) \right\rangle,\label{eq:L1}\\
\mathcal{ L}_{ \nu}^{(2)}g(\theta_{ 1}, \theta_{ 2})&:= \frac{ 1}{ 2} \Delta g(\theta_1, \theta_2) + \left(\left\langle \nu ({\rm d}\theta^{ \prime})\, ,\, \Gamma (\theta_{ 1}, \theta^{ \prime})\right\rangle,  \left\langle \nu ({\rm d}\theta^{ \prime})\, ,\, \Gamma (\theta_{ 2}, \theta^{ \prime})\right\rangle \right) \cdot \nabla g(\theta_{ 1}, \theta_{ 2}).\label{eq:L2}
\end{align}
\begin{theorem}[General global fluctuations]
\label{th:limit eta and hat eta}
Recall the definition of $(r_{ 0}, r_{ 1}, r_{ 2})$ in \eqref{eq:indices} and that $\mu$ solves the Fokker-Planck equation \eqref{eq:limit PDE}. Suppose that Assumptions~\ref{ass:Gamma},~\ref{ass:initial} hold as well as either of Assumption~\ref{ass:jointconv_0W_qu} or~\ref{ass:jointconv_0W_ann}. Suppose finally that 
\begin{equation}
\label{eq:dilution_cond}
\lim_n n p_n^4=\infty.
\end{equation} 
Then $(\eta^n,\hat \eta^n)$ converges in law in $\mathcal{ C}([0,T],H^{-r_{ 1}}(\bbT) \otimes H^{ -r_{ 1}} \left(\mathbb{ T}^{ 2}\right))$ to $(\eta,\hat \eta)$, unique solution in $ \mathcal{ C} \left([0, T], H^{ -r_{ 2}}\left( \mathbb{ T}\right) \oplus H^{ -r_{ 2}}( \mathbb{ T}^{ 2})\right)$ to
\begin{equation}
\label{eq:limit_etas}
\begin{cases}
\eta_{ t}= \eta_{ 0} + \int_{ 0}^{t} \mathcal{ L}_{ \mu_s}^{ (1), \ast}\eta_{ s} {\rm d}s + \int_{ 0}^{t} \Theta^{ \ast} \hat{ \eta}_{ s} {\rm d}s + W_{ t},\\
\hat{ \eta}_{ t}=  \hat{ \eta}_{ 0} + \int_{ 0}^{t}\mathcal{ L}_{ \mu_s}^{(2), \ast}\hat{ \eta}_{ s} {\rm d}s,
\end{cases}
\end{equation}
where for any $r>2$, $(W_t)_{ t\in [0, T]}$ is a Gaussian process in $ \mathcal{ C} \left([0, T], H^{ -r}\right)$, independent of $( \eta_{ 0}, \hat{ \eta}_{ 0})$, with covariance
\begin{equation}
\label{eq:cov_Weta}
\mathbf{ E} \left[ W_{ s}(f_{ 1}) W_{ t}(f_{ 2})\right]= \int_{ 0}^{s} \left\langle \mu_{ u}\, ,\, \partial_{ \theta}f_{ 1} \partial_{ \theta}f_{ 2}\right\rangle {\rm d}u,\ f_{ 1}, f_{ 2} \in H^{ r}, \ 0\leq s \leq t\leq T.
\end{equation} 
In case Assumption~\ref{ass:jointconv_0W_qu} holds, the above convergence is almost sure w.r.t. the randomness of the graph (quenched convergence) whereas in case of Assumption~\ref{ass:jointconv_0W_ann}, the convergence is understood under the annealed law $ \mathbb{ P}\otimes \mathbf{ P}$.
\end{theorem}
A particular case of Theorem~\ref{th:limit eta and hat eta} concerns the case where the initial condition for the second order fluctuation process $ \hat{ \eta}_{ 0}^{ n}$ goes to $0$ as $n\to\infty$:
\begin{theorem}[Universal mean-field fluctuations]
\label{th: limit eta independent}
Suppose Assumptions~\ref{ass:Gamma},~\ref{ass:initial} and either Assumption~\ref{ass:jointconv_0W_qu} or~\ref{ass:jointconv_0W_ann} are true. Suppose in addition that the limit of $ \hat{ \eta}_{ 0}^{ n}$ given in Assumptions~\ref{ass:jointconv_0W_qu} or~\ref{ass:jointconv_0W_ann} is $ \hat{ \eta}_{ 0}\equiv 0$. Then the process $(\eta^n)$ converges in law as $n\to\infty$ in $\mathcal{ C}([0,T],H^{-r_{ 1}}(\bbT))$ to $\eta$, unique solution in $ \mathcal{ C} \left([0, T], H^{ -r_{ 2}} \left(\mathbb{ T}\right)\right)$ to 
\begin{equation}
\label{eq:limit_eta_only}
\eta_{ t}= \eta_{ 0} + \int_{ 0}^{t} \mathcal{ L}_{ \mu_s}^{ (1), \ast}\eta_{ s} {\rm d}s + W_{ t},
\end{equation}
with $ \eta_{ 0}$ independent of $W$. In case of Assumption~\ref{ass:jointconv_0W_qu}, the above convergence is almost-sure w.r.t. the randomness of the graph (quenched convergence) and in case of Assumption~\ref{ass:jointconv_0W_ann}, the convergence holds w.r.t. the annealed law $ \mathbb{ P} \otimes \mathbf{ P}$.
\end{theorem}
\begin{proof}[Proof of Theorem~\ref{th: limit eta independent}]
It suffices to note that the limiting dynamics of $ \hat{ \eta}$ in \eqref{eq:limit_etas} is deterministic and linear, so that if initially $ \hat{ \eta}_{ 0}\equiv 0$ one obtains by uniqueness that $ \hat{ \eta}_{ t} \equiv 0$ uniformly in $t\in[0, T]$. Hence, Theorem~\ref{th: limit eta independent} follows immediately from Theorem~\ref{th:limit eta and hat eta}.
\end{proof}
Observe that \eqref{eq:limit_eta_only} is nothing else than the limiting SPDE of the fluctuation process in the pure mean-field case $ p_{ n}\equiv 1$ that has been obtained in \cite{Fernandez1997}, under i.i.d. initial condition. In this case, Theorem~\ref{th: limit eta independent} is of course compatible with the result of \cite{Fernandez1997} as, when $ p_{ n}\equiv 1$, $ \hat{\eta}_{ 0}^{ n}$ is equally $0$ for all $n$ and $ \eta_{ 0}^{ n}$ converges to a Gaussian process so that Assumption~\ref{ass:jointconv_0W_qu} is trivially true. One can see Theorem~\ref{th: limit eta independent} as a universality result, valid beyond the mean-field case, under the dilution condition \eqref{eq:dilution_cond}: the system \eqref{eq:wips} conserves the same fluctuations as in the mean-field case, as long as one can verify Assumption~\ref{ass:jointconv_0W_qu} or~\ref{ass:jointconv_0W_ann} and $ \hat{ \eta}_{ 0}\equiv 0$.  It is likely that the dilution condition \eqref{eq:dilution_cond} may not to be optimal: the critical point on this matter is the concentration estimates on quantities $S_{ n}^{ \mathcal{ T}}$ given in Definition~\ref{def:Sn}. Any improvement in the rates of convergence found in Proposition~\ref{prop:concentration_SnT} would lead to corresponding improvements in \eqref{eq:dilution_cond}. The second direction in which Theorem~\ref{th: limit eta independent} generalises the Central Limit Theorem of \cite{Fernandez1997} is the following: a crucial observation is that the fluctuation result in \cite{Fernandez1997} was proven in the case where the initial datum $ \left(\theta_{ 0}^{ 1, n}, \ldots, \theta_{ 0}^{ n,n}\right)$ consists of i.i.d. random variables. This hypothesis, being perfectly reasonable in the pure mean-field context as a natural means to preserve exchangeability between particles, is not really relevant in the context of \eqref{eq:wips}, as exchangeability is lost, due to the presence of the graph. Anticipating on Section~\ref{sec:examples} (where sufficient conditions for the result are given), we indeed show that these universal fluctuations go well beyond the i.i.d. case as they remain valid as long as the initial condition is chosen independently on the graph.

In an opposite way, we also describe in Proposition~\ref{prop:fluct_Kur} an example of initial condition, depending on the graph structure, such that $\hat \eta^n_0$ has a non zero limit, and thus for which the limit fluctuations are completely described by \eqref{eq:limit_etas} and no longer by the mean field fluctuations \eqref{eq:limit_eta_only}.

\subsection{Local fluctuations}
\label{sec:local_fluct}
We now give our result concerning the local fluctuations (recall the definitions of the local fluctuation processes $ \zeta^{ i}_{ n}$, $i=1,2$ in \eqref{def:zetas}).  As global fluctuations compete with local fluctuations, the main result concerns the convergence of the joint fluctuation process $\left(\zeta^{ n, 1}, \zeta^{ n, 2},\eta^n\right)$. We place ourselves in the case of i.i.d. initial condition, independent on the graph. Anticipating on Proposition~\ref{prop:limit_hateta0_indep} and Example~\ref{ex:cas_iid}, we see that Theorem~\ref{th: limit eta independent} is true: the global fluctuations of $ \eta^{ n}$ are completely described in terms of \eqref{eq:limit_eta_only}.
\begin{theorem}
\label{th:local_fluct}
Suppose Assumption~\ref{ass:Gamma} and that $(\theta_{ 0}^{ 1, n},\ldots, \theta_{ 0}^{ n,n})$ are i.i.d. random variables with law $ \mu_{ 0}$, independent from the graph. Suppose that $\liminf_n n p_n^5 = \infty$ and denote by $p:= \lim_{ n\to\infty} p_{ n}\in [0, 1]$. Then, $\mathbb{P}_g$ almost surely, the joint fluctuation process $\left(\zeta^{ n, 1}, \zeta^{ n, 2},\eta^n\right)$ converges as $n\to\infty$ in $ \mathcal{ C} \left([0, T], \left(H^{ - r_{ 1}} \left(\mathbb{ T}\right)\right)^3\right)$ to $\left(\zeta^{1}, \zeta^{2}, \eta\right)$ solution in  $ \mathcal{ C} \left([0, T], \left(H^{ - r_{ 2}} \left(\mathbb{ T}\right)\right)^3\right)$ to the system
\begin{equation}
\label{eq:SDPEs_zetas_eta}
\begin{cases}
\begin{split}
\zeta_{ t}^{l} &= \zeta_{ 0}^{l} + \int_{ 0}^{t} \mathcal{ U}_{ s}^{\ast} \zeta_{ s}^{l}{\rm d}s+ \sqrt{ p}\int_{ 0}^{t} \mathcal{ V}_{ s}^{ \ast} \eta_{ s} {\rm d}s+ W_{ t}^{ l},\ l=1,2,\\
\eta_{ t}&= \eta_{ 0} + \int_{ 0}^{t} \mathcal{ L}_{ \mu_s}^{(1),\ast} \eta_{ s} {\rm d}s +  W_{ t}.
\end{split}
\end{cases}
\end{equation}
where 
\begin{align}
\mathcal{ U}_{ s}f \left(\theta\right)&=\frac{ 1}{ 2} \partial_{ \theta}^{ 2} f(\theta) + \partial_{ \theta}f \left(\theta\right)  \left\langle \mu_{ s}({\rm d}\theta^{ \prime})\, ,\, \Gamma \left(\theta, \theta^{ \prime}\right)\right\rangle, \label{eq:Us}\\
\mathcal{ V}_{ s}f(\theta)&=\left\langle \mu_{ s}({\rm d}\theta^{ \prime})\, ,\, \partial_{ \theta}f \left(\theta^{ \prime}\right) \Gamma \left(\theta^{ \prime}, \theta\right)\right\rangle, \label{eq:Vs}
\end{align}
for $(\zeta_{ 0}^{ 1}, \zeta_{ 0}^{ 2}, \eta_{ 0})$ a Gaussian process with explicit covariance given in \eqref{eq:cov_zetas_0} and $ \left(W_{ t}^{ 1}, W_{ t}^{ 2}, W_t\right)$ Gaussian process with explicit covariance given in \eqref{eq:cov_noise_12}, the initial condition $(\zeta_{ 0}^{ 1}, \zeta_{ 0}^{ 2}, \eta_{ 0})$ and the noise $(W^{ 1}, W^{ 2}, W)$ being independent.
\end{theorem}
A closer look at the structure of covariance of both initial condition in \eqref{eq:cov_zetas_0} and noise in \eqref{eq:cov_noise_12} shows that in the diluted case $p=\lim_{ n\to\infty} p_{ n}=0$, the process $ \left(\zeta^{ 1}, \zeta^{ 2}, \eta\right)$ are mutually independent and each $ \zeta^{ l}$ ($l=1,2$) satisfy
\begin{equation*}
\zeta_{ t}^{l} = \zeta_{ 0}^{l} + \int_{ 0}^{t} \mathcal{ U}_{ s}^{\ast} \zeta_{ s}^{l}{\rm d}s+  W_{ t}^{ l},\ l=1,2.
\end{equation*}
In the dense case $p>0$, $ \zeta^{ 1}$ and $ \zeta^{ 2}$ are correlated in several ways: with a nontrivial correlation of their noise and initial condition, and through the coupling of the global fluctuation process $ \eta$. Theorem~\ref{th:local_fluct} is proven in Section~\ref{sec:proof_local_fluct}.

\subsection{Examples}
\label{sec:examples}
We give in this section examples of initial condition verifying Assumptions~\ref{ass:jointconv_0W_qu} or~\ref{ass:jointconv_0W_ann} as well as sufficient conditions for the universality hypothesis $ \hat{ \eta}_{ 0}\equiv 0$ of Theorem~\ref{th: limit eta independent}.
\subsubsection{Universal mean-field fluctuations}
The main important point of this paragraph is to note that the universality condition $ \hat{ \eta}_{ 0}\equiv 0$ of Theorem~\ref{th: limit eta independent} is true as long as one chooses the initial condition to be independent of the graph (but not necessarily i.i.d.!).
\begin{assumption}[Initial datum independent from the graph]
	\label{ass:theta0_indep_xi}
	Suppose that $ \mathbb{ P}= \mathbb{ P}_{ 0}\otimes \mathbb{ P}_{ g}$, i.e., the initial condition $(\theta^{i,n}_0)_{n\geq 0}$ is independent from the graph $( \xi^{n})_{n\geq 0}$. 
\end{assumption}

\begin{proposition}
	\label{prop:limit_hateta0_indep}
	Suppose Assumption~\ref{ass:theta0_indep_xi}. If $\liminf_n p_n^3 n^{1-\gep}=\infty$ for some $\gep\in (0,1)$, then, for all $r> 1$,
	\begin{equation}
		\label{eq:limit_hateta0_indep}
		\lim_{n\rightarrow\infty} \mathbb{ E}_{ 0} \left[\left\Vert \hat{ \eta}_{ 0}^{ n} \right\Vert_{-r}^{ 2}\right] =0, \quad \bbP_g-\text{a.s.}.
	\end{equation}
	In particular, \eqref{eq:control_hateta0_qu} is true for any $ \alpha\in (0, 1)$ and $ \hat{ \eta}_{ 0}\equiv 0$.
\end{proposition}
Proof of Proposition~\ref{prop:limit_hateta0_indep} is given in Section~\ref{sec:proofs_examples}.
\subsubsection{Quenched mean-field fluctuations}
We give now illustrating examples of initial condition that is independent of the graph, hence particular cases of Assumption~\ref{ass:theta0_indep_xi}, for which the initial fluctuation process $ \eta_{ 0}^{ n}$ verifies Assumption~\ref{ass:jointconv_0W_qu}.
\begin{example}[The case of i.i.d. initial condition]
	\label{ex:cas_iid}
	Suppose that $(\theta_{ 0}^{ i, n})$ are i.i.d. with law $ \mu_{ 0}$ (and independent on the graph and on the Brownian motions $(B^{ i})$). Then Assumptions~\ref{ass:jointconv_0W_qu} is valid and the limit $ \eta_{ 0}$ is the Gaussian process with covariance 
	\begin{equation}
		\label{eq:cov_eta0_dense}
		C_{ \eta_{ 0}} (f_{ 1}, f_{ 2})= \int_{ \mathbb{ T}} \left( f_{ 1} - \int_{ \mathbb{ T}} f_{ 1} {\rm d}\mu_{ 0}\right)\left( f_{ 2} - \int_{ \mathbb{ T}} f_{ 2} {\rm d}\mu_{ 0}\right) {\rm d}\mu_{ 0}.
	\end{equation}
\end{example}
The following example shows that Assumption~\ref{ass:jointconv_0W_qu} is sufficiently weak to possibly include non necessarily i.i.d. initial condition. We do not try to give any sharp condition here, we refer to the references mentioned in Example~\ref{ex:beta_mixing} for details.
\begin{example}[The case of $ \alpha$-mixing sequence]
\label{ex:beta_mixing}
	On $ \left(\Omega, \mathcal{ A}, \mathbb{ P}_{ 0}\right)$, define $T: \Omega \mapsto \Omega$ a bijective bimeasurable transformation preserving $ \mathbb{ P}_{ 0}$. Let $ \mathcal{ M}_{ 0}$ a $ \sigma$-algebra of $ \mathcal{ A}$ such that $ \mathcal{ M}_{ 0}\subset T^{ -1} \left(\mathcal{ M}_{ 0}\right)$ and $ \theta_{ 0}^{ 0}$ be a $ \mathcal{ M}_{ 0}$-measurable random variable on $ \mathbb{ T}$. Define finally $ \theta_{0}^{ i}:= \theta_{ 0}^{ 0} \circ T^{ i}$ for $i\geq1$. Then applying \cite[(4.1) and Th.~2]{Dedecker2007}, supposing that 
	\begin{equation}
		\label{eq:sum_beta}
		\sum_{ k>0} \alpha \left( \mathcal{ M}_{ 0}, \sigma (X_{ k})\right)<\infty,
	\end{equation}
	(where $ \alpha( \mathcal{ A}, \mathcal{ B})$ is the Rosenblatt $ \alpha$-mixing coefficient between $ \mathcal{ A}$ and $ \mathcal{ B}$), we have that $ \eta_{ 0}^{ n}$ converges as $n\to\infty$ in $H^{ -1} \left(\mathbb{ T}\right)$ to a Gaussian process with explicit covariance. Moreover, under the same condition, applying \cite[Th.~1]{Dedecker2003} for $ \varphi(\eta):= \left\Vert \eta \right\Vert_{ -1}^{ 2}$, we obtain the uniform bound \eqref{eq:control_eta0_qu} for $ \alpha=1$, so that Assumption~\ref{ass:jointconv_0W_qu} is satisfied. In the context of Markov chains, condition \eqref{eq:sum_beta} is true as soon as the chain is \emph{ergodic of degree $2$} \cite{nummelin_1984} and includes geometrically ergodic Markov chains \cite{Chen_1999}.
\end{example}
\subsubsection{An example of non-universal fluctuations}
We construct in Section~\ref{ss:initial_cond_on_graph} an example of initial condition (depending on the graph sequence) such that the limit fluctuations are non-universal:
\begin{proposition}
	\label{prop:fluct_Kur}
	Take $ \Gamma(\theta, \theta^{ \prime})= -K\sin \left(\theta- \theta^{ \prime}\right)$ with $K>0$. For any graph sequence $ \left(\xi_{ij}^{ (n)}\right)$, there exists a choice of initial condition $(\theta_{ 0}^{ 1,n},\ldots, \theta_{ 0}^{ n,n})$ such that $(\eta_{ 0}^{ n}, \hat{ \eta}_{ 0}^{ n})$ satisfies Assumption~\ref{ass:jointconv_0W_ann} and converges in law (w.r.t. the annealed law $ \mathbb{ P}\otimes \mathbf{ P}$) to $(\eta_{ 0}, \hat{ \eta}_{ 0})$ with $ \eta_{ 0} = Z_{ 1} \delta_{ 0} + Z_{ 2} \delta_{ \frac{ \pi}{ 2}}$ (where $(Z_{ 1}, Z_{ 2}) \sim \mathcal{ N}(0, C)$ with the covariance matrix $C$ defined as $C= \begin{pmatrix}
		\frac{ 1}{ 4}& - \frac{ 1}{ 4}\\ - \frac{ 1}{ 4} & \frac{ 1}{ 4}
	\end{pmatrix}$ and $ \hat{ \eta}_{ 0}= \frac{ 1}{ 6 \sqrt{ \pi}} \left(- \delta_{ (0, 0)} + 2 \delta_{ \left(\frac{ \pi}{ 2}, 0\right)} - \delta_{ \left(\frac{ \pi}{ 2}, \frac{ \pi}{ 2}\right)}\right)$. In particular, $ \Gamma\ast \hat{ \eta}_{ 0}^{ n}$ converges to $ - \frac{ 1}{ 3 \sqrt{ 2 \pi}} \delta_{ \frac{ \pi}{ 2}} \not \equiv 0$, so that the limiting process $ \eta_{ t}$ is governed by \eqref{eq:limit_etas} and \emph{not} by the universal mean-field SPDE \eqref{eq:limit_eta_only}.
\end{proposition}
Proof of Proposition~\ref{prop:fluct_Kur} is given in Section~\ref{ss:initial_cond_on_graph}.

\begin{figure}[ht]
\centering
\includegraphics{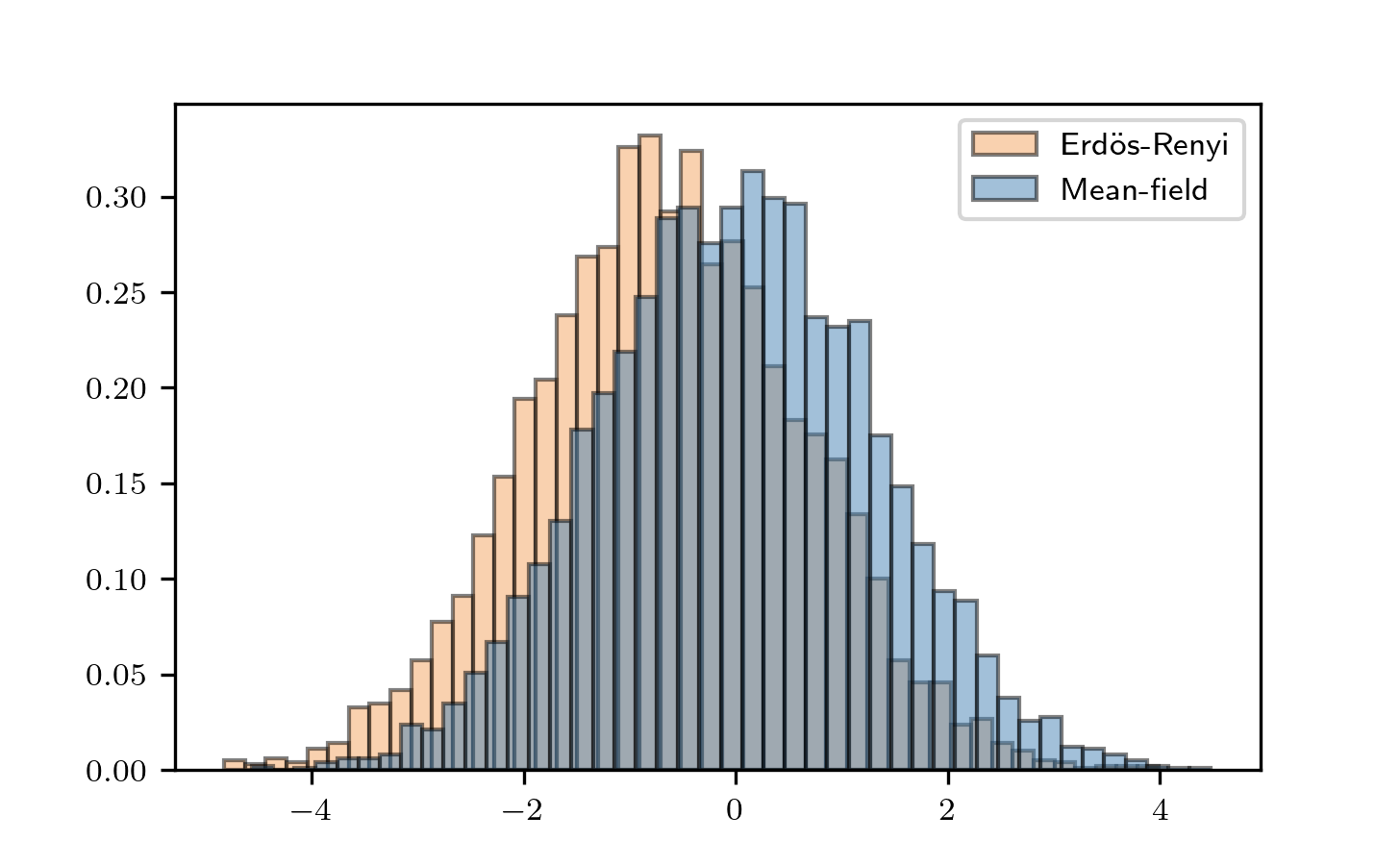}
\label{subfig:toy_model_apos}
\caption{Histograms representing $5000$ realizations of $\sqrt{n}(\psi^n_1-\psi_1)$, where $r^n_t e^{i\psi^n_t}=\langle \mu^n_t(\dd \theta),e^{ i \theta}\rangle$ and  $r_t e^{i\psi_t}=\langle \mu_t(\dd \theta),e^{ i \theta}\rangle$, for the choice of interaction kernel $\Gamma(\theta,\theta')=- K \sin(\theta-\theta')$ with $K=2$. For the blue histogram the interaction is of mean-field type with i.i.d. initial condition of distribution $\frac12 \gd_0+\frac12 \gd_{\frac{\pi}{2}}$, while for the brown one it is of symmetric Erd\H{o}s-Renyi type with $p=0.5$ and initial condition as described in Section~\ref{ss:initial_cond_on_graph}. We observe a dephasing at time $t=1$ at the level of fluctuations, induced by the graph-dependent initial condition.}
\end{figure}

\subsection{On possible generalisations to inhomogeneous graphs} Even though we have focused on Erd\H{o}s-Rényi random graphs, the same formalism should easily adapt to inhomogeneous situations, i.e., to graphons (see e.g. \cite{bayraktar2020graphon,bet2020weakly,Lucon2020}). Notably, the concentration results used in this paper (i.e. Bernstein inequalities or the concentration result for the operator norm of a random matrix \cite{tao2012} used in the proof of Proposition~\ref{prop:concentration_SnT}) only require the $ \xi_{i j}^{ (n)}$ to be independent, not necessarily identically distributed. For the sake of conciseness, we give one illustrating example generalising the present homogeneous case and leave further generalisations to the reader. This example is an elementary instance of the Stochastic Block model in case of only two communities: let $n$ an even number and divide the population into two clusters $C_{ 1}^{ n}= \left\lbrace1, \ldots, \frac{ n}{ 2}\right\rbrace$ and $C_{ 2}^{ n}= \left\lbrace \frac{ n}{ 2}+1,\ldots, n\right\rbrace$, and suppose that the $ \xi_{ij}^{ (n)}$ are independent with Bernoulli law with parameter $p_{ i,j}= p$ if $i,j$ belong to the same cluster and $p_{ i,j}= q$ if $i,j$ belong to different clusters. Then the mean degree of each node is $ nr$ with $r=\frac{ p+q}{ 2}$, so that the dynamics is
\begin{equation*}
\dd \theta^{i,n}_t = \frac{1}{n r}\sum_{j=1}^n \xi_{ij}^{(n)}\Gamma \left(\theta^{i,n}_t, \theta^{j,n}_t \right)\dd t+\dd B^i_t, \quad 0 < t \leq T, \quad i= 1, \dots, n,
\end{equation*}
There are here two global empirical measures, each on one cluster: $ \mu^{ n, C_{ l}}= \frac{ 2}{ n} \sum_{ i\in C_{ l}^{ n}} \delta_{ \theta^{ i, n}}$. Suppose for simplicity that the initial condition is chosen independently from the graph, with the appropriate convergence hypotheses of the empirical measure as $n\to\infty$.
Similar standard Ito's calculations as for Lemma~\ref{lem:mu^n_t} (note here that \eqref{eq:hatxi} has to be replaced by $ \hat{ \xi}_{ij}^{ (n)}:= \xi_{ij}^{ (n)} - p_{ i,j}$), show that, for $ \alpha= \frac{ p}{ p+q}$, $f$ regular and $M^{ n,l}$ appropriate martingales
\begin{equation}
\label{eq:mun_SBM}
\begin{cases}
\< \mu^{ n, C_{ 1}}_t, f> &= \< \mu^{ n, C_{ 1}}_0, f> + \int_0^t \left\langle  \mu^{ n, C_{ 1}}_s\, ,\, \frac 12 \partial_\theta^2 f + \partial_{ \theta} f \left( \Gamma \ast \left\lbrace \alpha\mu_{ s}^{ n, C_{ 1}} + (1- \alpha)\mu_{ s}^{ n, C_{ 2}} \right\rbrace\right)\right\rangle \dd s \\
&+ \< M^{ n, 1}_t, f > + \int_0^t \frac{2}{n^2 r}\sum_{i\in C_{ 1}} \sum_{j=1}^n \hat{ \xi}_{ij}^{ (n)} \Gamma\left(\theta^{i,n}_s, \theta^{j,n}_s\right)\partial_\theta f\left(\theta^{i,n}_s\right) \dd s,\\
\< \mu^{ n, C_{ 2}}_t, f> &= \< \mu^{ n, C_{ 2}}_0, f> + \int_0^t \left\langle \mu^{ n, C_{ 2}}_s\, ,\, \frac 12 \partial_\theta^2 f + \partial_{ \theta}f \left( \Gamma \ast \left\lbrace (1- \alpha) \mu_{ s}^{ n, C_{ 1}} + \alpha \mu_{ s}^{ n, C_{ 2}} \right\rbrace\right) \right\rangle \dd s\\
& + \< M^{ n, 2}_t, f >  + \int_0^t \frac{2}{n^2 r}\sum_{i\in C_{ 2}} \sum_{j=1}^n \hat{ \xi}_{ij}^{ (n)} \Gamma\left(\theta^{i,n}_s, \theta^{j,n}_s\right)\partial_\theta f\left(\theta^{i,n}_s\right) \dd s.
\end{cases}
\end{equation}
The mean-field limit is given by the system of coupled PDEs
\begin{equation}
\label{eq:mu_SBM}
\begin{cases}
\< \mu^{C_{ 1}}_t, f> &= \< \mu^{C_{ 1}}_0, f> + \int_0^t \left\langle \mu^{C_{ 1}}_s\, ,\, \frac 12 \partial_\theta^2 f + \partial_{ \theta} f\left(\Gamma \ast \left\lbrace \alpha\mu_{ s}^{C_{ 1}} + (1- \alpha)\mu_{ s}^{C_{ 2}}\right\rbrace\right)\right\rangle \dd s,\\
\< \mu^{C_{ 2}}_t, f> &= \< \mu^{C_{ 2}}_0, f> + \int_0^t \left\langle  \mu^{C_{ 2}}_s\, ,\, \frac{ 1}{ 2} \partial_\theta^2 f + \partial_{ \theta} f\left(\Gamma \ast \left\lbrace (1- \alpha)\mu_{ s}^{C_{ 1}} + \alpha\mu_{ s}^{C_{ 2}}\right\rbrace\right) \right\rangle  \dd s.
\end{cases}
\end{equation}
Setting now the fluctuations processes
\begin{equation}
\label{eq:fluct_SBM}
\eta^{ n}:= \left( \eta^{ n, C_{ 1}}, \eta^{ n, C_{ 2}}\right):= \left( \sqrt{ \frac{ n}{ 2}} \left( \mu^{ n, C_{ 1}} - \mu^{ C_{ 1}}\right), \sqrt{ \frac{ n}{ 2}} \left( \mu^{ n, C_{ 2}} - \mu^{ C_{ 2}}\right)\right),
\end{equation}
from \eqref{eq:mun_SBM} and \eqref{eq:mu_SBM} and using the same techniques as in the present paper, it is not difficult to show that the proper limit for \eqref{eq:fluct_SBM} is given by
\begin{equation}
\label{eq:limfluct_SBM}
\begin{cases}
\eta^{C_{ 1}}_t &=\eta^{C_{ 1}}_0 + \int_0^t \left\lbrace \mathcal{ L}_{ \mu_s}^{ (1), \ast}\eta^{C_{ 1}}_s +(1- \alpha) \mathcal{ U} \left(\mu_{ s}^{ C_{ 1}}\right)^{ \ast} \eta_{ s}^{C_{ 2}} \right\rbrace \dd s + W^{1}_t,\\
\eta^{C_{ 2}}_t&=  \eta^{C_{ 2}}_0 + \int_0^t \left\lbrace \mathcal{ L}_{ \mu_s}^{ (2), \ast}\eta^{C_{ 2}}_s + (1-\alpha) \mathcal{ U}\left(\mu_{ s}^{ C_{ 2}}\right)^{ \ast} \eta_{ s}^{C_{ 1}} \right\rbrace \dd s +W^{2}_t,
\end{cases}
\end{equation}
for 
\begin{align*}
\mathcal{ L}_{ \mu_s}^{ (1)}f&=\frac 12 \partial_\theta^2 f  + \partial_{ \theta} f \left( \Gamma \ast \left\lbrace \alpha\mu_{ s}^{C_{ 1}} + (1- \alpha)\mu_{ s}^{C_{ 2}}\right\rbrace\right) + \alpha \mathcal{ U} \left(\mu_{ s}^{ C_{ 1}}\right)f,\\
\mathcal{ L}_{ \mu_s}^{ (2)}f&=\frac 12 \partial_\theta^2 f  + \partial_{ \theta} f \left( \Gamma \ast \left\lbrace (1-\alpha)\mu_{ s}^{C_{ 1}} + \alpha\mu_{ s}^{C_{ 2}}\right\rbrace\right) + \alpha \mathcal{ U} \left(\mu_{ s}^{ C_{ 2}}\right)f,\\
\mathcal{ U}(\mu)f&=\int  \partial_{ \theta} f(\theta)  \Gamma \left(\theta, \cdot\right) \mu({\rm d}\theta).
\end{align*}
Here $ \left(W^{ 1}, W^{ 2}\right)$ are independent Gaussian process with covariance $ \mathbf{E} \left[ W_{ s}^{ l}(f) W_{ t}^{ l}(g)\right]= \int_{ 0}^{t\wedge s} \left\langle \mu_{ u}^{ C_{ l}}\, ,\, \partial_{ \theta}f\partial_{ \theta}g\right\rangle {\rm d}u$. We retrieve from this calculations several easy particular cases:
\begin{enumerate}
\item when $q=0$, that is $ \alpha=1$: then we see from \eqref{eq:limfluct_SBM} that $ (\eta^{ C_{ 1}}, \eta^{ C_{ 2}})$ are independent copies of the same process solving \eqref{eq:limit_eta_only} (this is of course normal as the two clusters $C_{ 1}$ and $C_{ 2}$ are now disjoint).
\item when $p=q$, that is $ \alpha= \frac{ 1}{ 2}$: for general initial condition $ \mu^{ C_{ 1}}_{ 0}$ and $\mu_{ 0}^{ C_{ 2}}$ (not necessarily identical),  $ \mu:= \frac{ 1}{ 2}( \mu^{ C_{ 1}} + \mu^{ C_{ 2}})$ solves the mean-field \eqref{eq:limit PDE} with initial condition $\frac{ 1}{ 2}( \mu^{ C_{ 1}}_{ 0} + \mu^{ C_{ 2}}_{ 0})$ and we see from \eqref{eq:limfluct_SBM} that $ \eta:= \frac{ 1}{ 2} \left( \eta^{ C_{ 1}} + \eta^{ C_{ 2}}\right)$ solves also \eqref{eq:limit_eta_only} (this is again obvious as the system consists now in a single Erd\H{o}s-R\'enyi with an homogeneous parameter $p=q$).
\item in the particular case $ \mu^{ C_{ 1}}_{ 0}= \mu_{ 0}^{ C_{ 2}}$, the statements of the previous item remain true for any $ \alpha\in [0,1]$.
\end{enumerate}

\subsection{A look at the literature}
\label{ss:literature}
Interacting particle systems of mean field type have been repeatedly addressed in the literature of the last fifty years, see \cite{McKean65, Oelsch1984, Gartner, Sznitman1991}, this list being in no way exhaustive. The first results focus on the Law of Large Numbers (LLN) for the empirical measure together with existence and uniqueness related to the limit Fokker-Planck equation, see, e.g., \cite{funaki84, leonard86, Gartner}. Shortly after, the Central Limit Theorem (CLT) \cite{braunHepp77, TanakaHitsuda81, sznitman85, tanaka84, hitsudaMitoma86} has been established in several scenarios. Two main methods have been proposed in the literature to study the fluctuations around the LLN limit. One method \cite{TanakaHitsuda81, sznitman84, sznitman85, ShigaTanaka85, BudhirajiaWu16} consists in focusing on the fluctuation field
	\begin{equation*}
		\left( \sqrt{ n} \left( \left\langle \mu_{ n}\, ,\, f\right\rangle- \left\langle \mu\, ,\, f\right\rangle\right), \quad f \in \mathcal{ F} \right),
	\end{equation*}
with $\cF$ some function space, typically $ \mathcal{ F}= \left\lbrace f\in L^{ 2}(\cC([0,T], \R^d)), \mathbf{ E}_{ \mu}(f)=0\right\rbrace$, and to prove that, as $n\to \infty$, it converges to some Gaussian field with prescribed covariance. The other method, the one followed here and firstly proposed in \cite{mitoma85, hitsudaMitoma86, Fernandez1997}, directly addresses the fluctuation process \eqref{def:fluctuation_process}, i.e., $\eta^n = \sqrt{n} (\mu_n - \mu)$, and aims at showing that $\eta^n$ converges to the solution of a stochastic partial differential equation. See \cite[Chapter 3]{TanakaHitsuda81} for an interesting relation between these two approaches when the particle interaction is linear.

We must stress that the first method has three main drawbacks: (i) the convergence only concerns finite-dimensional marginals (ii) its proof relies heavily on exchangeability properties of the system (that is in no way applicable in our quenched context) and (iii) the covariance of the limit Gaussian field is not explicit but involves Radon-Nikodym derivatives and integrals of the dynamics operators. The second method is more challenging, yet it translates the limit dynamics in terms of a "classical" linear SPDE that can be further studied, see \cite{kurtzXiong04} for a general result on this kind of limit equations. Finally, observe that in some cases, the CLT for finite dimensional marginals can be derived from a Large Deviation Principle, see \cite[Chapter 4]{daiPra96} and \cite{Bolthausen86}.

In the case of interacting particle systems on graphs, most of the literature focuses on LLN \cite{Delattre2016, coppini2018law, bayraktar2020graphon, bet2020weakly}, with some exceptions concerning large deviations \cite{coppini2018law,MacLaurin2020, OliveiraReis2019,dupuis_medvedev_20}. It is interesting to compare the LLN established in Theorem~\ref{th:conv_empmeas} with the previous ones in the literature. To the authors' knowledge, there exists no result under weaker assumptions on graph and initial condition than our Theorem \ref{th:conv_empmeas}. Although the assumption $np_n\to \infty$ is common across other results (but only appears in an annealed context, e.g., \cite{bayraktar2020graphon, OliveiraReis2019}, whereas the best condition so far in a quenched context was $ \liminf_{ n\to\infty}\frac{ np_{ n}}{ \log(n)}>0$ \cite{coppini2018law}), the only work assuming general initial condition that may depend on the graph is given by \cite{Coppini2019}. However, \cite{Coppini2019} focuses on a specific particle system, the Kuramoto model, and proves a result in $H^{-1}$-norm whereas Theorem \ref{th:conv_empmeas} is stated in terms of the classical weak convergence. Given the results \cite{Oliveira2020, Lacker2019} on the sparse regime $np_n \to c >0$, the condition \eqref{eq:dilution_cond_opt} of Theorem \ref{th:conv_empmeas} appears to be optimal.

To the authors knowledge, there exists only one work addressing the CLT for diffusions on graphs, i.e., \cite{BBW19}, which addresses interacting diffusions on $\R^d$ and dense inhomogeneous random graphs. Despite the generality of the particle systems and the graphs under consideration, the CLT statement is substantially weaker than the one presented here: it concerns finite dimensional marginals, the underlying graph sequence is dense and the result is proven in probability with respect to the graph, whereas we prove a $\bbP_g$-a.s. convergence and consider graph sequences in possibly diluted regimes. Finally, the initial condition in \cite{BBW19} are taken i.i.d. and independent of the graph, whereas we only suppose the weak convergence of $ \mu_{ 0}^{ n}$ towards some probability measure, see Assumption \ref{ass:initial}.

\medskip

\subsection{Organisation of the paper} 
The rest of the paper is organised as follows: we present in Section~\ref{sec:grothendieck} the main argument that we use for closing the hierarchy of empirical measures, that is extension of Grothendieck inequalities. Section~\ref{sec:global fluctuations} contains the proofs of the global fluctuations result (Theorem~\ref{th:limit eta and hat eta}). The local fluctuations (Theorem~\ref{th:local_fluct}) are treated in Section~\ref{sec:proof_local_fluct}. We gather in App.~\ref{sec:sobolev} and~\ref{sec:App concentration} some technical estimates (namely Sobolev inequalities and further concentration estimates). App.~\ref{sec:conv_empmeas} gives the proof of Theorem~\ref{th:conv_empmeas}. Uniqueness results are gathered in App.~\ref{sec:uniqueness}. We finally discuss on the importance of renormalisation of the interaction in App.~\ref{sec:renorm}.

\section{Grothendieck inequality and concentration estimates}
\label{sec:grothendieck}
Before giving the main result on Grothendieck inequality, we give a characterisation of the processes $\eta^n$ and $\hat{\eta}^n$ in terms of semimartingales.

\subsection{Characterisation of the processes}
For the proofs of Theorem~\ref{th:limit eta and hat eta} and Theorem~\ref{th: limit eta independent}, we rely on the following characterisation of the processes $\eta^n$ and $\hat \eta^n$. Their proof can be found at the end of Section \ref{sec:global fluctuations}. For the definition of the Doob-Meyer process for Hilbert-valued martingale, we refer to \cite{metivier2011semimartingales}. Recall the definitions of $\cL^{(1)}$ and $\cL^{(2)}$, see \eqref{eq:L1} and \eqref{eq:L2} respectively.

\begin{proposition}\label{prop:charac eta tilde eta} 
For any $r> \frac{ 3}{ 2}$ and $r^{ \prime}> \frac{ 5}{ 2}$ the joint process $ (\eta^{ n}, \hat \eta^n)$ belongs $ \mathbb{ P}\otimes \mathbf{ P}$-a.s. to $ \mathcal{ C} \left([0, T], H^{ -r} \left( \mathbb{ T}\right)\otimes H^{ -r^{ \prime}}\left(\mathbb{ T}^{ 2}\right)\right)$. Moreover, $ \eta^{ n}$ and $ \hat{ \eta}^{n}$ satisfy the following semimartingale representations in $H^{-r_{ 1}}(\bbT)$ and $H^{-r_{ 1}}(\bbT^{ 2})$ respectively,
\begin{equation}\label{eq:decomp eta_n}
\begin{split}
\eta^n_t =  \eta_{ 0}^{ n} + \int_{0}^{t} \mathcal{ L}_{ \mu^n_s}^{(1), \ast} \eta_{ s}^{ n} {\rm d}s +\int_{ 0}^{t}\Theta^{ \ast}\hat{\eta}^n_s {\rm d}s+ W_{ t}^{ n}, 
\end{split}
\end{equation}
and
\begin{equation}\label{eq:decomp hat eta_n}
\begin{split}
\hat{\eta}^n_t =  \hat{\eta}_{ 0}^{ n} + \int_{0}^{t} \mathcal{ L}_{ \mu^n_s}^{(2), \ast} \hat{\eta}_{ s}^{ n} {\rm d}s + \sqrt{n}  C_{ t}^{ n} + \hat{ W}_{ t}^{ n}, 
\end{split}
\end{equation}
where, for any regular test function $g$
\begin{equation}\label{eq:def Cn}
\begin{split}
C^n_t(g) =& \int_0^t \frac 1{n^3} \sum_{i,j,k=1}^n \hat{\xi}^{(n)}_{ij} \hat{\xi}^{(n)}_{ik} \partial_{\theta_1} g(\theta^{i,n}_s, \theta^{j,n}_s)  \Gamma(\theta^{i,n}_s, \theta^{k,n}_s)\dd s  \\
&+ \int_0^t \frac 1{n^3} \sum_{i,j,k=1}^n \hat{\xi}^{(n)}_{ij} \hat{\xi}^{(n)}_{jk} \partial_{\theta_2} g(\theta^{i,n}_s, \theta^{j,n}_s)  \Gamma(\theta^{j,n}_s, \theta^{k,n}_s)  \dd s,
\end{split}
\end{equation}
the process $(W_{ t}^{ n})_{t \in [0,T]}$ is a martingale in $ \cC \left([0, T], H^{ -r} \left(\mathbb{ T}\right)\right)$ for  $r> \frac{ 3}{ 2}$, with Doob-Meyer process $\llangle W^{ n} \rrangle$ taking values in $ \mathcal{ L} \left(H^{ r}, H^{ -r}\right)$ and given for $t \in [0,T]$ and $\varphi, \psi\in H^{ r}$ by
\begin{equation}
\label{eq:DM_Mn}
\llangle W^{ n}\rrangle_{ t}\cdot \varphi(\psi)= \left\langle W^{ n}(\varphi)\, ,\, W^{ n}(\psi)\right\rangle_{ t}= \frac{ 1}{ n} \sum_{i=1}^{ n} \int_{ 0}^{t} \varphi^{ \prime}(\theta_{ s}^{ i, n}) \psi^{ \prime}(\theta_{ s}^{ i, n}) {\rm d}s,
\end{equation}
and the process $ (\hat{ W}_{ t}^{ n})_{t \in [0,T]}$ (whose explicit form is given in \eqref{eq:hatMn}) is a martingale in $ \cC \left([0, T], H^{ -r}\left(\mathbb{ T}^{ 2}\right)\right)$ for  $r\geq 4$.
\end{proposition}
Note that one point of the proof, see Proposition \ref{prop:Wn_prop}, will be to show that the noise $\hat{ W}^{ n}$ goes to $0$ as $n\to\infty$.

A crucial step in the proofs of Theorem~\ref{th:limit eta and hat eta} and Theorem~\ref{th: limit eta independent} is to show that the graph dependent term $C_{ n}$ in \eqref{eq:def Cn} goes effectively to $0$ as $n\to\infty$. In view of its structure, a strategy would be to take advantage on concentration estimates, with respect to the graph, on quantities such as $\frac 1{n^3} \sum_{i,j,k=1}^n \hat{\xi}^{(n)}_{ij} \hat{\xi}^{(n)}_{ik} u_{ i,jk}$ where $ (u_{ i,j,k})$ is any fixed sequence, bounded by $1$. However, $u_{ i,j,k}:= \partial_{\theta_1} g(\theta^{i,n}_s, \theta^{j,n}_s)  \Gamma(\theta^{i,n}_s, \theta^{k,n}_s)$ in \eqref{eq:def Cn} depends in a highly nontrivial manner on the graph sequence $(\xi_{ij}^{ (n)})$, and standard concentration results (e.g. Bernstein inequalities) cannot be applied. The main novelty of the present work is to circumvent this difficulty in using multi-linear extensions of the classical Grothendieck inequality proved by R.~Blei \cite{blei1979,blei2014grothendieck}.

The use of Grothendieck inequalities is detailed in the following subsection~\ref{ss:grothendieck}, in the case of the term $C_{ n}$: an immediate consequence of Propositions~\ref{prop:bound Cn} and~\ref{prop:concentration_SnT} is that the term $C_{ n}$ does not contribute to the limit as $n\to\infty$. Note however that this strategy will be applied repeatedly in this work to various other functionals of the graph sequence $ \left(\xi_{ij}^{ (n)}\right)$. For the sake of readability, we postpone the definitions of these other quantities and the corresponding concentration results to Appendix~\ref{sec:App concentration}.
\subsection{Grothendieck inequality}
\label{ss:grothendieck}
The classical Grothendieck equality has received a lot of attention in the recent years, as it was shown to be a powerful tool for the study of graph concentration \cite{alon_naor_2006, guedon_vershynin_2016}. Let us consider an infinite dimensional Euclidian space with coordinates indexed by a space $A$:
\begin{equation*}
l^2(A) = \left\{x=(x_\ga)_{\ga\in A}\in \bbC^A: \, \sum_{\ga\in A}\left\vert x_\ga\right\vert^2 <\infty  \right\},
\end{equation*}
endowed with the usual scalar product $ \left\langle x\, ,\, y\right\rangle_{l^2(A)}=\sum_{\ga\in A} x_\ga \bar y_\ga$ and the associated norm $\Vert \cdot\Vert_{l^2(A)}$. In this context
the classical Grothendieck inequality states that there exists a universal constant $\mathcal{K}$ such that for any finite scalar array $\left(a_{jk}\right)$,
\begin{multline}
\label{eq:classical_grothendieck}
\sup\left\{ \left|\sum_{j,k} a_{jk} \langle x_j,y_k\rangle_{l^2(A)}\right|:\, x_j,y_k\in l^2(A),\, \max(\Vert x_j\Vert_{l^2(A)}, \Vert y_k\Vert_{l^2(A)})\leq 1\right\}\\
\leq \mathcal{K}\sup\left\{\left|\sum_{j,k}a_{jk} s_j t_k\right|:\, s_j =\pm 1, \, t_k  =\pm 1\right\}.
\end{multline}
This inequality is known to fail in general when the scalar product is replaced by a bounded trilinear functional, see for example \cite{pisier2012grothendieck}. However, in \cite{blei1979,blei2014grothendieck} R. Blei describes a family of multilinear functionals for which this inequality remains valid. We present this result in the following.

Consider a positive integer $m$ and a sequence $\cU=(S_1,\ldots,S_N)$ of non empty sets that satisfy $\cup_{i=1}^N S_i=\{1,2,\ldots,m\}$. For $\ga=(\ga_j)_{1\leq j\leq m}\in A^m$ define the projections $\pi_{S_i}(\ga)=(\ga_j)_{j\in S_i}$. Consider, for $\theta :A^m\rightarrow \bbC$ bounded, the functional $\nu_{\theta,\cU}: l^2\left(A^{|S_1|}\right)\times\ldots\times l^2\left(A^{|S_N|}\right)\rightarrow \bbC$ defined as follows:
\begin{equation*}
\nu_{\cU,\theta}(x_1,\ldots,x_N) = \sum_{\ga\in A^m} \theta(\ga) x_1(\pi_{S_1}(\ga))\cdots x_N(\pi_{S_N}(\ga)).
\end{equation*}
The functional $\nu_{\cU,\theta}$ will satisfy a Grothendieck inequality under some assumptions on the covering sequence $\cU$ and on $\theta$. Following the notations of \cite{blei2014grothendieck} we denote, for $1\leq j\leq m$, by $k_j$ the incidence of $j$ in the covering sequence $\cU$, that is
\begin{equation*}
k_j(\cU)= \left|\left\{i\in \{1,\ldots, N\}:\, j\in S_i\right\}\right|,
\end{equation*}
and by $\mathcal{I}_\cU$ the minimal incidence:
\begin{equation*}
\mathcal{I}_\cU = \min\left\{k_j(\cU):\, j\in \{1,\ldots,m\}\right\}.
\end{equation*}
Moreover, we say that the mapping $\theta$ belongs to the space $\tilde \cV_{\cU}(A^m)$ if there exists a probability space $(\tilde \Omega, \tilde \cA, \tilde \mu)$ and family of functions $g^{(i)}_\omega$, indexed by $\omega\in \tilde \Omega$ and defined on $A^{S_i}$ for $i\in \{1,\ldots,N\}$, such that for all $x\in \bbC^m$ the mappings 
\begin{equation*}
\omega \mapsto \left(g^{(i)}_\omega \circ \pi_{S_i}\right)(x) \quad \text{and} \quad \omega \mapsto \left\Vert g^{(i)}_\omega \right\Vert_\infty 
\end{equation*}
are measurable, with 
\begin{equation}\label{eq: int norm bounded}
\int_{\tilde \Omega} \left\Vert g^{(1)}_\omega \right\Vert_\infty \cdots \left\Vert g^{(N)}_\omega \right\Vert_\infty \dd \tilde\mu(\omega)<\infty ,
\end{equation}
and such that we have the representation
\begin{equation*}
\theta(x) =\int_{\tilde \Omega}\left(g^{(1)}_\omega \circ \pi_{S_i}\right)(x)\cdots \left(g^{(N)}_\omega \circ \pi_{S_i}\right)(x)\dd \tilde \mu(\omega).
\end{equation*}
The norm $\Vert \theta\Vert_{\tilde \cV_{\cU}(A^m)}$ corresponds to the infimum of the left-hand side of \eqref{eq: int norm bounded} over all possible representations. The following generalisation of the Grothendieck inequality corresponds to Theorem 11.11 and Section 12.4 in \cite{blei2014grothendieck}.
\begin{theorem}\label{th:Grothendieck}
Suppose that $\mathcal{I}_\cU\geq 2$ and that $\theta\in \tilde \cV_{\cU}(A^m)$. Then there exists a positive constant $\mathcal{K}_\cU$, depending only on the covering $\mathcal{U}$, such that for any finitely supported scalar $n$-array $a_{j_1\ldots j_N}$,
\begin{multline*}
\sup\left|\left\{\sum_{j_1\ldots,j_N}a_{j_1\ldots j_N} \nu_{\cU,\theta} (x_1,\ldots,x_N):\, \Vert x_1\Vert_{l^2(A^{|S_1|})}\leq 1,\ldots, \Vert x_N\Vert_{l^2(A^{|S_N|})}\leq 1\right\}\right|\\
\leq \mathcal{K}_\cU \Vert \theta\Vert_{\tilde \cV_{\cU}(A^m)}\sup\left\{\left|\sum_{j_1,\ldots,j_N}a_{j_1\ldots j_N} s_{1,j_1}\cdots s_{N,j_N}\right|:\, s_{1,j_1}=\pm 1,\ldots, s_{N,j_N}=\pm1\right\}.
\end{multline*}
\end{theorem}
In this paper we will use this result with $A=\bbZ$, $m=2$, $N=3$ and $\theta(x)=1$, which is trivially an element of $\cV_{\cU}(\bbZ^2)$. Let us now show how this inequality can be applied to bound the term $C^n$ in \eqref{eq:def Cn}. 

A remark on notation: as $\hat{\xi}^{(n)}_{ij} \hat{\xi}^{(n)}_{ik}$ in the first term of \eqref{eq:def Cn} is the recentered version of $\xi^{(n)}_{ij} \xi^{(n)}_{ik}$, which encodes for the presence of both (directed) edges $i\to j$ and $i\to k$ in the graph, it is natural to label the corresponding quantity by the local tree $\ijik$ (where each $\ij$ stands for a directed edge $i\to j$). Every similar quantities in the following will be labeled according to this principle. Therefore, define
\begin{equation}
S_{ n}^{\ijik}:= \sup_{ r,s,t\in \left\lbrace\pm 1\right\rbrace^{ n}}\left|\frac{ 1}{ n^{ 3}} \sum_{ i,j,k=1}^{ n} \hat{\xi}^{(n)}_{ij} \hat{\xi}^{(n)}_{ik} r_{ i} s_{ j} t_{ k}\right|, \label{eq:Sijik}
\end{equation}
and
\begin{equation}
S_{ n}^{\ijjk}:= \sup_{ r,s,t\in \left\lbrace\pm 1\right\rbrace^{ n}}\left|\frac 1{n^3} \sum_{i,j,k=1}^n \hat{\xi}^{(n)}_{ij}  \hat{\xi}^{(n)}_{jk}r_{ i} s_{ j} t_{ k}\right|.\label{eq:Sijjk}
\end{equation}

\begin{proposition}\label{prop:bound Cn}
Let $C^n$ given in \eqref{eq:def Cn}. Then there exists a constant $C_\Gamma$, depending only on $\Gamma$, such that for $n$ large enough
\begin{equation}
\label{eq:control_Cn}
\sup_{t\in[0,T]} \norm{C^n_t}_{-3} \leq TC_\Gamma \left(S_{ n}^{\ijik}+S_{ n}^{\ijjk}\right).
\end{equation}
\end{proposition}

\begin{proof}
Let $C^{n,1}$ and $C^{n,2}$ such that $\< C^n, f > = \< C^{n,1}, f > + \< C^{n,2}, f>$, the first term $C^{n,1}$ is given by
\begin{equation*}
\< C^{n,1}_t, f> = \int_0^t \frac 1{n^3} \sum_{i,j,k=1}^n \hat{\xi}^{(n)}_{ij} \hat{\xi}^{(n)}_{ik}  \left[  \partial_{\theta_1} f(\theta^{i,n}_s, \theta^{j,n}_s) \right] \Gamma(\theta^{i,n}_s, \theta^{k,n}_s) \dd s,
\end{equation*}
and the second term $C^{n,2}$ by
\begin{equation*}
\< C^{n,2}_t, f> = \int_0^t \frac 1{n^3} \sum_{i,j,k=1}^n \hat{\xi}^{(n)}_{ij}  \hat{\xi}^{(n)}_{jk} \left[ \partial_{\theta_2} f(\theta^{i,n}_s, \theta^{j,n}_s)\right] \Gamma(\theta^{j,n}_s, \theta^{k,n}_s)  \dd s.
\end{equation*}
Let's focus on $\< C^{n,1}_t, f>$. The point is to establish a bound on $\< C^{n,1}_t, f>$ that only depends on $T$, some norm of $\Gamma$, the $H^3(\bbT^2)$-norm of $f$, and the underlying graph $\xi^{(n)}$, using Theorem~\ref{th:Grothendieck}.
To simplify notations, let $\Xi^{(n)}_{ijk} : = n^{-3} \hat{\xi}^{(n)}_{ij} \hat{\xi}^{(n)}_{ik}$ for every choice of $i,j$ and $k$. Then
\begin{equation}
\label{C:eq:C^{n,1}}
\< C^{n,1}_t, f> = \int_0^t \sum_{i,j,k=1}^n \Xi^{(n)}_{ijk}  \left[  \partial_{\theta_1} f(\theta^{i,n}_s, \theta^{j,n}_s) \right] \Gamma(\theta^{i,n}_s, \theta^{k,n}_s) \dd s.
\end{equation}
Let $(e_a)_{a \in \Z}$ be the canonical basis of $L^2(\bbT)$: observe that $\partial_{\theta_1} f$ and $\Gamma$ can be rewritten in $L^2(\bbT^2)$ as
\begin{eqnarray*}
\partial_{\theta_1} f(\theta_1, \theta_2) &=& \sum_{a \in \Z} e_a(\theta_2) \int_\bbT \partial_{\theta_1} f(\theta_1, \theta) \bar e_a(\theta) \dd \theta, \\
\Gamma(\theta_1, \theta_2)& =&  \sum_{b \in \Z} e_b(\theta_2) \int_\bbT \Gamma (\theta_1, \theta) \bar e_b(\theta) \dd \theta.
\end{eqnarray*}
Moreover, define $y_{1,i}, y_{2,j}$ and $y_{3,k}$ by
\begin{equation}
\begin{split}
\label{C:eq:y}
y_{1,i}(a,b) &= \left( \int_\bbT \partial_{\theta_1} f(\theta^{i,n}_s, \theta) \bar e_a(\theta) \dd \theta \right)  
\left(\int_\bbT \Gamma(\theta^{i,n}_s, \theta) \bar e_b(\theta) \dd \theta\right), \; a,b \in \Z,\\
y_{2,j}(a) &= e_a(\theta^{j,n}_s), \quad a \in \Z,\\
y_{3,k}(b) &= e_b(\theta^{k,n}_s), \quad b \in \Z.
\end{split}
\end{equation}
With the previous notation, the term $\< C^{n,1}_t, f>$ can be decomposed as follows:
\begin{equation*}
\< C^{n,1}_t, f> = \int_0^t \sum_{i,j,k=1}^n \Xi^{(n)}_{ijk}  \sum_{a,b \in \Z} y_{1,i}(a,b) y_{2,j}(a) y_{3,k}(b) \dd s.
\end{equation*}
In order to apply Theorem~\ref{th:Grothendieck}, we replace the factors defined in \eqref{C:eq:y} by $\ell^2$-summable elements. For some $\gd>0$ define the following functions:
\begin{equation*}
\begin{split}
x_{1,i}(a,b) &= \, C^2_\delta \, \big((1+a^2)(1+b^2)\big)^{1/4 + \delta} y_{1,i}(a,b), \quad a,b \in \Z, \\
x_{2,j}(a) &= \, C^{-1}_\delta \, (1+a^2)^{-1/4-\delta} y_{2,j}(a), \quad a \in \Z,\\
x_{3,k}(b) &= \, C^{-1}_\delta \, (1+b^2)^{-1/4 -\delta} y_{3,k}(b), \quad b \in \Z,
\end{split}
\end{equation*}
where $C_\delta = (\sum_{a\in \Z} (1+a^2)^{-1/2-2\delta})^{1/2}$ for some $\delta > 0$. Observe that $x_{1,i}(a,b) x_{2,j}(a) x_{3,k}(b) = y_{1,i}(a,b) y_{2,j}(a) y_{3,k}(b)$ for every $a,b \in \Z$. By construction, the $\ell^2(\bbZ)$-norms of $x_{2, j}$ and $x_{3, k}$ are equal to 1. Moreover the $\ell^2(\bbT^2)$-norm of $x_{1,i}$ can be bounded by
\begin{equation*}
\begin{split}
&\norm{x_{1,i}}^2_{\ell^2(\bbT^2)}\\
&= C^4_\delta \sum_{a,b \in \bbZ} \big((1+a^2)(1+b^2)\big)^{1/2+2\delta} \left\vert \int_\bbT \partial_{\theta_1} f(\theta^{i,n}_s, \theta) \bar e_a(\theta) \dd \theta \right\vert^2
\left\vert\int_\bbT \Gamma(\theta^{i,n}_s, \theta) \bar e_b(\theta) \dd \theta\right\vert  ^2\\
&= C^4_\delta \left( \sum_{a \in \Z} (1+a^2)^{1/2+2\delta} \left| \int_\bbT \partial_{\theta_1} f(\theta^{i,n}_s, \theta) \bar e_a(\theta) \dd \theta \right|^2\right) \\
&\qquad \qquad \qquad \qquad\qquad \times \left( \sum_{b \in \Z} (1+b^2)^{1/2+2\delta} \left| \int_\bbT \Gamma(\theta^{i,n}_s, \theta) \bar e_b(\theta) \dd \theta \right|^2\right)\\
&\leq C^4_\delta \, C_{\Gamma, \delta} \sum_{a \in \Z} (1+a^2)^{1/2+2\delta} \left| \int_\bbT \partial_{\theta_1} f(\theta^{i,n}_s, \theta) \bar e_a(\theta) \dd \theta \right|^2 = C^4_\delta \, C_{\Gamma, \delta} \norm{\partial_{\theta_1} f(\theta^{i,n}_s, \cdot)}^2_{H^{1/2 +2\delta}(\dd \theta_2)},
\end{split}
\end{equation*}
where the constant $C_{\Gamma, \delta}$ only depends on $\Gamma$ and $\delta$. By the definition of the fractional $H^s$-norm, for $0<s<1$, one has that
\begin{equation*}
\begin{split}
\norm{\partial_{\theta_1} f(\theta^{i,n}_s, \cdot)}^2_{H^{1/2 +2\delta}(\dd \theta_2)} = \norm{\partial^{1/2 + 2\delta}_{\theta_2}\partial_{\theta_1} f(\theta^{i,n}_s, \cdot)}^2_{L^2(\dd \theta_2)}
\end{split}.
\end{equation*}
We can bound the previous norm with a fractional Hilbert norm on $\bbT^2$, no longer dependent on the value of $\theta^{i,n}_s$, see Lemma \ref{lem:mixed_sobolev_ineq}. Thus, it holds that
\begin{equation*}
\begin{split}
\sup_{\theta_1 \in \bbT} \norm{\partial^{1/2 + 2\delta}_{\theta_2}\partial_{\theta_1} f(\theta^{i,n}_s, \cdot)}^2_{L^2(\dd \theta_2)} &\leq C \norm{\partial^{1/2 + 2\delta}_{\theta_2}\partial_{\theta_1} f(\cdot, \cdot)}^2_{H^1(\dd \theta_1, \dd \theta_2)}.
\end{split}
\end{equation*}
This last expression is further bounded by $C' \norm{f(\cdot, \cdot)}_{H^{2+ 1/2 + 2\delta}(\dd \theta_1, \dd \theta_2)}$ because of Sobolev's embeddings. Choosing $\delta=1/4$, we conclude that there exists a constant $C_\Gamma$, depending only on $\Gamma$, such that
\begin{equation}
\label{eq:norm_z1}
\norm{x_{1,i}}^2_{\ell^2(\bbT^2)} \leq C_\Gamma \norm{f}_{H^3(\dd \theta_1, \dd \theta_2)}^2 < \infty.
\end{equation}
We are now able to apply Theorem~\ref{th:Grothendieck} with $A=\bbZ$, $m=2$, $N=3$, $\cU=\{\{1,2\},\{1\},\{2\}\}$ and $\theta(x)=1$:
\begin{equation}
\label{eq:Groth_Cn1}
\left|\sum_{i,j,k=1}^n \Xi^{(n)}_{ijk} \sum_{a,b \in \Z}  x_{1,i}(a,b) x_{2,j}(a) x_{3,k}(b)\right| \leq \mathcal{K}_\cU \norm{x_{1,i}}^2_{\ell^2(\bbT^2)} S_{ n}^{\ijik}.
\end{equation}
Taking the supremum in $t \in [0,T]$ and in $f\in H^{ 3}\left(\mathbb{ T}^{ 2}\right)$, we finally obtain
\begin{equation*}
\norm{C^{n,1}}_{\cC([0,T],H^{-3}(\bbT^2))} \leq \, T \, C_\Gamma S_{ n}^{\ijik}.
\end{equation*}
Using similar arguments, one can show that
\begin{equation*}
\norm{C^{n,2}}_{\cC([0,T],H^{-3}(\bbT^2))} \leq \, T \, C_\Gamma S_{ n}^{\ijjk},
\end{equation*}
which concludes the proof.
\end{proof}
Other controls on similar quantities have been gathered in Appendix~\ref{sec:App concentration}.
\subsection{Concentration estimates}
Recall the definition of the $ \hat{ \xi}_{ij}^{ (n)}$ in \eqref{eq:hatxi}, and the definitions of $S_{ n}^{\ijik}$ and $S_{ n}^{\ijjk}$ given in \eqref{eq:Sijik} and \eqref{eq:Sijjk} respectively. In view of Proposition~\ref{prop:bound Cn}, our aim is to obtain a bound on these two terms, as well as others, which will be of constant use in the paper:
\begin{definition}
\label{def:Sn}
For fixed $n\geq1$, define
\begin{align}
S_{ n}^{\ij}&:= \sup_{ s,t\in \left\lbrace\pm 1\right\rbrace^{ n}}\left|\frac{ 1}{ n^{ 2}} \sum_{ i,j=1} \hat{ \xi}^{(n)}_{ ij}s_{ i} t_{j}\right|, \label{eq:Sij}\\
S_{ n}^{\ikjk}&:= \sup_{ r,s,t\in \left\lbrace\pm 1\right\rbrace^{ n}}\left|\frac 1{n^3} \sum_{i,j,k=1}^n \hat{\xi}^{(n)}_{ik}  \hat{\xi}^{(n)}_{jk}r_{ i} s_{ j} t_{ k}\right|,\label{eq:Sikjk}\\
S_{ n}^{\ijki}&:= \sup_{ r,s,t\in \left\lbrace\pm 1\right\rbrace^{ n}}\left|\frac 1{n^3} \sum_{i,j,k=1}^n \hat{\xi}^{(n)}_{ij}  \hat{\xi}^{(n)}_{ki}r_{ i} s_{ j} t_{ k}\right|,\label{eq:Sijki}\\
S_{ n}^{\lijik}(l)&:= \sup_{ r,s,t\in \left\lbrace\pm 1\right\rbrace^{ n}}\left|\frac 1{n^3} \sum_{i,j,k=1}^n \hat{ \xi}^{(n)}_{ li}\hat{\xi}_{ij}^{ (n)}  \hat{\xi}^{(n)}_{ik}r_{ i} s_{ j} t_{ k}\right|,\ l\in \left\lbrace1, 2\right\rbrace,\label{eq:Slijik}\\
S_{ n}^{\lijk}(l)&:= \sup_{ r,s,t\in \left\lbrace\pm 1\right\rbrace^{ n}}\left|\frac 1{n^3} \sum_{i,j,k=1}^n \hat{ \xi}^{(n)}_{ li}\hat{\xi}_{ij}^{ (n)}  \hat{\xi}^{(n)}_{jk}r_{ i} s_{ j} t_{ k}\right|,\ l\in \left\lbrace1, 2\right\rbrace.\label{eq:Slijk}
\end{align}
\end{definition}

\begin{proposition}
\label{prop:concentration_SnT}
Under the assumption $np_{ n}\underset{n\rightarrow\infty}{\longrightarrow}\infty$, we have
\begin{equation}\label{eq:bound Snij}
\limsup_{n\rightarrow\infty} \, \sqrt{ np_{ n}} \,  S_{ n}^{ \ij}  \leq 3,\, \mathbb{ P}-a.s.
\end{equation}
Moreover there exists a positive constant $\kappa$ such that for all $ \mathcal{ T} \in \left\lbrace \ijik, \ijjk, \ikjk, \ijki\right\rbrace$,
\begin{equation}\label{eq:bound SnT}
\limsup_{n\rightarrow\infty} \, n p_n^2 \, S_{ n}^{ \mathcal{ T}}  \leq \kappa ,\, \mathbb{ P}-a.s.
\end{equation}
and such that for $l\in\{1,2\}$ and $ \mathcal{ T} \in \left\lbrace \lijik, \lijk \right\rbrace$,
\begin{equation}\label{eq:bound Snlijik}
\limsup_{n\rightarrow\infty} np_n^3\, S_{ n}^{\mathcal{ T} }(l)  \leq \kappa,\, \mathbb{ P}-a.s.
\end{equation}
\end{proposition}

\begin{proof}
Let us first prove \eqref{eq:bound Snij}. Relying on Bernstein's inequality and on a union bound, we obtain
\begin{equation*}
\bbP\left(\sup_{s,t\in \{\pm 1\}^n }\left|\sum_{i,j=1}^n \hat \xi_{ij}^{(n)}s_it_j\right|>t\right)\leq 2\cdot 4^n\exp\left(-\frac12\frac{t^2p_n}{2n^2+\frac{t}{3}}\right).
\end{equation*}
Thus, the choice $t= \frac{c}{\sqrt{np_n}}$ leads to
\begin{equation*}
\bbP\left(S_{ n}^{\ij}>\frac{c}{\sqrt{np_n}}\right)\leq  2\cdot 4^n\exp\left(-\frac12\frac{c^2 n}{2+\frac{c}{3\sqrt{np_n}}}\right),
\end{equation*}
which is summable for $c=3$ with the hypothesis $np_n \underset{n\rightarrow\infty}{\longrightarrow}\infty$. Let us now prove \eqref{eq:bound SnT} for $\cT=\ijjk$. The proof for the other cases are the same, up to a transposition of matrix. Remark that, considering $A$ the matrix $\left(\hat \xi_{ij}\right)_{1\leq i,j\leq n}$ and $S$ the diagonal matrix with values $(s_i)_{1\leq i\leq n}$ on the diagonal, we have, denoting $\Vert \cdot\Vert$ the operator norm of matrices,
\begin{equation*}
\sum_{ i,j,k=1} \hat{ \xi}_{ ij}^{(n)}\hat{ \xi}^{(n)}_{ jk}r_i s_{ j} t_{k} = r A S A t\leq n \Vert ASA\Vert\leq n \Vert A\Vert^2.
\end{equation*} 
Consider the matrix $A$ multiplied by $p_n$, its coefficients are bounded by one: it is well known (see for example Corollary 2.3.5 in \cite{tao2012}), that there exist absolute constants $c$ and $C$ such that, for all $\lambda \geq C$,
\begin{equation}\label{eq:bound norm op A}
\bbP\left(\left\Vert p_n A\right\Vert > \lambda \sqrt{n} \right)\leq C e^{-c\lambda n}.
\end{equation}
Recalling the previous inequality, this means that 
\begin{equation*}
\bbP\left(S_{ n}^{\ijjk}> \frac{\lambda}{np_n^2}\right)\leq C e^{-c\lambda n},
\end{equation*} 
which is summable. The bound \eqref{eq:bound Snlijik} can be treated in the same way. Let us fix $l\in \left\lbrace1, 2\right\rbrace$, $\cT=\lijik$, and define moreover $\hat R$ to the the diagonal matrix with diagonal values given by $\left(\hat \xi_{li}^{(n)}  r_i\right)_{1\leq i\leq n}$. We then have
\begin{equation*}
\sum_{ i,j,k} \hat{ \xi}^{(n)}_{li}\hat{\xi}_{ij}^{(n)} \hat{ \xi}_{ ik}^{(n)} r_{ i}s_{ j} t_{ k} = s A^t \hat R A t\leq \frac{n}{p_n} \Vert A\Vert^2,
\end{equation*}
where we used the rough bound $\Vert \hat R\Vert\leq \frac{1}{p_n}$. We can then proceed as above. The proof for $\cT=\lijk$ is similar.
\end{proof}

We will consider in many places of the paper various (possibly weighted) empirical means involving the centered variables $ \hat{ \xi}_{ij}^{ (n)}$. We refer to Appendix~\ref{sec:App concentration} where the definitions and the corresponding asymptotics for these quantities have been gathered.
\section{Proofs concerning the global fluctuations}
\label{sec:global fluctuations}

In this section, we will refer to two linear forms and their continuity properties.
\begin{lemma}
\label{lem:linear_forms}
Let $ \theta, \theta^{ \prime}\in \mathbb{ T}$ be fixed. The following linear forms
\begin{align*}
D_{ \theta}(f)&:= f^{ \prime}(\theta),\\
\Delta_{ \theta, \theta^{ \prime}}(f)&:= f(\theta)- f( \theta^{ \prime}),
\end{align*}
are continuous on $H^{ m} \left( \mathbb{ T}\right)$ for any $m> \frac{ 1}{ 2}$: there exists a constant $c_{ m}>0$ such that 
\begin{align*}
\sup_{ \theta\in \mathbb{ T}} \left\Vert D_{ \theta} \right\Vert_{ -(m+1)}&< c_{ m},\\
\sup_{ \theta, \theta^{ \prime}\in \mathbb{ T}} \left\Vert \Delta_{ \theta, \theta^{ \prime}} \right\Vert_{ -(m+1)}&< c_{ m} \left\vert \theta- \theta^{ \prime} \right\vert.
\end{align*}
\end{lemma}
\begin{proof}
We have for all regular $f$ and $ \theta, \theta^{ \prime}\in \mathbb{ T}$,
\begin{equation*}
\left\vert D_{ \theta}(f) \right\vert= \left\vert f^{ \prime}(\theta) \right\vert \leq \left\Vert f \right\Vert_{ C^{ 1}} \leq C_{ m} \left\Vert f \right\Vert_{ H^{ m+1}},
\end{equation*}
where the last inequality is due to Sobolev embedding, see, e.g., \cite[Theorem 5.4, Case C]{adams_fournier_2003}. Similarly, for $ \Delta_{ \theta, \theta^{ \prime}}$ we have
\begin{equation*}
\left\vert \Delta_{ \theta, \theta^{ \prime}} (f)\right\vert \leq \left\vert \theta- \theta^{ \prime} \right\vert \left\Vert f \right\Vert_{ C^{ 1}} \leq C_m \left\vert \theta- \theta^{ \prime} \right\vert \left\Vert f \right\Vert_{ H^{ m+1}} \leq 2\pi C_m \left\Vert f \right\Vert_{ H^{ m+1}}.
\end{equation*}
\qedhere
\end{proof}

\subsection{Regularity and semimartingale representations}

We present the stochastic differential equations satisfied by $\mu^n$, $\eta^n$ and $\hat{\eta}^n$, and aim in particular at proving Proposition~\ref{prop:charac eta tilde eta}. We define $S:= (S_t)_{t\geq 0}$ as the analytic semigroup associated to the Laplacian operator. We present here a well-known argument (\cite{bechtold_coppini_2021,flandoli_oliveira_2020}) concerning the regularity of $(S_{ t})$, that we will employ at multiple steps.

\begin{lemma}
\label{lem:regularity_in_H_r}
Let $r>d/2$ and $k\geq r$. Then, there exists $\varepsilon>0$ such that the following conditions hold:
\begin{enumerate}
\item $\cP(\bbT^d) \subset H^{-r+\varepsilon}(\bbT^d)$ continuously and there exists $C$, only depending on $\varepsilon$, such that for every $\mu \in \cP(\bbT^d)$
\begin{equation*}
\sup_h \frac{ \< \mu, h>_{-r+\varepsilon, r-\varepsilon}}{ \norm{h}_{r-\varepsilon}} \leq C,
\end{equation*}
\item For every $h \in H^r(\bbT^d)$, it holds that
\begin{equation*}
\norm{S_{t-s} h}_{k} \leq C\left(1+\frac{1}{(t-s)^\frac{k-r}{2}}\right)\norm{h}_r.
\end{equation*}
\end{enumerate}
\end{lemma}
\begin{proof}
The first statement is a consequence of Sobolev's inequalities. The second statement comes from the regularity of the semigroup (see for example \cite{henry}).
\end{proof}
We start with the semimartingale representation of $\mu^n$.
\begin{lemma}
\label{lem:mu^n_t}
The empirical measure $\mu^n$ satisfies the following weak semimartingale representation: for any $f \in \cC^2(\bbT)$ and $t\in[0,T]$
\begin{equation}
\label{eq:mu^n_t}
\begin{split}
\< \mu^n_t, f> =& \< \mu^n_0, f> + \int_0^t \< \mu^n_s, \frac 12 \partial_\theta^2 f + (\mu^n_s * \Gamma)  \partial_\theta f > \dd s\\
&+ \< \hat{\nu}^n_t(\dd \theta_1, \dd \theta_2), \Gamma\left(\theta_1, \theta_2\right)\partial_\theta f\left(\theta_1\right) > + \< M^n_t, f >,
\end{split}
\end{equation}
where $ \hat{ \nu}_{ t}^{ n}$ is given by \eqref{def:centered_graph_emp}. The noise term $M_{ t}^{ n}$ in \eqref{eq:mu^n_t} is defined by
\begin{equation}
\label{eq:M^n}
\< M^n_t, f > = \frac{1}{n}\sum_{i=1}^n \int_0^t \partial_\theta f\left(\theta^{i,n}_s\right)\dd B^i_s.
\end{equation}
Let $r>1/2$. Then, $\mu^n$ satisfies the following weak-mild equation: for any $h \in H^r$ and $t \in [0,T]$,
\begin{equation*}
\begin{split}
\< \mu^n_t, h>_{-r,r} =& \< \mu^n_0, S_t h >_{-r,r} + \int_0^t \< \mu^n_s, (\mu^n_s * \Gamma)  (\partial_\theta S_{t-s} h) >_{-r,r} \dd s\\
&+ \int_0^t \frac{1}{n^2} \sum_{i,j=1}^n \hat{ \xi}_{ij}^{ (n)} \< \delta_{\theta^{i,n}_s}, (\Gamma * \delta_{\theta^{j,n}_s}) \partial_\theta S_{t-s}h >_{-r,r} \dd s + m^n_t(h),
\end{split}
\end{equation*}
where
\begin{equation}
\label{eq:mn}
m^n_t(h) = \frac{1}{n}\sum_{i=1}^n \int_0^t \partial_\theta S_{t-s} h \left(\theta^{i,n}_s\right)\dd B^i_s.
\end{equation}
\end{lemma}

\begin{proof}
The proof is based on the Itô formula. Consider a regular function $f=f(\theta)$, then
\begin{equation*}
\dd f (\theta^{i,n}_t) = \partial_\theta f(\theta^{i,n}_t)\dd \theta^{i,n}_t + \frac 12 \partial_\theta^2 f(\theta^{i,n}_t) \dd \llbracket \theta^{i,n} \rrbracket_t.
\end{equation*}
Observe that $\dd \llbracket \theta^{i,n} \rrbracket_t = \dd t$ for every $i=1,\dots,n$. By summing over $i$ and integrating with respect to the time the previous expression, one obtains that, writing $ \frac{ \xi_{ij}^{ (n)}}{ p_{ n}} = \hat{ \xi}_{ij}^{ (n)}+1$,
\begin{equation*}
\begin{split}
\< \mu^n_t, f> &= \< \mu^n_0, f> + \int_0^t \< \mu^n_s, \frac 12 \partial_\theta^2 f + (\mu^n_s * \Gamma)  \partial_\theta f > \dd s + \< M^n_s, f > \\
+& \int_0^t \frac{1}{n^2}\sum_{i,j=1}^n \hat{ \xi}_{ij}^{ (n)} \Gamma\left(\theta^{i,n}_s, \theta^{j,n}_s\right)\partial_\theta f\left(\theta^{i,n}_s\right) \dd s.
\end{split}
\end{equation*}
Finally, observe that the last expression is equivalent to \eqref{eq:mu^n_t} since
\begin{equation}
\label{eq:graph_term_hat_eta_n}
\int_0^t \frac{1}{n^2}\sum_{i,j=1}^n \hat{\xi}^{(n)}_{ij} \Gamma\left(\theta^{i,n}_s, \theta^{j,n}_s\right)\partial_\theta f\left(\theta^{i,n}_s\right) \dd s = \< \hat{\nu}^n_t(\dd \theta_1, \dd \theta_2), \Gamma\left(\theta_1, \theta_2\right)\partial_\theta f\left(\theta_1\right) >.
\end{equation}
The second part of the proposition follows from Ito formula on the test function $f(\theta, s)=(S_{t-s}h)(\theta)$ and by using the fact that $\partial_s S_{t-s} h = - \frac 12 \partial_\theta^2 S_{t-s} h$. Because of Sobolev inequalities, recall \eqref{eq:sobolev_ineq}, $\cP(\bbT) \subset H^{-r}$ continuously for any $r>1/2$: any bracket $\<\cdot, \cdot>_{-r,r}$ is indeed the action of an element of $H^{-r}$ against an element of $H^r$. Well-posedness of $m^n$ given in \eqref{eq:mn} as an element of $H^{-r}$ have already been addressed in the literature, see, e.g., \cite{Bertini:2013aa, bechtold_coppini_2021, Coppini2019}, we will give another proof in Proposition \ref{prop:eta^n_tight}.
\end{proof}

Using Lemma~\ref{lem:mu^n_t} as well as the limit PDE \eqref{eq:limit PDE}, we can write down an equation for $\eta^n$ as defined in \eqref{def:fluctuation_process}. 
\begin{lemma}
\label{lem:fluctuation_process_semimart}
Suppose Assumption~\ref{ass:Gamma}. The process $ (\eta^{ n})$ given in \eqref{def:fluctuation_process}  belongs $ \mathbb{ P}\otimes \mathbf{ P}$-a.s. to $ \mathcal{ C} \left([0, T], H^{ -r} \left( \mathbb{ T}\right)\right)$ for any $r> \frac{ 3}{ 2}$. Moreover, $\eta^n$ satisfies the following weak semimartingale representation: for every regular test function $f$ (recall the definition of $\mathcal{ L}^{(1)}$ given in \eqref{eq:L1}),
\begin{equation}
\label{eq:eta_weak_semimart}
\left\langle \eta_{ t}^{ n}\, ,\, f\right\rangle = \left\langle \eta_{ 0}^{ n}\, ,\, f\right\rangle + \int_{ 0}^{t} \left\langle \eta_{ s}^{ n}\, ,\,  \mathcal{ L}_{ \mu^n_s}^{(1)}(f)\right\rangle {\rm d}s + \int_{ 0}^{t} \left\langle \hat{ \eta}_{ s}^{ n}({\rm d}\theta, {\rm d}\theta^{ \prime})\, ,\, \Gamma \left(\theta, \theta^{ \prime}\right)\partial_{ \theta}f(\theta)\right\rangle {\rm d}s + W_{ t}^{ n}(f),
\end{equation} 
where $ \hat{ \eta}_{ t}^{ n}$ is given in \eqref{eq:hat_etan} and where $W^n_t$ is defined by
\begin{equation}
\label{eq:Wn}
\< W^n_t, f > := \frac{1}{\sqrt n}\sum_{i=1}^n \int_0^t \partial_\theta f\left(\theta^{i,n}_s\right)\dd B^i_s.
\end{equation}
Let $r>1/2$. The process $ \eta^{ n}$ satisfies the following weak-mild equation: for every $h \in H^r(\bbT)$ and $t \in [0,T]$:
\begin{equation}
\label{eq:eta_mild_semimart}
\begin{split}
\left\langle \eta_{ t}^{ n}\, ,\, h\right\rangle_{-r,r} =& \left\langle \eta_{ 0}^{ n}\, ,\, S_{t}h\right\rangle_{-r,r} + \int_{ 0}^{t} \left\langle \eta_{ s}^{ n}\, ,\, (\Gamma*\mu^n_s) \partial_\theta S_{t-s} h  \right\rangle_{-r,r} {\rm d}s \\
&+ \int_{ 0}^{t} \left\langle \mu_{ s}\, ,\, (\Gamma*\eta^n_s) \partial_\theta S_{t-s} h  \right\rangle_{-r,r} {\rm d}s + w_{ t}^{ n}(h)\\
&+ \int_0^t \frac{1}{n^{3/2}} \sum_{i,j=1}^n \hat{ \xi}_{ij}^{ (n)} \< \delta_{\theta^{i,n}_s}, (\Gamma * \delta_{\theta^{j,n}_s}) \partial_\theta S_{t-s}h >_{-r,r} \dd s,
\end{split}
\end{equation}
where
\begin{equation*}
w^n_t(h) = \frac{1}{\sqrt n}\sum_{i=1}^n \int_0^t (\partial_\theta S_{t-s} h) \left(\theta^{i,n}_s\right)\dd B^i_s.
\end{equation*}
\end{lemma}

\begin{proof}
We first address the continuity of the processes $ \eta^{ n}$: for $ \theta, \theta^{ \prime}\in \mathbb{ T}$ and $ \varphi$ a regular test function
\begin{align*}
\left\vert \left\langle \delta_{ \theta} - \delta_{ \theta^{ \prime}}\, ,\, \varphi\right\rangle \right\vert &= \left\vert \varphi(\theta) - \varphi(\theta^{ \prime}) \right\vert \leq \left\Vert \varphi^{ \prime} \right\Vert_{ \infty} \left\vert \theta- \theta^{ \prime} \right\vert \leq \left\Vert \varphi \right\Vert_{ C^{ 1}} \left\vert \theta- \theta^{ \prime} \right\vert\leq C\left\Vert \varphi \right\Vert_{ H^{ r}} \left\vert \theta- \theta^{ \prime} \right\vert,
\end{align*}
by Sobolev embedding, for any $r> \frac{ 3}{ 2}$. This means that \[ \left\Vert \delta_{ \theta}- \delta_{ \theta^{ \prime}} \right\Vert_{ H^{ -(r+1)}}\leq C \left\vert \theta- \theta^{ \prime} \right\vert.\]
Now, a.s. for all $n\geq1$, $ t \mapsto \left(\theta_{ t}^{ 1, n}, \ldots, \theta_{ t}^{ n,n}\right)$ is continuous from $[0, T]$ to $ \mathbb{ T}^{ n}$. In particular, for all $i=1, \ldots, n$, for all $t$ and $t_{ m}\to t$, we have $ \left\Vert \delta_{ \theta_{ t_{ m}}^{ i , n}}- \delta_{ \theta_{ t}^{ i, n}} \right\Vert_{ H^{ -r}}\leq C \left\vert \theta_{ t_{ m}}^{ i, n}- \theta_{ t}^{ i, n}\right\vert$ which goes to $0$ as $t_{ m}\to t$. As a conclusion: $ t \mapsto \mu^n_t$ is a.s. continuous from $[0, T]$ to $H^{ -r}$ for $r> \frac{ 3}{ 2}$. Classical results \cite{Fernandez1997, Sznitman1991} assure that $\mu$ solves the Fokker-Planck equation \eqref{eq:limit PDE}, i.e.,
\begin{equation*}
\partial_t \mu_t(\theta) = \frac12 \partial^2_{\theta}\mu_t(\theta)-\partial_\theta\left[\mu_t(\theta) (\Gamma*\mu_t)(\theta)\right], \quad 0 < t \leq T,
\end{equation*}
and that it is an element of $ \mathcal{ C}([0, T], H^{ -r})$ (see for example \cite{sell2013dynamics}). As a conclusion, $ \eta^{ n}$ is a.s. continuous in $H^{ -r}(\mathbb{ T})$. Equation \eqref{eq:eta_weak_semimart} (respectively \eqref{eq:eta_mild_semimart}) is derived by subtracting the representations (resp. the weak-mild formulations) satisfied by $\mu^n$ and $\mu$, and by multiplying everything by $\sqrt{n}$. Observe that $W^n$ in \eqref{eq:Wn} is nothing but $M^n$ in \eqref{eq:M^n} multiplied by $\sqrt{n}$, and similarly for $w^n$ and $m^n$. Well-posedness of $w^n$ as an element of $H^{-r}$ is postponed to Proposition \ref{prop:eta^n_tight}.
\end{proof}
Recall the definition of the noise term $W_{ t}^{ n}$ in \eqref{eq:Wn}. 
\begin{lemma}
\label{lem:Mn_prop}
Under Assumption~\ref{ass:Gamma}, $ \mathbb{ P}$-a.s., for any $n \geq 1$, the process $ (W_{ t}^{ n})_{t \in [0,T]}$ is a martingale in $ \cC \left([0, T], H^{ -r}\left(\mathbb{ T}\right)\right)$, for  $r> \frac{ 3}{ 2}$. The corresponding Doob-Meyer process $\llangle W^{ n} \rrangle$ takes values in $ \mathcal{ L} \left(H^{ r}, H^{ -r}\right)$ and is given by \eqref{eq:DM_Mn}. Moreover, $ \left(W^{ n}\right)_{n\geq 1}$ satisfies 
\begin{equation}
\label{eq:control_supMn}
\sup_{ n} \mathbf{ E} \left( \sup_{ t\in [0, T]} \left\Vert W_{ t}^{ n}\right\Vert_{-r}^{ 2}\right)< +\infty.
\end{equation}
\end{lemma}
\begin{proof}
This proofs follows closely the arguments given in \cite[Prop.~4.1]{Fernandez1997}. Let $ \left(\varphi_{ p}\right)_{ p\geq1}$ be a orthonormal system in $H^{ r}$. Let us first prove that 
\begin{equation}
\label{eq:bound_Mn}
\sup_{ n} \sum_{ p\geq 1} \mathbf{ E} \left( \sup_{ t\in [0, T]} \left\vert W_{ t}^{ n} \left(\varphi_{ p}\right) \right\vert^{ 2}\right)= \sup_{ n} \mathbf{ E} \left( \sum_{ p\geq 1}  \sup_{ t\in [0, T]} \left\vert W_{ t}^{ n} \left(\varphi_{ p}\right) \right\vert^{ 2}\right)< + \infty.
\end{equation} By Doob's inequality,
\begin{align*}
\sum_{ p\geq 1} \mathbf{ E} \left( \sup_{ t\in [0, T]} \left\vert W_{ t}^{ n} \left(\varphi_{ p}\right) \right\vert^{ 2}\right) &\leq C \sum_{ p\geq 1} \mathbf{ E} \left(\left\vert W_{ T}^{ n} \left(\varphi_{ p}\right) \right\vert^{ 2}\right)\\
&= C \sum_{ p\geq 1} \mathbf{ E} \left(\frac{ 1}{ n} \sum_{ i=1}^{ n} \int_{ 0}^{T} \left\vert \partial_{ \theta} \varphi_{ p}(\theta_{ s}^{ i, n} )\right\vert^{ 2} {\rm d}s\right),\\
&=C \frac{ 1}{ n} \sum_{ i=1}^{ n} \mathbf{ E} \left( \int_{ 0}^{T}  \left\Vert D_{ \theta_{ s}^{ i, n}} \right\Vert_{ -r}^{ 2} {\rm d}s\right)\leq Cc_{ r-1}^{ 2} T,
\end{align*}
for the constant $c_{ r-1}>0$ given in Lemma~\ref{lem:linear_forms}, which gives \eqref{eq:bound_Mn}.

We now prove that the trajectories of $W^{ n}$ are almost surely continuous in $H^{ -r}$. By \eqref{eq:bound_Mn}, the series $\sum_{ p\geq 1}  \sup_{ t\in [0, T]} \left\vert W_{ t}^{ n} \left(\varphi_{ p}\right) \right\vert^{ 2}$ is a.s. convergent, and hence, for fixed $n$ and all $ \varepsilon>0$, there exists $p_{ 0}\geq1$ sufficiently large so that $\sum_{ p>p_{ 0}}  \sup_{ t\in [0, T]} \left\vert W_{ t}^{ n} \left(\varphi_{ p}\right) \right\vert^{ 2}< \frac{ \varepsilon}{ 6}$. Let $(t_{ m})_{ m\geq1}$ be a sequence in $[0, T]$ such that $t_{ m} \xrightarrow[ m\to\infty]{}t$, then
\begin{align*}
\left\Vert W^{ n}_{ t_{ m}} - W^{ n}_{ t} \right\Vert_{ -r}^{ 2} &= \sum_{ p\geq1} \left\vert \left\langle W^{ n}_{ t_{ m}} - W^{ n}_{ t}\, ,\, \varphi_{ p}\right\rangle \right\vert^{ 2} \\
&\leq \sum_{ p=1}^{ p_{ 0}}  \left\vert \left\langle W^{ n}_{ t_{ m}} - W^{ n}_{ t}\, ,\, \varphi_{ p}\right\rangle \right\vert^{ 2} + 2 \sum_{ p> p_{ 0}} \left(  \left\vert \left\langle W^{ n}_{ t_{ m}}\, ,\, \varphi_{ p}\right\rangle \right\vert^{ 2} +  \left\vert \left\langle W^{ n}_{ t}\, ,\, \varphi_{ p}\right\rangle \right\vert^{ 2}\right).
\end{align*}
The last summand is smaller than $ \frac{ 4 \varepsilon}{ 6}$ and the first one can be made smaller than $ \frac{ \varepsilon}{ 3}$ by the a.s. continuity of $t \mapsto \left\langle W^{ n}_{ t}\, ,\, \varphi\right\rangle$ for all $ \varphi$. Hence, a.s. $W^{ n}\in \cC \left([0, T], H^{ -r}\right)$. The expression of $ \llangle W^{ n} \rrangle$ in \eqref{eq:DM_Mn} follows directly from Ito's isometry and \eqref{eq:control_supMn} is a direct consequence of the stronger statement \eqref{eq:bound_Mn}.
\end{proof}

The supplementary term in the drift in the semimartingale representation \eqref{eq:eta_weak_semimart} can be expressed in terms of the higher order fluctuation process $ \hat{ \eta}^{ n}$ defined in \eqref{eq:hat_etan}. The idea is to proceed further and write the semimartingale decomposition of $ \hat{ \eta}^{ n}$. For $1\leq i \leq n$, define the measures
\begin{equation}
\label{def:hat_mu_i_n}
\hat{\mu}^{n,i}_t = \frac 1n \sum_{j=1}^n \hat{\xi}_{ij}^{ (n)} \delta_{\theta^{j,n}_t} \in H^{-r}, \qquad r > \frac 12.
\end{equation}
\begin{lemma}
\label{lem:centered_emp_mes_spde}
Suppose Assumption~\ref{ass:Gamma}. The process $ (\hat \eta^n)$ defined in \eqref{eq:hat_etan} belongs $ \mathbb{ P}\otimes \mathbf{ P}$-a.s. to $ \mathcal{ C} \left([0, T], H^{ -r^{ \prime}}\left(\mathbb{ T}^{ 2}\right)\right)$ for any $r^{ \prime}> \frac{ 5}{ 2}$ and satisfies the following weak semimartingale decomposition: for all regular test function $(\theta_{ 1}, \theta_{ 2}) \mapsto g(\theta_{ 1}, \theta_{ 2})$ (recall the definition of $\mathcal{ L}^{(2)}$ and $C^n$ given respectively in \eqref{eq:L2} and \eqref{eq:def Cn}),
\begin{equation}
\label{eq:hat_eta_semimart}
\left\langle \hat{ \eta}_{ t}^{ n}\, ,\, g\right\rangle= \left\langle \hat{ \eta}_{ 0}^{ n}\, ,\, g\right\rangle + \int_{ 0}^{t} \left\langle \hat{ \eta}_{ s}^{ n}\, ,\, \mathcal{ L}_{ \mu^n_s}^{(2)}(g)\right\rangle {\rm d}s + \sqrt{ n} C_{t}^{ n}(g) + \hat{ W}_{ t}^{ n} (g),
\end{equation}
\begin{equation}\label{eq:hatMn}
\hat{W}^n_t(g) := \frac 1{n^{ 3/2}} \sum_{i,j=1}^n \hat{\xi}^{(n)}_{ij} \int_0^t \nabla g(\theta^{i,n}_s, \theta^{j,n}_s) \cdot (\dd B^i_s, \dd B^j_s).
\end{equation}
Let $r>3/2$. The process $\hat{\eta}^{ n}$ satisfies the following weak-mild equation: for every $h \in H^r(\bbT^2)$ and $t \in [0,T]$:
\begin{equation}
\label{eq:hat_eta_mild_semimart}
\begin{split}
\left\langle \hat{\eta}_{ t}^{ n}\, ,\, h\right\rangle_{-r,r} =& \left\langle \hat{\eta}_{ 0}^{ n}\, ,\, S_{t}h\right\rangle_{-r,r}  + \int_{ 0}^{t} \left\langle \hat{\eta}^n_{ s}\, ,\, \Lambda^n_s \cdot [ \nabla S_{t-s} h] \right\rangle_{-r,r} {\rm d}s + \hat{w}_{ t}^{ n}(h) \\
&+ \int_0^t \frac{1}{n^{3/2}} \sum_{i,j=1}^n \hat{ \xi}_{ij}^{ (n)} \< \delta_{\theta^{i,n}_s}  \otimes \ \delta_{ \theta_{s}^{j, n}}, \hat{\Lambda}^n_{s, ij} \cdot \nabla S_{t-s}h >_{-r,r} \dd s,
\end{split}
\end{equation}
with, recall definition \eqref{def:hat_mu_i_n}, $ \Lambda^{ n}_{ s}$ and $\hat{\Lambda}^n_{s, ij}$ are respectively given by
\begin{align}
\label{eq:Lambdasn}
\Lambda^n_s(\theta_1, \theta_2) &:= \left(\left\langle \mu_{ s}^{ n}({\rm d}\theta^{ \prime})\, ,\, \Gamma (\theta_{ 1}, \theta^{ \prime})\right\rangle,  \left\langle \mu_{ s}^{ n}({\rm d}\theta^{ \prime})\, ,\, \Gamma (\theta_{ 2}, \theta^{ \prime})\right\rangle \right) \in \R^2,\\
\hat{\Lambda}^n_{s, ij} (\theta_1, \theta_2) &:= \left(\left\langle \hat{\mu}_{ s}^{n,i}({\rm d}\theta^{ \prime})\, ,\, \Gamma (\theta_{ 1}, \theta^{ \prime})\right\rangle,  \left\langle \hat{\mu}_{ s}^{n,j}({\rm d}\theta^{ \prime})\, ,\, \Gamma (\theta_{ 2}, \theta^{ \prime})\right\rangle \right) \in \R^2,
\end{align}
and where
\begin{equation}
\label{eq:hatwn}
\hat{w}^n_t(h) := \frac 1{n^{ 3/2}} \sum_{i,j=1}^n \hat{\xi}^{(n)}_{ij} \int_0^t \nabla S_{t-s} h(\theta^{i,n}_s, \theta^{j,n}_s) \cdot (\dd B^i_s, \dd B^j_s).
\end{equation}
\end{lemma}

\begin{proof}
The proof of continuity of the trajectories of $ \hat{ \eta}^{ n}$ follows the same argument as in the proof of Lemma~\ref{lem:fluctuation_process_semimart}, as for any $n\geq1$, $ \hat{ \eta}^{ n}$ is a finite weighted sum of Dirac measures $ \delta_{ (\theta_{ t}^{ i,n}, \theta_{ t}^{ j, n})}$. One has for any $(\theta_{ 1}, \theta_{ 2}), (\theta_{ 1}^{ \prime}, \theta_{ 2}^{ \prime})\in \mathbb{ T}^{ 2}$, any regular test function $ \psi$ on $ \mathbb{ T}^{ 2}$, for any $r> \frac{ 5}{ 2}$, 
\begin{multline*}
\left\vert \left\langle \delta_{ (\theta_{ 1}, \theta_{ 2})} - \delta_{ (\theta_{ 1}^{ \prime}, \theta_{ 2}^{\prime})}\, ,\, \psi\right\rangle \right\vert= \left\vert \psi(\theta_{ 1}, \theta_{ 2}) - \psi \left(\theta_{ 1}^{ \prime}, \theta_{ 2}^{ \prime}\right)\right\vert \leq \left\Vert \psi \right\Vert_{ \mathcal{ C}^{ 1}} \left\lbrace\left\vert \theta_{ 1}- \theta_{ 1}^{ \prime} \right\vert + \left\vert \theta_{ 2}- \theta_{ 2}^{ \prime} \right\vert\right\rbrace\\ \leq C \left\Vert \psi \right\Vert_{H^{ r}(\mathbb{ T}^{ 2})} \left\lbrace\left\vert \theta_{ 1}- \theta_{ 1}^{ \prime} \right\vert + \left\vert \theta_{ 2}- \theta_{ 2}^{ \prime} \right\vert\right\rbrace,
\end{multline*}
which proves the desired continuity, with the same arguments as for Lemma~\ref{lem:fluctuation_process_semimart}. The semimartingale representation \eqref{eq:hat_eta_semimart} is derived as the one of Lemma~\ref{lem:mu^n_t} but where Ito formula is applied to test functions of two variables, i.e., to $g(\theta^{i,n}_t, \theta^{j,n}_t)$. The second part of the proposition follows again from Ito formula, but with test functions $g(\theta, \theta', s)=(S_{t-s}h)(\theta, \theta')$, and by using the fact that $\partial_s S_{t-s} h = - \frac 12 \nabla S_{t-s} h$.
The choice of $r>3/2$ and Sobolev inequalities, recall \eqref{eq:sobolev_ineq}, assure that any bracket $\<\cdot, \cdot>_{-r,r}$ makes sense as the action of an element of $H^{-r}(\bbT^2)$ against an element of $H^r(\bbT^2)$. Well-posedness of $m^n$ given in \eqref{eq:mn} is given in Proposition \ref{prop:hat_eta_n_tight}.
\end{proof}
\begin{proposition}
\label{prop:Wn_prop}
For any $r>3$, the process $ \left( \hat{ W}_{ t}^{ n}\right)$ is a martingale in $ \cC \left([0, T], H^{ -r}(\mathbb{ T}^{ 2})\right)$, whose Doob-Meyer process $ \llangle \hat{ W}^{ n} \rrangle_{ t}$ taking values in $ \mathcal{ L} \left(H^{ r}(\mathbb{ T}^{ 2}), H^{ -r}(\mathbb{ T}^{ 2})\right)$ is given for every $ \varphi, \psi\in H^{ r}(\mathbb{ T}^{ 2})$ by
\begin{align*}
\label{eq:DM_Wn}
\llangle \hat{ W}^{ n}\rrangle_{ t}\cdot \varphi(\psi)&= \left\langle \hat{ W}^{ n}(\varphi)\, ,\, \hat{ W}^{ n}(\psi)\right\rangle_{ t}\\
&= \frac{ 1}{ n^{ 3}} \sum_{ i,j,k} \hat{ \xi}_{ij}^{ (n)} \hat{ \xi}_{ik}^{ (n)} \int_{ 0}^{ t} \partial_{ \theta_{ 1}} \varphi \left(\theta_{ s}^{i,n}, \theta_{ s}^{ j,n}\right) \partial_{ \theta_{ 1}} \psi \left(\theta_{ s}^{ i,n}, \theta_{ s}^{ k, n}\right) {\rm d}s\\
&+ \frac{ 1}{ n^{ 3}} \sum_{ i,j,k} \hat{ \xi}_{ij}^{ (n)} \hat{ \xi}_{jk}^{ (n)} \int_{ 0}^{ t} \partial_{ \theta_{ 2}} \varphi \left(\theta_{ s}^{i,n}, \theta_{ s}^{ j,n}\right) \partial_{ \theta_{ 1}} \psi \left(\theta_{ s}^{ j,n}, \theta_{ s}^{ k, n}\right) {\rm d}s\\
&+ \frac{ 1}{ n^{ 3}} \sum_{ i,j,k} \hat{ \xi}_{ij}^{ (n)} \hat{ \xi}_{ki}^{ (n)} \int_{ 0}^{ t} \partial_{ \theta_{ 1}} \varphi \left(\theta_{ s}^{i,n}, \theta_{ s}^{ j,n}\right) \partial_{ \theta_{ 2}} \psi \left(\theta_{ s}^{ k,n}, \theta_{ s}^{ i, n}\right) {\rm d}s\\
&+ \frac{ 1}{ n^{ 3}} \sum_{ i,j,k} \hat{ \xi}_{ij}^{ (n)} \hat{ \xi}_{k,j}^{ (n)} \int_{ 0}^{ t} \partial_{ \theta_{ 2}} \varphi \left(\theta_{ s}^{i,n}, \theta_{ s}^{ j,n}\right) \partial_{ \theta_{ 2}} \psi \left(\theta_{ s}^{ k,n}, \theta_{ s}^{ j, n}\right) {\rm d}s.
\end{align*}
Moreover, $ \left( \hat{ W}^{ n}\right)$ satisfies $\bbP$-almost surely
\begin{equation}
\label{eq:control_supWn}
 \mathbf{ E} \left( \sup_{ t\in [0, T]} \left\Vert \hat{ W}_{ t}^{ n}\right\Vert_{-r}^{ 2}\right) \xrightarrow[ n\to\infty]{}0.
\end{equation}
\end{proposition}
\begin{proof}
Once again, we follow a very similar approach as in \cite{Fernandez1997}. Recall from \eqref{eq:hatMn}, that we may write 
\begin{align*}
\hat{ W}_{ t}^{ n}(f) &= \frac 1{n^{ 3/2}} \sum_{i,j=1}^n \hat{\xi}^{(n)}_{ij} \int_0^t \partial_{ \theta_{ 1}} f(\theta^{i,n}_s, \theta^{j,n}_s) \dd B^i_s + \frac 1{n^{ 3/2}} \sum_{i,j=1}^n \hat{\xi}^{(n)}_{ij} \int_0^t \partial_{ \theta_{ 2}} f(\theta^{i,n}_s, \theta^{j,n}_s) \dd B^j_s,\\ &:= \hat{ W}_{ t}^{ n, 1}(f) + \hat{ W}_{ t}^{ n, 2}(f),
\end{align*}
so that it is sufficient to prove \eqref{eq:control_supWn} for $ \hat{ W}_{ t}^{ n, 1}$ and $ \hat{ W}_{ t}^{ n, 2}$ separately. Let $ \left(\varphi_{ p}\right)_{ p\geq1}$ be a orthonormal system in $H^{ r}\left(\mathbb{ T}^{ 2}\right)$. We have that 
\begin{align*}
\mathbf{ E} \left( \sup_{ t\in [0, T]} \left\Vert \hat{ W}_{ t}^{ n, 1}\right\Vert_{-r}^{ 2}\right) &= \mathbf{ E} \left( \sup_{ t\in [0, T]} \sum_{ p\geq1}\left\vert \hat{ W}_{ t}^{ n, 1} \left(\varphi_{ p}\right)\right\vert^{ 2}\right) \leq \sum_{ p\geq1} \mathbf{ E} \left( \sup_{ t\in [0, T]} \left\vert \hat{ W}_{ t}^{ n, 1} \left(\varphi_{ p}\right)\right\vert^{ 2}\right)\\
&\leq C\sum_{ p\geq1} \mathbf{ E} \left( \left\langle \hat{ W}^{ n, 1} \left(\varphi_{ p}\right) \right\rangle_{ T} \right),
\end{align*}
by Doob's inequality. This last quantity is further bounded by
\begin{align*}
\mathbf{ E} \left( \sup_{ t\in [0, T]} \left\Vert \hat{ W}_{ t}^{ n, 1}\right\Vert_{-r}^{ 2}\right) &\leq  \int_{ 0}^{ T} \frac{ C}{ n^{ 3}} \bE \sum_{ i,j,k=1}^{ n} \hat{ \xi}_{ij}^{ (n)} \hat{ \xi}_{ik}^{ (n)}  \sum_{ p\geq1}\partial_{ \theta_{ 1}} \varphi_{ p} \left(\theta_{ s}^{i,n}, \theta_{ s}^{ j,n}\right) \partial_{ \theta_{ 1}} \bar\varphi_{ p} \left(\theta_{ s}^{ i,n}, \theta_{ s}^{ k, n}\right) {\rm d}s\\
&= C\int_{ 0}^{T} \bE \left[ c^{\ijik}_{ n} \left(\Phi, I_d,\theta_{ s}^{ n}\right) \right] {\rm d}s,
\end{align*}
with the notations of Lemma~\ref{lem:Grothendieck_fonct} (with $k=r$). Doing the same for $ \hat{ W}_{ t}^{ n, 2}$ requires to control now the term $ c^{\ikjk}_{ n} \left(\Phi, I_d,\theta_{ s}^{ n}\right)$, and hence we obtain directly \eqref{eq:control_supWn} from \eqref{eq:Grothendieck_fonct} and \eqref{eq:bound SnT}.
The continuity of the trajectories of $ \hat{ W}^{ n}$ follows from the very same argument previously used from Lemma~\ref{lem:Mn_prop}.
\end{proof}
We now proceed further with moment estimates. Note that from here, our proof differs significantly from the line of proof followed in \cite{Fernandez1997}. One could see Propositions~\ref{prop:eta^n_tight} and~\ref{prop:hat_eta_n_tight} as equivalents of \cite[Prop.~3.5]{Fernandez1997}, where a similar estimate on the fluctuation process is proven. Note however that the proof of \cite[Prop.~3.5]{Fernandez1997} uses in an essential way the exchangeability of the particles as well as the fact that the initial condition of \eqref{eq:wips} are i.i.d., which is no longer the case here. We circumvent this difficulty by taking advantage of the mild formulations \eqref{eq:eta_mild_semimart} and \eqref{eq:hat_eta_mild_semimart} and the regularising properties of the heat kernel. By this method (see \eqref{eq:graph_eta^n} and Remark~\ref{rem:St}), the control we have on the moments of $ \eta^{ n}$ and $ \hat{ \eta}^{ n}$ is weaker than in \cite[Prop.~3.5]{Fernandez1997}, in the sense that one needs to have $ \alpha\in[0, 1)$ in \eqref{eq:first_control_etan} and \eqref{eq:first_control_hat_etan}, whereas the same estimate is proven directly for $ \alpha=1$ in \cite[Prop.~3.5]{Fernandez1997}. Note finally that a stronger uniform in time control $ \bE [\sup_{t \in [0,T]}\norm{\eta^n_t}^{2}_{-r}]$ is proven in \cite[Prop.~4.3]{Fernandez1997}, but a careful reading shows that this estimate is only used in \cite{Fernandez1997} to prove the continuity of the fluctuation process, for which we have provided an alternative proof in Lemma~\ref{lem:fluctuation_process_semimart}. Recall that $ \mathbf{ E} \left[\cdot\right]$ stands for the expectation w.r.t. the noise only, the results below hold for a fixed realisation of the graph and initial condition.
\begin{proposition}
\label{prop:eta^n_tight}
Suppose that Assumption~\ref{ass:Gamma} hold and take $r> \frac 32$. The sequence of processes $\left(\eta^{ n}\right)_{n \geq 1}$ satisfies $\bbP$-a.s.
\begin{equation}
\label{eq:first_control_etan}
\sup_{t \in [0,T]} \bE [\norm{\eta^n_t}^{ 1+ \alpha}_{-r}] \leq C \left(1+\norm{\eta^n_0}^{ 1+ \alpha}_{-r} + \sup_{t \in [0,T]} \bE [\norm{\hat{\eta}^n_t}^{ 1+ \alpha}_{-r}] \right).
\end{equation}
\end{proposition}
\begin{proof}
Recall the weak-mild formulation \eqref{eq:eta_mild_semimart}. Let $g^n(h)$ be the term involving the graph, i.e.,
\begin{equation*}
\begin{split}
g^n_t(h) := \int_0^t \frac{1}{n^{3/2}} \sum_{i,j=1}^n \hat{ \xi}_{ij}^{ (n)} \< \delta_{\theta^{i,n}_s}, (\Gamma * \delta_{\theta^{j,n}_s}) \partial_\theta S_{t-s}h >_{-r,r} \dd s.
\end{split}
\end{equation*}
In \eqref{eq:eta_mild_semimart}, taking the supremum with respect to $h$ in $H^r$ such that $\norm{h}_r \leq  1$, one can write:
\begin{equation*}
\begin{split}
\norm{\eta^n_t}_{-r} \leq& \norm{\eta^n_0}_{-r} + \int_0^t \norm{\eta^n_s}_{-r} \sup_{\norm{h}_r \leq 1}\norm{(\Gamma*\mu^n_s) \partial_\theta S_{t-s} h}_r \dd s  \\
&+  \int_0^t \norm{\mu_t}_{-r} \sup_{\norm{h}_r \leq 1}\norm{(\Gamma*\eta^n_s) \partial_\theta S_{t-s} h}_r \dd s + \norm{w^n_t}_{-r} + \sup_{\norm{h}_r \leq 1} g^n_t(h).
\end{split}
\end{equation*}
By using the classical inequality $(\sum_{i=1}^m a_i)^{ 1+ \alpha} \leq m^{ \alpha} \sum_{i=1}^m a_i^{ 1+ \alpha}$, one obtains
\begin{equation}
\label{aux:eta_alpha}
\begin{split}
\frac{ 1}{ 5^{ \alpha}} \norm{\eta^n_t}^{ 1+ \alpha}_{-r} \leq& \norm{\eta^n_0}^{ 1+ \alpha}_{-r} + T^{ \alpha}\int_0^t \norm{\eta^n_s}^{ 1+ \alpha}_{-r}\sup_{\norm{h}_r \leq 1} \norm{(\Gamma*\mu^n_s) \partial_\theta S_{t-s} h}^{ 1+ \alpha}_r \dd s  \\
&+ T^{ \alpha} \int_0^t \norm{\mu_t}^{ 1+ \alpha}_{-r} \sup_{\norm{h}_r \leq 1}\norm{(\Gamma*\eta^n_s) \partial_\theta S_{t-s} h}^{ 1+ \alpha}_r \dd s + \norm{w^n_t}^{ 1+ \alpha}_{-r}\\
& + \left(\sup_{\norm{h}_r \leq 1} g^n_t(h)\right)^{ 1+ \alpha}.
\end{split}
\end{equation}
Observe that similar to the proof of \cite[Proposition 2.2]{bechtold_coppini_2021}, one has that
\begin{equation*}
\begin{split}
\norm{(\Gamma*\eta^n_s) \partial_\theta S_{t-s} h}_r  &\leq \norm{\partial_\theta S_{t-s} h}_{r} \norm{\Gamma*\eta^n_s}_{W^{r,\infty}} \\
&\leq C\left(1+\frac 1{ \sqrt{t-s}}\right) \norm{h}_r \norm{\eta^n_s}_{-r},
\end{split}
\end{equation*}
where we have used the fact that $\norm{\Gamma*\eta^n_s}_{W^{r,\infty}} \leq C'  \norm{\eta^n_s}_{-r}$ given Assumption \ref{ass:Gamma}.
Note that, as pointed out in Lemma~\ref{lem:regularity_in_H_r}, uniformly on $s, n$, $ \left\langle \mu_{ s}^{ n}\, ,\, h\right\rangle \leq \left\Vert h \right\Vert_{ \infty} \leq C \left\Vert h \right\Vert_{ r}$, since $ \mu_{ s}^{ n}$ is a probability measure. Hence, by Assumption~\ref{ass:Gamma}, there is a constant $C>0$ independent on $s, n$ such that
\begin{equation}
\label{aux:semigroup1+alpha}
\sup_{\norm{h}_r \leq 1}\norm{(\Gamma*\mu^n_s) \partial_\theta S_{t-s} h}^{ 1+ \alpha}_r \leq C\left(1+\frac{ 1}{ (t-s)^{ \frac{ 1+ \alpha}{ 2}}}\right).
\end{equation}
For the graph term, one could repeat the second part of the proof in \cite[Lemma 3.2]{Coppini2019}, where it is shown that there exists a positive constant $C$ (only depending on $\Gamma$) such that
\begin{equation*}
\sup_{\norm{h}_r \leq 1} g^n_t(h) \leq C \sqrt{t} \, \sqrt{n} \, S_{ n}^{\ij}.
\end{equation*}
In turn, this would mean that, using Proposition \ref{prop:concentration_SnT}, for any $t \in [0,T]$ and $\bbP$-a.s.
\begin{equation*}
\sup_{\norm{h}_r \leq 1} g^n_t(h) \leq C \frac {\sqrt{T}}{\sqrt{p_n}}. 
\end{equation*}
However this is not enough when one considers $p_n$ converging to zero. To tackle the interesting case $p_n \to 0$, we need to take advantage of the representation of $g^n_t(h)$ through $\hat{\eta}^n$, recall \eqref{eq:graph_term_hat_eta_n}. Observe that, using \eqref{aux:semigroup1+alpha}, we have
\begin{equation}
\label{eq:graph_eta^n}
\begin{split}
\left(\sup_{\norm{h}_r \leq 1}g^n_t(h)\right)^{1+ \alpha} &\leq T^{ \alpha}\sup_{\norm{h}_r \leq 1} \int_{ 0}^{t} \left\langle \hat{ \eta}_{ s}^{ n}({\rm d}\theta, {\rm d}\theta^{ \prime})\, ,\, \Gamma \left(\theta, \theta^{ \prime}\right) \partial_{ \theta}S_{t-s} h(\theta)\right\rangle_{-r, r}^{ 1+ \alpha} {\rm d}s\\
&\leq T^{ \alpha}\int_0^t \norm{\hat{\eta}^n_s}^{ 1+ \alpha}_{-r} \sup_{\norm{h}_r \leq 1}\norm{\Gamma \left(\theta, \theta^{ \prime}\right) \partial_{ \theta}S_{t-s} h(\theta)}^{ 1+ \alpha}_{r} \dd s\\
&\leq C_{ \Gamma, T} \int_0^t \left(1+ \frac 1{(t-s)^{ \frac{ 1+ \alpha}{ 2}}}\right) \norm{\hat{\eta}^n_s}^{ 1+ \alpha}_{-r} \dd s,
\end{split}
\end{equation}
that is, by Lemma~\ref{lem:fluctuation_process_semimart}, a finite quantity, since $r> \frac{ 3}{ 2}$. Taking the expectation with respect to the Brownian motions, we get
\begin{equation*}
\bE\left[\sup_{\norm{h}_r \leq 1} g^n_t(h)^{ 1+ \alpha}\right] \leq C_{T, \Gamma} \sup_{t \in [0,T]} \bE [\norm{\hat{\eta}^n_t}_{-r}^{ 1+ \alpha}].
\end{equation*}
Taking the expectation $\bE$ in the inequality \eqref{aux:eta_alpha} yields 
\begin{equation*}
\begin{split}
\bE [\norm{\eta^n_t}_{-r}^{ 1+ \alpha}] \leq& \norm{\eta^n_0}_{-r}^{ 1+ \alpha} + C \int_0^t \left(1+\frac {1}{(t-s)^{ \frac{ 1+ \alpha}{ 2}}} \right) \bE[ \norm{\eta^n_s}_{-r}^{ 1+ \alpha}] \dd s \\
&+ \bE[ \norm{w^n_t}_{-r}^{ 1+ \alpha}] +  C_{T, \Gamma} \sup_{t \in [0,T]} \bE [\norm{\hat{\eta}^n_t}^{ 1+ \alpha}_{-r}].
\end{split}
\end{equation*}
One can then apply a version of Gronwall-Henry’s inequality (see for example \cite{GPPP}, Lemma 5.2), and obtain that
\begin{equation*}
\bE [\norm{\eta^n_t}_{-r}^{ 1+ \alpha}] \leq C \left(\norm{\eta^n_0}_{-r}^{ 1+ \alpha} + \sup_{t \in [0,T]} \bE [\norm{\hat{\eta}^n_t}_{-r}^{ 1+ \alpha}] +  \bE[\norm{w^n_t}_{-r}^{ 1+ \alpha}] \right).
\end{equation*}
Concerning the noise term, we have, as $ \alpha< 1$, $\bE[\norm{w^n_t}_{-r}^{ 1+ \alpha}] \leq \bE[\norm{w^n_t}_{-r}^{2}]^{ \frac{ 1+ \alpha}{ 2}}$ and for $(\varphi_p)_{p\geq 0}$ an orthonormal system in $H^r(\bbT)$, 
\begin{equation*}
\begin{split}
\bE[\norm{w^n_t}_{-r}^{2}]&=\sum_{p\geq 1}\bE\left[\left|w^n_t(\varphi_p)\right|^{2}\right]\leq C\sum_{p\geq 1}\bE\left[\left\langle w^n(\varphi_p)\right\rangle_t\right]\\
&\leq \frac{C}{n}\sum_{p\geq 1}\sum_{i=1}^n \int_0^t \bE\left[\left\vert\partial_\theta S_{t-s} \varphi_p(\theta^{i,n}_s)\right\vert^2\right]\dd s =\frac{C}{n}\sum_{i=1}^n \int_0^t \bE\left[\left\Vert U_{ s,t,i}\right\Vert_{ -r}^2\right]\dd s,
\end{split}
\end{equation*}
where $U_{ s,t,i}:= f \mapsto \partial_\theta S_{t-s} f(\theta^{i,n}_s)$.
\begin{rem}
\label{rem:St}
We cannot use here the fact that $\norm{\partial_\theta S_{t-s}}_{-r} \leq C(1+1/\sqrt{t-s})$ because $(t-s)^{-1}$ is not integrable. We follow arguments similar to \cite[Lemma 11]{flandoli_oliveira_2020}: by exploiting Sobolev embeddings, fractional operators and the fact that there exists $\varepsilon > 0$ such that $r-\varepsilon > 1/2$, we can conclude that
\begin{equation*}
\begin{split}
\sup_{\norm{h}_r \leq 1} \left| (\partial_\theta S_{t-s} h) \left(\theta^{i,n}_s\right) \right|^2 &= \sup_{\norm{h}_r \leq 1} \left| \< \delta_{\theta^{i,n}_s}, (\partial_\theta S_{t-s} h) >_{-r+ \varepsilon, r - \varepsilon}\right|^2 \\
&\leq C \left(1+\frac{1}{(t-s)^{1-\varepsilon}}\right) \norm{\delta_{ \theta_{ s}^{ i, n}}}^2_{-r + \varepsilon}  \leq C' \left(1+\frac{1}{(t-s)^{1-\varepsilon}}\right),
\end{split}
\end{equation*}
where we rely on Lemma~\ref{lem:regularity_in_H_r}.
\end{rem}
Finally, gathering all these estimates, we deduce that there exists $C>0$, depending on $T$, such that
\begin{equation*}
\sup_{t \in [0,T]} \bE[\norm{w^n_t}_{-r}^{ 1+ \alpha}] \leq C.
\end{equation*}
Putting every estimate together, we end up with \eqref{eq:first_control_etan}.
\end{proof}
\begin{proposition}
\label{prop:hat_eta_n_tight}
Suppose that Assumption~\ref{ass:Gamma} holds and take $r > 2$. For $n$ large enough, the sequence of processes $\left(\hat{\eta}^{ n}\right)_{n \geq 1}$ satisfies $\bbP$-a.s.
\begin{equation}
\label{eq:first_control_hat_etan}
\sup_{t \in [0,T]} \bE [ \norm{\hat{\eta}^n_t}_{-r}^{ 1+ \alpha} ] \leq C \left(\norm{\hat{\eta}^n_0}_{-r}^{ 1+ \alpha} + \frac 1{ \left(\sqrt{n}p_n\right)^{ 1+ \alpha}} + \frac 1{ \left(\sqrt{ n}p_{ n}^{ 2}\right)^{ 1+ \alpha}} \right).
\end{equation}
\end{proposition}

\begin{proof}
From the mild formulation \eqref{eq:hat_eta_mild_semimart} of $\hat{\eta}^n$, one has that
\begin{equation}
\label{eq:decomp_hateta_alpha}
\begin{split}
\frac{ 1}{ 4^{ \alpha}}\norm{\hat{\eta}^n_t}_{-r}^{ 1+ \alpha} &\leq \norm{\hat{\eta}^n_0}_{-r}^{ 1+ \alpha} + T^{ \alpha}\int_0^t \norm{\hat{\eta}^n_s}_{-r}^{ 1+ \alpha} \sup_{\norm{h}_r \leq 1}\norm{\Lambda^n_s \cdot \nabla S_{t-s} h}_r^{ 1+ \alpha} \dd s + \norm{\hat{w}^n_t}_{-r}^{ 1+ \alpha} \\
&\quad + \left(\int_0^t \sup_{\norm{h}_r \leq 1} \frac{1}{n^{3/2}} \sum_{i,j=1}^n \hat{ \xi}_{ij}^{ (n)} \< \delta_{\theta^{i,n}_s}  \otimes \ \delta_{ \theta_{s}^{j, n}}, \hat{\Lambda}^n_{s, ij} \cdot \nabla S_{t-s}h >_{-r,r}  \dd s\right)^{ 1+ \alpha}\\&:= \norm{\hat{\eta}^n_0}_{-r}^{ 1+ \alpha} + (A) + (B)+ (C)
\end{split}
\end{equation}
The noise term $(B)$ can be treated similarly to $\hat{W}^n$, recall Proposition \ref{prop:Wn_prop}, using the same arguments of $w^n$, recall Proposition \ref{prop:eta^n_tight}. In particular, for an orthonormal basis $\Phi = (\varphi_p)_{p\geq 1}$ of $H^r$ one has that
\begin{equation*}
\begin{split}
\sup_{t \in [0,T]} \bE \left[ \norm{\hat{w}^n_t}^2_{-r} \right] & = \sup_{t \in [0,T]} \sum_{p\geq 1} \bE \left[ \left(\hat{w}^n_t(\varphi_p)\right)^2 \right]\\
&\leq C \sup_{t \in [0,T]} \sum_{p\geq 1} \bE \left[ \left\langle \hat{w}^n( \varphi_p) \right\rangle_t \right]\\
& \leq C' \sup_{t \in [0,T]}  \int_{ 0}^{t}\left[ \sup_{u \in \bbT^n} c^{\ijik}_n (\Phi,S_{ t-s}, u) + \sup_{u \in \bbT^n} c^{\ikjk}_n (\Phi,S_{ t-s}, u) \right] {\rm d}s,
\end{split}
\end{equation*}
where the definitions of $c_{ n}^{ \mathcal{ T}} \left(\Phi, S_{t-s},u\right)$ are given in Lemma~\ref{lem:Grothendieck_fonct} (taking in particular $k=3+\gep$ if $\gep>0$ is such that $r-2\gep>2$, so that $k-r=1-\gep$). Since $\Vert S_{t-s}\Vert_{\cL(H^r,H^k)}\leq C\left(1+\frac{1}{(t-s)^\frac{k-r}{2}}\right)$, Lemma~\ref{lem:Grothendieck_fonct} implies that $\sup_{ u\in \mathbb{ T}^{ n}}  \left\vert c_{ n}^{ \mathcal{ T}}(\Phi, S_{ t-s} , u) \right\vert \leq C\left(1+\frac{ 1}{ (t-s)^{ 1- \varepsilon}}\right) S_{ n}^{ \mathcal{ T}}$ for any $ \mathcal{ T}\in \left\lbrace \ijik, \ikjk\right\rbrace$ for some constant $C>0$ that does not depend on $s, t$ or $n$. Integrating w.r.t. $s$ and using Proposition \ref{prop:concentration_SnT} gives (recall that $ \alpha\in (0, 1)$)
\begin{equation*}
\sup_{t \in [0,T]} \bE \left[ \norm{\hat{w}^n_t}^{ 1+ \alpha}_{-r} \right] \leq \sup_{t \in [0,T]} \bE \left[ \norm{\hat{w}^n_t}^{2}_{-r} \right]^{ \frac{ 1+ \alpha}{ 2}} \leq C \left(S_{ n}^{ \ijik}+S_n^{\ikjk}\right)^{ \frac{ 1+ \alpha}{ 2}},
\end{equation*}
and from Proposition~\ref{prop:concentration_SnT} we deduce that the last quantity is $ \mathbb{ P}$-a.s. of order $ \frac{ 1}{ ( \sqrt{ n} p_{ n})^{ 1+ \alpha}}$. Consider now the second term $(A)$ in \eqref{eq:decomp_hateta_alpha}: we have (recall the definition of $ \Lambda_{ s}^{ n}$ in \eqref{eq:Lambdasn}) 
\begin{align*}
\sup_{\norm{h}_r \leq 1}\norm{\Lambda^n_s \cdot \nabla S_{t-s} h}_r \leq \sup_{\norm{h}_r \leq 1}\left\Vert \Lambda_{ s}^{ n} \right\Vert_{ W^{ r, \infty}} \left\Vert  \nabla S_{t-s} h \right\Vert_{ r} \leq  C\left(1+\frac{ 1}{ (t-s)^{ 1/2}}\right).
\end{align*}
Elevating everything to the power $ 1+ \alpha$ (recall that $ \alpha<1$) and taking the expectation gives that 
\begin{align*}
\mathbf{ E} \left[\int_0^t \norm{\hat{\eta}^n_s}_{-r}^{ 1+ \alpha} \sup_{\norm{h}_r \leq 1}\norm{\Lambda^n_s \cdot \nabla S_{t-s} h}_r^{ 1+ \alpha} {\rm d}s\right] \leq C\int_{ 0}^{t} \left(1+\frac{ 1}{ (t-s)^{ \frac{ 1+ \alpha}{ 2}}}\right)\mathbf{ E} \left[ \norm{\hat{\eta}^n_s}_{-r}^{ 1+ \alpha}\right] {\rm d}s.
\end{align*}
The expression within the integral in the last term $(C)$ in \eqref{eq:decomp_hateta_alpha}, is the sum $\sqrt{n}g^{n, (1)}_s+ \sqrt{n}g^{n, (2)}_s$ where
\begin{equation*}
\begin{split}
g^{n, (1)}_s =  \sup_{\norm{h}_r \leq 1}\frac{ 1}{n^3} \sum_{i,j,k=1}^n \hat{\xi}^{(n)}_{ij} \hat{\xi}^{(n)}_{ik}  \left[  \partial_{\theta_1}  S_{t-s} h (\theta^{i,n}_s, \theta^{j,n}_s) \right] \Gamma(\theta^{i,n}_s, \theta^{k,n}_s),\\
g^{n, (2)}_s =  \sup_{\norm{h}_r \leq 1}\frac{ 1}{n^3} \sum_{i,j,k=1}^n \hat{\xi}^{(n)}_{ij}  \hat{\xi}^{(n)}_{jk} \left[ \partial_{\theta_2} S_{t-s} h (\theta^{i,n}_s, \theta^{j,n}_s)\right] \Gamma(\theta^{j,n}_s, \theta^{k,n}_s).
\end{split}
\end{equation*}
These terms are very similar to the term $C^n_t$ defined in  \eqref{eq:def Cn}, and can be treated similarly as done Proposition~\ref{prop:bound Cn}, but relying moreover on the regularity of the semigroup $S_{ t-s}$. More precisely one can follow the steps of the proof of Proposition~\ref{prop:bound Cn}, only replacing $f$ with $S_{t-s}g$, the bound given \eqref{eq:norm_z1} becoming then
\begin{equation*}
\Vert  x_{1,i}\Vert_{l^2(\bbT^2)}\leq C\Vert S_{t-s}f\Vert_{3}\leq C' \left(1+\frac{1}{(t-s)^{\frac{3-r}{2}}}\right)\Vert f\Vert_{r}.
\end{equation*}
which leads to the bounds, for some constant $C>0$ independent on $s< t\leq T$
\begin{align*}
g_{ s}^{ n, (1)} \leq \left(1+\frac{1}{(t-s)^{\frac{3-r}{2}}}\right)S_{ n}^{ \ijik} \quad \text{and} \quad g_{ s}^{ n, (2)} \leq\left(1+\frac{1}{(t-s)^{\frac{3-r}{2}}}\right) S_{ n}^{ \ijjk}
\end{align*}
Integrating w.r.t. $s$, elevating to the power $1+ \alpha$ and taking the expectation gives finally that $ \mathbf{ E} \left[ (C)\right]\leq \frac{ C}{ \left(\sqrt{ n} p_{ n}^{ 2}\right)^{ 1+ \alpha}}$, $ \mathbb{ P}$-a.s. Finally, with another application of Gronwall-Henri inequality, one obtains that the process $\hat{\eta}^n$ satisfies $\bbP$-a.s.
\begin{equation*}
\sup_{t \in [0,T]} \bE [ \norm{\hat{\eta}^n_t}_{-r}^{ 1+ \alpha} ] \leq C \left( \norm{\hat{\eta}^n_0}_{-r}^{ 1+ \alpha} + \frac 1{ \left(\sqrt{n}p_n\right)^{ 1+ \alpha}} + \frac 1{ \left(\sqrt{ n}p_{ n}^{ 2}\right)^{ 1+ \alpha}} \right).
\end{equation*}
This concludes the proof.
\end{proof}

To prove Proposition~\ref{prop:charac eta tilde eta} it remains to give some regularity results for the operators $ \mathcal{ L}_{ \mu^n_s}^{(1)}$ and $ \mathcal{ L}_{ \mu^n_s}^{(2)}$. They are continuous in a suitable class of Hilbert spaces, as stated below.

\begin{lemma}
\label{lem:continuity_Ln12}
Fix $r > 0$. For every probability measure $\nu$, the linear operator $ \mathcal{ L}_{ \nu}^{(1)}$ (resp. $ \mathcal{ L}_{ \nu}^{(2)}$) is continuous from $H^{ -(r+2)}(\mathbb{ T})$ to $H^{ -r}(\mathbb{ T})$ (resp. from $H^{ -(r+2)}(\mathbb{ T}^{ 2})$ to $H^{ -r}(\mathbb{ T}^{ 2})$): there exist two positive constants $C_{ 1,T}$ and $C_{ 2, T}$ such that for all test function $f(\theta)$ and $g(\theta_{ 1}, \theta_{ 2})$,
\begin{align}
\left\Vert \mathcal{ L}_{ \nu}^{(1)}(f) \right\Vert_{ r} &\leq C_{ 1, T} \left\Vert f \right\Vert_{ r+2}, \label{eq:L1n_bound}\\
\left\Vert \mathcal{ L}_{ \nu}^{(2)}(g) \right\Vert_{ r} &\leq C_{ 2, T} \left\Vert g \right\Vert_{ r+2}. \label{eq:L2n_bound}
\end{align}
\end{lemma}

\begin{proof}
This is a straightforward consequence of the definitions \eqref{eq:L1} and \eqref{eq:L2}, Assumption \ref{ass:Gamma} and the fact that $\nu$ is a probability measure. See \cite[Lemma 3.7]{Fernandez1997} for a very similar proof.
\end{proof}

\begin{proof}[Proof of Proposition~\ref{prop:charac eta tilde eta}]
The continuity of the processes $ \left( \eta^{ n}, \hat{ \eta}^{ n}\right)$ has been addressed in Lemmas~\ref{lem:fluctuation_process_semimart} and~\ref{lem:centered_emp_mes_spde}. Both semimartingale decompositions in \eqref{eq:limit_etas} hold if each term make sense in $H^{ - r_{ 1}}(\mathbb{ T})$ and $H^{ - r_{ 1}}\left(\mathbb{ T}^{ 2}\right)$ respectively.  Since $ \left\Vert \mathcal{ L}_{ \mu^n_s}^{ (1), \ast} \eta_{ s}^{ n} \right\Vert_{ -(r+2)} \leq C \left\Vert \eta_{ s}^{ n} \right\Vert_{ -r}$ for any $r> 0$ by Lemma~\ref{lem:continuity_Ln12}, one obtains from \eqref{eq:first_control_etan} and \eqref{eq:first_control_hat_etan} that, under both Assumptions~\ref{ass:jointconv_0W_qu} and~\ref{ass:jointconv_0W_ann} $, \int_{ 0}^{ T} \left\Vert \mathcal{ L}_{ \mu^n_s}^{ (1), \ast} \eta_{ s}^{ n} \right\Vert_{ -r_{ 1}} {\rm d}s$ is $ \mathbb{ P}\otimes \mathbf{ P}$ almost-surely finite (recall the definition of $r_{ 0}, r_{ 1}$ in \eqref{eq:indices}), so that this drift term makes sense as a Bochner integral. It suffices now to gather this estimate and Lemma~\ref{lem:Mn_prop} to conclude. The same argument works also for $ \hat{ \eta}^{ n}$.
\end{proof}
\subsection{Tightness results}
We use the following tightness criterion \cite[pp. 34-35]{joffe_metivier_86}: a sequence of $ \left( \Omega^{ n}, \mathcal{ F}_{ t}^{ n}\right)$-adapted processes $ (Y^{ n})_{ n\geq1}$ with path in $ \mathcal{ C}([0, T], H)$, where $H$ is an Hilbert space is tight if both of the following conditions hold:
\begin{enumerate}
\item For every $t$ in some dense subset of $[0, T]$, the law of $Y_{ t}^{ n}$ is tight in $H$,
\item (Aldous condition) For every $ \varepsilon_{ 1}, \varepsilon_{ 2}>0$, there exists $ \delta>0$ and $n_{ 0}\geq1$ such that for every $ \left(\mathcal{ F}_{ t}^{ n}\right)$-stopping time $ \tau_{ n}\geq T$,
\begin{equation*}
\sup_{ n\geq n_{ 0}} \sup_{ \theta\leq \delta} \mathbf{ P} \left( \left\Vert Y_{ \tau_{ n}}^{ n} - Y_{ \tau_{ n} + \theta}^{ n} \right\Vert_{ H} \geq \varepsilon_{ 1}\right)\leq \varepsilon_{ 2}.
\end{equation*}
\end{enumerate}
\begin{rem}
\label{rem:compact}
Suppose that there exists a Hilbert space $H_{ 0}$ such that the injection $H_{ 0} \to H$ is compact and such that, there exists $m\geq1$ such that for fixed $t\in[0, T]$, we have $ \sup_{ n} \mathbf{ E} \left( \left\Vert Y_{ t}^{ n} \right\Vert_{ H_{ 0}}^{ m}\right)< +\infty$. Then, for this $t$, Condition (1) above is satisfied. Indeed, for any $R>0$, $B_{ R}:= \left\lbrace h\in H_{ 0}, \ \left\Vert h \right\Vert_{ H_{ 0}}\leq R\right\rbrace$ is compact in $H$, and, by Markov inequality, \[ \mathbf{ P} \left( Y_{ t}^{ n} \notin B_{ R}\right)= \mathbf{ P} \left( \left\Vert Y_{ t}^{ n} \right\Vert_{ H_{ 0}} > R\right)\leq \frac{ \mathbf{ E}\left( \left\Vert Y_{ t}^{ n} \right\Vert_{ H_{ 0}}^{ m}\right)}{ R^{ m}},\] which goes to $0$ as $R\to\infty$, uniformly in $n$.
\end{rem}
\begin{theorem}
\label{th:tight}
Recall the definition of $(r_{ 0}, r_{ 1})$ in \eqref{eq:indices}. Under Assumptions~\ref{ass:Gamma},~\ref{ass:initial} and either Assumption~\ref{ass:jointconv_0W_qu} or~\ref{ass:jointconv_0W_ann}, the sequence of laws of both $ \left(\eta^{ n}\right)_{ n\geq1}$ (resp. $  \left(\hat{ \eta}^{ n}\right)_{ n\geq1}$) is tight in $ \cC \left([0, T], H^{ - r_{ 1}} \left(\mathbb{ T}\right)\right)$ (resp. in $ \cC \left([0, T], H^{ - r_{ 1}} \left(\mathbb{ T}^{ 2}\right)\right)$) (and the tightness is $ \mathbb{ P}_{ g}$-almost sure in the case of Assumption~\ref{ass:jointconv_0W_qu}).
\end{theorem}
\begin{proof}
We deal first with the process $ \hat{ \eta}^{ n}$. Recall the definition of $(r_{ 0}, r_{ 1})$ in \eqref{eq:indices}. We apply the above tightness criterion in the case $H_{ 0}= H^{ - r_{ 0}}\left(\mathbb{ T}^{ 2}\right)$ and $H=H^{ - r_{ 1}}\left(\mathbb{ T}^{ 2}\right)$. Applying \eqref{eq:first_control_hat_etan} for $r=r_{ 0} > 3$, we obtain from Remark~\ref{rem:compact} that Condition (1) of the tightness result is satisfied. Let us verify now the Aldous condition: let $ \varepsilon_{ 1}, \varepsilon_{ 2}>0$, $\theta\geq0$ and $ \tau_{ n}$ a stopping time. Then, from \eqref{eq:decomp hat eta_n},
\begin{align}
\mathbf{ P} \Big( &\left\Vert \hat{ \eta}^{ n}_{ \tau_{ n}+ \theta}  - \hat{ \eta}_{ \tau_{ n}}^{ n}\right\Vert_{ -r_{ 1}} \geq \varepsilon_{ 1}\Big) \nonumber\\
&= \mathbb{ P} \left( \left\Vert \int_{ \tau_{ n}}^{ \tau_{ n}+ \theta} \mathcal{ L}_{ \mu^n_s}^{(2), \ast} \hat{\eta}_{ s}^{ n} {\rm d}s + \sqrt{n}\int_{ \tau_{ n}}^{ \tau_{ n}+ \theta} C_{ s}^{ n} {\rm d}s+ \hat{ W}_{ \tau_{ n}+ \theta}^{ n} - \hat{ W}_{ \tau_{ n}}^{ n}\right\Vert_{ -r_{ 1}} \geq \varepsilon_{ 1}\right) \nonumber\\
&\leq \mathbf{ P} \left( \left\Vert \int_{ \tau_{ n}}^{ \tau_{ n}+ \theta} \mathcal{ L}_{ \mu^n_s}^{(2), \ast} \hat{\eta}_{ s}^{ n} {\rm d}s \right\Vert_{ -r_{ 1}} \geq \frac{ \varepsilon_{ 1}}{ 3}\right) \label{aux:tight1}\\
&+\mathbf{ P} \left( \left\Vert \sqrt{n}\int_{ \tau_{ n}}^{ \tau_{ n}+ \theta} C_{ s}^{ n} {\rm d}s\right\Vert_{ -r_{ 1}} \geq \frac{ \varepsilon_{ 1}}{ 3}\right)\label{aux:tight2}\\
&+\mathbf{ P} \left( \left\Vert \hat{ W}_{ \tau_{ n}+ \theta}^{ n} - \hat{ W}_{ \tau_{ n}}^{ n}\right\Vert_{ -r_{ 1}} \geq \frac{ \varepsilon_{ 1}}{ 3}\right).\label{aux:tight3}
\end{align}
Concentrate on the first term \eqref{aux:tight1}: by Markov and Jensen's inequalities
\begin{align*}
\mathbf{ P} \left( \left\Vert \int_{ \tau_{ n}}^{ \tau_{ n}+ \theta} \mathcal{ L}_{ \mu^n_s}^{(2), \ast} \hat{\eta}_{ s}^{ n} {\rm d}s \right\Vert_{ -r_{ 1}} \geq \frac{ \varepsilon_{ 1}}{ 3}\right) & \leq \frac{ 1}{ \left( \varepsilon_{ 1}/3\right)^{ 1+ \alpha}} \mathbf{ E} \left( \left(\int_{ \tau_{ n}}^{ \tau_{ n}+ \theta} \left\Vert \mathcal{ L}_{ \mu^n_s}^{(2), \ast} \hat{\eta}_{ s}^{ n} \right\Vert_{ -r_{ 1}} {\rm d}s\right)^{ 1+ \alpha}\right)\\
& \leq \frac{ \theta^{ \alpha}}{ \left( \varepsilon_{ 1}/3\right)^{ 1+ \alpha}} \mathbf{ E} \left(\int_{ \tau_{ n}}^{ \tau_{ n}+ \theta} \left\Vert \mathcal{ L}_{ \mu^n_s}^{(2), \ast} \hat{\eta}_{ s}^{ n} \right\Vert_{ -r_{ 1}}^{ 1+ \alpha} {\rm d}s\right)\\
& \leq \frac{ C\theta^{ \alpha}}{ \left( \varepsilon_{ 1}/3\right)^{ 1+ \alpha}} \mathbf{ E} \left(\int_{ 0}^{T} \left\Vert \hat{\eta}_{ s}^{ n} \right\Vert_{ -r_{ 0}}^{ 1+ \alpha} {\rm d}s\right),
\end{align*}
where we used Lemma~\ref{lem:continuity_Ln12}. We are now in position to apply \eqref{eq:first_control_hat_etan} for $r=r_{ 0}$, so that, for another constant $C>0$ independent of $n$, 
\begin{equation*}
\mathbf{ P} \left( \left\Vert \int_{ \tau_{ n}}^{ \tau_{ n}+ \theta} \mathcal{ L}_{ \mu^n_s}^{(2), \ast} \hat{\eta}_{ s}^{ n} {\rm d}s \right\Vert_{ -r_{ 1}} \geq \frac{ \varepsilon_{ 1}}{ 3}\right)  \leq \frac{ C\theta^{ \alpha}}{ \left( \varepsilon_{ 1}/3\right)^{ 1+ \alpha}}.
\end{equation*}
The second term \eqref{aux:tight2} is treated using Proposition~\ref{prop:bound Cn}: since $r_{ 1} > 5\geq 3$,
\begin{align*}
\mathbf{ P} \left( \left\Vert \sqrt{n}\int_{ \tau_{ n}}^{ \tau_{ n}+ \theta} C_{ s}^{ n} {\rm d}s\right\Vert_{ -r_{ 1}} \geq \frac{ \varepsilon_{ 1}}{ 3}\right)\leq \frac{ T \sqrt{ n} \mathbf{ E} \left(\sup_{ s\in[0, T]} \left\Vert C_{ s}^{ n} \right\Vert_{ -r_{ 1}}\right)}{ \varepsilon_{ 1}/3},
\end{align*}
which goes to $0$ as $n\to\infty$, by Proposition~\ref{prop:bound Cn}, uniformly in $ \theta\leq \delta$. The last term \eqref{aux:tight3} follows from a similar argument, using now Proposition~\ref{prop:Wn_prop}. We have the rough bound 
\begin{align*}
\mathbf{ P} \left( \left\Vert \hat{ W}_{ \tau_{ n}+ \theta}^{ n} - \hat{ W}_{ \tau_{ n}}^{ n}\right\Vert_{ -r_{ 1}} \geq \frac{ \varepsilon_{ 1}}{ 3}\right) \leq \frac{ 4 \mathbf{ E} \left(\sup_{ s\in[0, T]} \left\Vert \hat{ W}_{ s}^{ n} \right\Vert_{ -r_{ 1}}^{ 2}\right)}{ \left(\varepsilon_{ 1}/3\right)^{ 2}}.
\end{align*}
Since $r_{ 1}\geq4$, we have by Proposition~\ref{prop:Wn_prop} that this term also goes to $0$ as $n\to\infty$. Putting all the previous estimates together, we obtain that Aldous criterion is verified, and hence the tightness of $ \hat{ \eta}^{ n}$ in $ \mathcal{ C} \left([0, T], H^{ - r_{ 1}}\left(\mathbb{ T}^{ 2}\right)\right)$.

Now turn to the tightness of $ \eta^{ n}$, which follows from very similar arguments. The point (1) of the tightness criterion follows directly from \eqref{eq:first_control_etan} and \eqref{eq:first_control_hat_etan} (for $r=r_{ 0}$) and Remark~\ref{rem:compact}. In a similar way, for the Aldous criterion, 
\begin{align*}
\mathbf{ P} \left( \left\Vert \eta^{ n}_{ \tau_{ n}+ \theta}  -\eta_{ \tau_{ n}}^{ n}\right\Vert_{ -r_{ 1}} \geq \varepsilon_{ 1}\right) 
&\leq \mathbf{ P} \left( \left\Vert \int_{ \tau_{ n}}^{ \tau_{ n}+ \theta} \mathcal{ L}_{ \mu^n_s}^{(1), \ast} \eta_{ s}^{ n} {\rm d}s \right\Vert_{ -r_{ 1}} \geq \frac{ \varepsilon_{ 1}}{ 3}\right)\\
&\quad +\mathbf{ P} \left( \left\Vert \int_{ \tau_{ n}}^{ \tau_{ n}+ \theta} \partial_{\theta_1} (\Gamma * \hat{\eta}^n_s) {\rm d}s\right\Vert_{ -r_{ 1}} \geq \frac{ \varepsilon_{ 1}}{ 3}\right)\\
&\quad +\mathbf{ P} \left( \left\Vert W_{ \tau_{ n}+ \theta}^{ n} - W_{ \tau_{ n}}^{ n}\right\Vert_{ -r_{ 1}} \geq \frac{ \varepsilon_{ 1}}{ 3}\right).
\end{align*}
The first term is treated in the exact same way as for \eqref{aux:tight1} above and we leave the details to the reader. Similarly, we have
\begin{align*}
\mathbf{ P} \left( \left\Vert \int_{ \tau_{ n}}^{ \tau_{ n}+ \theta} \partial_{\theta_1} (\Gamma * \hat{\eta}^n_s) {\rm d}s\right\Vert_{ -r_{ 1}} \geq \frac{ \varepsilon_{ 1}}{ 3}\right)&\leq \frac{ 1}{ \left( \varepsilon_{ 1}/3\right)^{ 1+ \alpha}} \mathbf{ E} \left( \left(\int_{ \tau_{ n}}^{ \tau_{ n}+ \theta} \left\Vert \partial_{\theta_1} (\Gamma * \hat{\eta}^n_s) \right\Vert_{ -r_{ 1}} {\rm d}s\right)^{ 1+ \alpha}\right),\\
& \leq \frac{ \theta^{ \alpha}}{ \left( \varepsilon_{ 1}/3\right)^{ 1+ \alpha}} \mathbf{ E} \left(\int_{ \tau_{ n}}^{ \tau_{ n}+ \theta} \left\Vert \partial_{\theta_1} (\Gamma * \hat{\eta}^n_s) \right\Vert_{ -r_{ 1}}^{ 1+ \alpha} {\rm d}s\right),\\
& \leq \frac{ C\theta^{ \alpha}}{ \left( \varepsilon_{ 1}/3\right)^{ 1+ \alpha}} \mathbf{ E} \left(\int_{ 0}^{T} \left\Vert \hat{\eta}_{ s}^{ n} \right\Vert_{ -r_{ 0}}^{ 1+ \alpha} {\rm d}s\right),
\end{align*}
where we used $ \left\Vert  \partial_{\theta_1} (\Gamma * \hat{\eta}^n_s) \right\Vert_{ -r_{ 1}} \leq \left\Vert  \partial_{\theta_1} (\Gamma * \hat{\eta}^n_s) \right\Vert_{ -(r_{ 1}-1)} \leq C \left\Vert \hat{ \eta}_{ s}^{ n} \right\Vert_{ -(r_{ 1}-2)}=C \left\Vert \hat{ \eta}_{ s}^{ n} \right\Vert_{ -r_{ 0}} $ for a constant $C>0$ independent of $s$. Recall now that $r_{ 0}> 3$ so that \eqref{eq:first_control_hat_etan} holds. We finally turn to the noise term $W^{ n}$ whose treatment differs slightly from the previous calculations: we reproduce here (without giving all the details) some parts of the calculation made in \cite[p. 40]{joffe_metivier_86} (as part of the proof of the Rebolledo's theorem): we have, for all $a>0$
\begin{align*}
\mathbf{ P} \left( \left\Vert W_{ \tau_{ n}+ \theta}^{ n} - W_{ \tau_{ n}}^{ n}\right\Vert_{ -r_{ 1}} \geq \varepsilon_{ 1}\right) &\leq \frac{ 1}{ \varepsilon_{ 1}^{ 2}} \mathbf{ E} \left( \left({\rm tr}_{ H^{ -r_{ 1}}}\llangle W^{ n}\rrangle_{ \tau_{ n}+ \theta} - {\rm tr}_{ H^{ -r_{ 1}}}\llangle W^{ n}\rrangle_{ \tau_{ n}}\right) \wedge a\right) \\
&+ \mathbf{ P} \left( {\rm tr}_{ H^{ -r_{ 1}}}\llangle W^{ n}\rrangle_{ \tau_{ n}+ \theta} - {\rm tr}_{ H^{ -r_{ 1}}}\llangle W^{ n}\rrangle_{ \tau_{ n}} \geq a\right),
\end{align*}
where $ \llangle W^{ n} \rrangle$ is the Doob-Meyer process associated to $W^{ n}$ given by \eqref{eq:DM_Mn}. Choosing here $a \leq \varepsilon_{ 1}^{ 2} \varepsilon_{ 2}/2$, we see that it suffices to control 
\begin{align*}
\mathbf{ P} \left( \left\vert {\rm tr}_{ H^{ -r_{ 1}}} \llangle W^{ n}\rrangle_{ \tau_{ n}+ \theta} -  {\rm tr}_{ H^{ -r_{ 1}}} \llangle W^{ n}\rrangle_{ \tau_{ n}}\right\vert> a\right) &\leq \frac{ 1}{ a} \mathbf{ E} \left( \int_{ \tau_{ n}}^{ \tau_{ n}+ \theta} \frac{ 1}{ n} \sum_{ i=1}^{ n} \sum_{ p\geq 1} \left\vert \psi_{ p}^{ \prime} \left( \theta_{ s}^{ i, n}\right) \right\vert^{ 2} {\rm d}s \right)\\
&\leq \frac{ 1}{ a} \mathbf{ E} \left( \int_{ \tau_{ n}}^{ \tau_{ n}+ \theta} \frac{ 1}{ n} \sum_{ i=1}^{ n} \left\Vert D_{ \theta_{ s}^{ i, n}} \right\Vert_{ -r_{ 1}}^{ 2} {\rm d}s\right)\leq \frac{ C \theta}{ a},
\end{align*}
where $ (\psi_{ p})$ is a complete orthonormal system in $H^{r_{ 1}}$ and using Lemma~\ref{lem:linear_forms}. This gives the result.
\end{proof}
\subsection{ Convergence}
The first result of this paragraph concerns the identification of the limit of the noise term:
\begin{proposition}[Identification of the noise]
\label{prop:limW}
Under the assumptions of Section~\ref{sec:hypotheses}, $ \mathbb{ P}$-a.s., the process $ \left(W^{ n}\right)$ converges in law in $ \mathcal{ C} \left([0, T], H^{ -r}\right)$ for all $r>2$ to the Gaussian process $W\in  \mathcal{ C} \left([0, T], H^{ -r}\right)$ with covariance given by \eqref{eq:cov_Weta}.
\end{proposition}
\begin{proof}
Tightness of $(W^{ n})$ follows from \eqref{eq:control_supMn} and the same calculations as the end of the proof of Theorem~\ref{th:tight} (with $H^{ -r}$ in place of $H^{ -r_{ 1}}$). Identification of the limit is a simple consequence of \eqref{eq:DM_Mn} and the weak convergence \eqref{eq:conv_empmeas} of the empirical measure $ \mu_{ n}$ (see also \cite{Fernandez1997}, Th.~5.2 for a similar proof).
\end{proof}
\begin{proposition}
\label{prop:ident_lim} 
Recall the definition of $(r_{ 0}, r_{ 1}, r_{ 2})$ in \eqref{eq:indices}. Under Assumption~\ref{ass:Gamma}, Assumption~\ref{ass:initial} and either Assumptions~\ref{ass:jointconv_0W_qu} and~\ref{ass:jointconv_0W_ann}, the process $ (\eta^{ n}, \hat{ \eta}^{ n})$ has convergent subsequences in $ \mathcal{ C} \left( [0, T], H^{ -r_{ 1}}\left(\mathbb{ T}\right)\oplus H^{ - r_{ 1}}( \mathbb{ T}^{ 2})\right)$ and any accumulation point $ \left(\eta, \hat{ \eta}\right)$ is a solution in the space $ \mathcal{ C} \left([0, T], H^{ -r_{ 2}}\left( \mathbb{ T}\right) \oplus H^{ -r_{ 2}}( \mathbb{ T}^{ 2})\right)$ to the system \eqref{eq:limit_etas}.
\end{proposition}
\begin{proof}
This proof follows the classical arguments of \cite{Fernandez1997} with adequate technical changes. Consider such a convergent subsequence that we rename $ \left(\eta^{ n}, \hat{ \eta}^{ n}\right)$ for convenience and let $ \left(\eta, \hat{ \eta}\right)$ its limit in $ \mathcal{ C} \left( [0, T], H^{ -r_{ 1}}\left(\mathbb{ T}\right)\oplus H^{ - r_{ 1}}( \mathbb{ T}^{ 2})\right)$. We prove that $ (\eta, \hat{ \eta})$ solves \eqref{eq:limit_etas}.

From Assumption~\ref{ass:jointconv_0W_qu}, \eqref{eq:first_control_etan} and \eqref{eq:first_control_hat_etan}, we deduce (replace it by the same estimate with some additional $ \mathbb{ E}_{ g}$ in case of Assumption~\ref{ass:jointconv_0W_ann}) that $ \mathbb{ P}_{ g}$-a.s., 
\begin{equation*}
\mathbb{ E}_{ 0}\mathbf{ E} \left[\sup_{ t\leq T} \left(\left\Vert \eta_{ t} \right\Vert_{ -r_{ 1}}^{ 1+\alpha} + \left\Vert \hat{ \eta}_{ t} \right\Vert_{ -r_{ 1}}^{ 1+ \alpha}\right)\right]< \infty,
\end{equation*}
so that $\int_{ 0}^{t}\mathcal{ L}_{ \mu_s}^{(1), \ast}\eta_{ s} {\rm d}s$ (resp. $\int_{ 0}^{t}\mathcal{ L}_{ \mu_s}^{(2), \ast}\hat{ \eta}_{ s} {\rm d}s$) makes sense as a Bochner integral in $H^{ -r_{ 2}} \left(\mathbb{ T}\right)$ (resp. $H^{ -r_{ 2}} \left(\mathbb{ T}^{ 2}\right)$). For the same reason, since $ \Gamma$ is regular with bounded derivatives (Assumption~\ref{ass:Gamma}), we easily see from the definition of $ \Theta$ in \eqref{eq:Theta} that $ \int_{ 0}^{t} \Theta^{ \ast} \hat{ \eta}_{ s} {\rm d}s$ makes sense as a Bochner integral in $ H^{ -r_{ 2}}(\mathbb{ T})$. Let us start with the second equation in \eqref{eq:limit_etas}: we know from \eqref{eq:control_supWn} that $ \hat{ W}^{ n}$ converges in $ \mathcal{ C}\left([0, T], H^{ -r_{ 1}}\left(\mathbb{ T}^{ 2}\right)\right)$ to $0$ as $n\to\infty$ and from Proposition~\ref{prop:bound Cn} combined with Proposition~\ref{prop:concentration_SnT} that, under the dilution condition \eqref{eq:dilution_cond}, $ \sqrt{ n} C^{ n}$ converges to $0$ in $\mathcal{ C}\left([0, T], H^{ -r_{ 1}}\left(\mathbb{ T}^{ 2}\right)\right)$. Hence, to prove the result, it suffices to show that $ \left\langle \hat{ \eta}_{ t}^{ n}\, ,\, g\right\rangle- \left\langle \hat{ \eta}_{ 0}^{ n}\, ,\, g\right\rangle - \int_{ 0}^{t} \left\langle \hat{ \eta}_{ s}^{ n}\, ,\, \mathcal{ L}_{ \mu^n_s}^{(2)}(g)\right\rangle {\rm d}s$ converges in law, as $n\to\infty$, to $\left\langle \hat{ \eta}_{ t}\, ,\, g\right\rangle- \left\langle \hat{ \eta}_{ 0}\, ,\, g\right\rangle - \int_{ 0}^{t} \left\langle \hat{ \eta}_{ s}\, ,\, \mathcal{ L}_{\mu_s}^{(2)}(g)\right\rangle {\rm d}s$ for all test function $g\in H^{ r_{ 2}}\left(\mathbb{ T}^{ 2}\right)$. Decompose the previous quantity into, for $F_{ g}(\alpha)_{ t}:=\left\langle \alpha_{ t}\, ,\, g\right\rangle- \left\langle \alpha_{ 0}\, ,\, g\right\rangle - \int_{ 0}^{t} \left\langle \alpha_{ s}\, ,\, \mathcal{ L}_{ \mu_s}^{(2)}(g)\right\rangle {\rm d}s$,
\begin{equation*}
\left\langle \hat{ \eta}_{ t}^{ n}\, ,\, g\right\rangle- \left\langle \hat{ \eta}_{ 0}^{ n}\, ,\, g\right\rangle - \int_{ 0}^{t} \left\langle \hat{ \eta}_{ s}^{ n}\, ,\, \mathcal{ L}_{ \mu^n_s}^{(2)}(g)\right\rangle {\rm d}s = F_{ g} \left( \hat{ \eta}^{ n}\right)_{ t} - \int_{ 0}^{t} \left\langle \hat{ \eta}_{ s}^{ n}\, ,\, \left\lbrace \mathcal{ L}_{ \mu^n_s}^{(2)} - \mathcal{ L}_{ \mu_s}^{(2)}\right\rbrace(g)\right\rangle {\rm d}s.
\end{equation*}
It is easy to see that for fixed $g$ in $H^{ r_{ 2}} \left(\mathbb{ T}^{ 2}\right)$, $ \alpha \in \mathcal{ C} \left([0, T], H^{ -r_{ 1}}( \mathbb{ T}^{ 2})\right) \mapsto F_{ g}(\alpha) \in \mathcal{ C} \left([0, T], \mathbb{ R}\right)$ is continuous, so that, since $ \left( \hat{ \eta}^{ n}\right)$ converges in law as $n\to\infty$ to $ \hat{ \eta}$ in $ \mathcal{ C} \left([0, T], H^{ -r_{ 1}}(\mathbb{ T}^{ 2})\right)$, we have that $ F_{ g} \left( \hat{ \eta}^{ n}\right)$ converges in law to $ F_{ g} \left( \hat{ \eta}\right)$ as $n\to\infty$. It remains to prove that $\int_{ 0}^{t} \left\langle \hat{ \eta}_{ s}^{ n}\, ,\, \left\lbrace \mathcal{ L}_{ \mu^n_s}^{(2)} - \mathcal{ L}_{ \mu_s}^{(2)}\right\rbrace(g)\right\rangle {\rm d}s $ converges in law to $0$. We prove that it goes to $0$ in $L^{ 1}$: we have
\begin{equation}
\label{aux:rest_hat_eta}
\mathbb{ E}_{ 0}\mathbf{ E} \left[ \int_{ 0}^{t} \left\vert \left\langle \hat{ \eta}_{ s}^{ n}\, ,\, \left\lbrace \mathcal{ L}_{ \mu^n_s}^{(2)} - \mathcal{ L}_{ \mu_s}^{(2)}\right\rbrace(g)\right\rangle \right\vert {\rm d}s\right] \leq \mathbb{ E}_{ 0}\mathbf{ E} \left[ \int_{ 0}^{t}  \left\Vert \hat{ \eta}^{ n}_{ s} \right\Vert_{ -r_{ 1}} \left\Vert \left\lbrace\mathcal{ L}_{ \mu^n_s}^{(2)} - \mathcal{ L}_{ \mu_s}^{(2)}\right\rbrace(g) \right\Vert_{ r_{ 1}}{\rm d}s\right].
\end{equation}
Concentrate on $\left\Vert \left\lbrace\mathcal{ L}_{ \mu^n_s}^{(2)} - \mathcal{ L}_{ \mu_s}^{(2)}\right\rbrace(g) \right\Vert_{ r_{ 1}}$: by definition of $ \mathcal{ L}^{(2)}$ (recall \eqref{eq:L2}), it suffices to estimate $ \left\Vert \partial_{ \theta_{ 1}}g \left\langle (\mu_{ s}^{ n} - \mu_{ s})({\rm d}\theta^{ \prime})\, ,\, \Gamma \left(\cdot, \theta^{ \prime}\right)\right\rangle \right\Vert_{ r_{ 1}}$ (the other term being treated analogously):
\begin{align*}
\big\Vert \partial_{ \theta_{ 1}}g &\left\langle (\mu_{ s}^{ n} - \mu_{ s})({\rm d}\theta^{ \prime})\, ,\, \Gamma \left(\cdot, \theta^{ \prime}\right)\right\rangle \big\Vert_{ r_{ 1}}^{ 2}\\
&= \sum_{ \underset{k+l\leq r_{ 1}}{k,l\geq 0}} \int_{ \mathbb{ T}^{ 2}} \left\vert \partial_{ \theta_{ 1}}^{ k}\left( \partial_{ \theta_{ 1}} \partial_{ \theta_{ 2}}^{ l} g(\theta_{ 1}, \theta_{ 2}) \int \Gamma(\theta_{ 1}, \theta^{ \prime})(\mu_{ s}^{ n}- \mu_{ s})({\rm d}\theta^{ \prime})\right)\right\vert^{ 2} {\rm d}\theta_{ 1} {\rm d}\theta_{ 2}\\
&=\sum_{ \underset{k+l\leq r_{ 1}}{k,l\geq 0}} \int_{ \mathbb{ T}^{ 2}} \left\vert \sum_{ j=0}^{ k} \binom{k}{j}\left( \partial_{ \theta_{ 1}}^{ j+1} \partial_{ \theta_{ 2}}^{ l} g(\theta_{ 1}, \theta_{ 2}) \int \partial_{ \theta_{ 1}}^{ k-j}\Gamma(\theta_{ 1}, \theta^{ \prime})(\mu_{ s}^{ n}- \mu_{ s})({\rm d}\theta^{ \prime})\right)\right\vert^{ 2} {\rm d}\theta_{ 1} {\rm d}\theta_{ 2}\\
&\leq \sum_{ \underset{k+l\leq 6}{k,l\geq 0}} 2^{ k} \sum_{ j=0}^{ k} \binom{k}{j}\int_{ \mathbb{ T}^{ 2}} \left\vert \partial_{ \theta_{ 1}}^{ j+1} \partial_{ \theta_{ 2}}^{ l} g(\theta_{ 1}, \theta_{ 2})\right\vert^{ 2} \left\vert \int \partial_{ \theta_{ 1}}^{ k-j}\Gamma(\theta_{ 1}, \theta^{ \prime})(\mu_{ s}^{ n}- \mu_{ s})({\rm d}\theta^{ \prime})\right\vert^{ 2} {\rm d}\theta_{ 1} {\rm d}\theta_{ 2}.
\end{align*}
Since $ \Gamma$ is regular with bounded derivatives, we have the following bound (recall the definition of $d_{ BL}$ in \eqref{eq:dBL})
\begin{equation*}
\left\vert \int \partial_{ \theta_{ 1}}^{ k-j}\Gamma(\theta_{ 1}, \theta^{ \prime})(\mu_{ s}^{ n}- \mu_{ s})({\rm d}\theta^{ \prime})\right\vert^{ 2} \leq C d_{ BL} \left( \mu_{ s}^{ n}, \mu_{ s}\right)^{ 2},
\end{equation*}
where the constant $C$ above is uniform in $s$ and $0\leq j \leq k \leq r_{ 1}$. Hence, we obtain, for another numerical constant $ \tilde{ C}>0$
\begin{align*}
\left\Vert \left\lbrace\mathcal{ L}_{ \mu^n_s}^{(2)} - \mathcal{ L}_{ \mu_s}^{(2)}\right\rbrace(g) \right\Vert_{ r_{ 1}}&\leq \tilde{ C} d_{ BL} \left( \mu_{ s}^{ n}, \mu_{ s}\right)\left\Vert g \right\Vert_{ r_{ 1}+1}.
\end{align*}
Going back to \eqref{aux:rest_hat_eta}, we obtain by H\"{o}lder inequality and using \eqref{eq:first_control_hat_etan}, 
\begin{align*}
\mathbb{ E}_{ 0}\mathbf{ E} &\left[ \int_{ 0}^{t} \left\vert \left\langle \hat{ \eta}_{ s}^{ n}\, ,\, \left\lbrace \mathcal{ L}_{ \mu^n_s}^{(2)} - \mathcal{ L}_{ \mu_s}^{(2)}\right\rbrace(g)\right\rangle \right\vert {\rm d}s\right] \leq \tilde{ C} \left\Vert g \right\Vert_{ r_{ 1}+1} \int_{ 0}^{t} \mathbb{ E}_{ 0}\mathbf{ E} \left[\left\Vert \hat{ \eta}^{ n}_{ s} \right\Vert_{ -r_{ 1}}d_{ BL} \left( \mu_{ s}^{ n}, \mu_{ s}\right)\right] {\rm d}s\\
&\leq \tilde{ C} \left\Vert g \right\Vert_{ r_{ 1}+1} \int_{ 0}^{t} \mathbb{ E}_{ 0}\mathbf{ E} \left[\left\Vert \hat{ \eta}^{ n}_{ s} \right\Vert_{ -r_{ 1}}^{ 1+ \alpha} \right]^{ \frac{ 1}{ 1+ \alpha}} \mathbb{ E}_{ 0}\mathbf{ E}\left[d_{ BL} \left( \mu_{ s}^{ n}, \mu_{ s}\right)^{ \frac{ 1+ \alpha}{ \alpha}}\right]^{ \frac{ \alpha}{ 1+ \alpha}} {\rm d}s\\
&\leq \tilde{ C} \left\Vert g \right\Vert_{ r_{ 1}+1} \left(\sup_{ n}\sup_{ s\leq T} \mathbb{ E}_{ 0}\mathbf{ E} \left[\left\Vert \hat{ \eta}^{ n}_{ s} \right\Vert_{ -r_{ 1}}^{ 1+ \alpha} \right]^{ \frac{ 1}{ 1+ \alpha}} \right)\int_{ 0}^{t}  \mathbb{ E}_{ 0}\mathbf{ E}\left[d_{ BL} \left( \mu_{ s}^{ n}, \mu_{ s}\right)^{ \frac{ 1+ \alpha}{ \alpha}}\right]^{ \frac{ \alpha}{ 1+ \alpha}} {\rm d}s\\
&\leq C_{ T} \left\Vert g \right\Vert_{ r_{ 1}+1}\int_{ 0}^{t}  \mathbb{ E}_{ 0}\mathbf{ E} \left[d_{ BL} \left( \mu_{ s}^{ n}, \mu_{ s}\right)^{ \frac{ 1+ \alpha}{ \alpha}} \right]^{ \frac{ \alpha}{ 1+ \alpha}} {\rm d}s,
\end{align*}
which goes to $0$ as $n\to\infty$, by \eqref{eq:conv_empmeas} applied for $ q= \frac{ 1+ \alpha}{ \alpha}\geq1$ (recall Remark~\ref{rem:conv_empmeas}). Now turn to the first equation in \eqref{eq:limit_etas}: for all $f\in H^{ r_{ 1}+1} \left(\mathbb{ T}\right)$, the mapping $ \alpha \mapsto \left\langle \alpha_{ \cdot}\, ,\, \Theta(f)\right\rangle$ is continuous from $ \mathcal{ C}([0, T], H^{ -r_{ 1}}(\mathbb{ T}^{ 2})$ to $ \mathcal{ C} \left([0, T], \mathbb{ R}\right)$ so that we have the convergence of $\int_{ 0}^{t} \Theta^{ \ast} \hat{ \eta}^{ n}_{ s} {\rm d}s$ towards $\int_{ 0}^{t} \Theta^{ \ast} \hat{ \eta}_{ s} {\rm d}s$ as $n\to\infty$. Decomposing in a same way as before the drift as $\left\langle \eta_{ s}^{ n}\, ,\,  \mathcal{ L}_{ \mu^n_s}^{(1)}(f)\right\rangle= \left\langle \eta_{ s}^{ n}\, ,\,  \mathcal{ L}_{ \mu_s}^{(1)}(f)\right\rangle + \left\langle \eta_{ s}^{ n}\, ,\,  \mathcal{ L}_{ \mu^n_s}^{(1)} - \mathcal{ L}_{ \mu_s}^{ (1)}(f)\right\rangle$, we see once again that is suffices to prove that $ \mathbb{ E}_{ 0}\mathbf{ E} \left[ \int_{ 0}^{t} \left\vert \left\langle \eta_{ s}^{ n}\, ,\,  \mathcal{ L}_{ \mu^n_s}^{(1)} - \mathcal{ L}_{ \mu_s}^{ (1)}(f)\right\rangle \right\vert{\rm d}s\right] \xrightarrow[ n\to\infty]{}0$ for fixed $f\in H^{ r_{ 1}+1} \left(\mathbb{ T}\right)$, which follows from the very same calculation as before, using again the fact that $\sup_{ s\leq T} \mathbb{ E}_{ 0}\mathbf{ E} \left[d_{ BL} \left( \mu_{ s}^{ n}, \mu_{ s}\right)^{ \frac{ 1+ \alpha}{ \alpha}} \right] \xrightarrow[ n\to\infty]{}0$ (recall \eqref{eq:conv_empmeas}). This concludes the proof of Proposition~\ref{prop:ident_lim}.
\end{proof}
The proof of Theorem~\ref{th:limit eta and hat eta} is then complete provided we prove the following uniqueness result:
\begin{proposition}
\label{prop:unique_SPDE}
Under the assumptions of Theorem~\ref{th:limit eta and hat eta}, there is pathwise uniqueness (as well as uniqueness in law) of a solution $ \left(\eta, \hat{ \eta}\right)$ in $ \mathcal{ C} \left([0, T], H^{ -r_{ 1}} \left(\mathbb{ T}\right)\oplus H^{-r_{ 1}}(\mathbb{ T}^{ 2})\right)$ to the coupled system \eqref{eq:limit_etas}.
\end{proposition}
Proof of Proposition~\ref{prop:unique_SPDE}, which relies heavily on classical estimates (see e.g. \cite{Jourdain1998,MR876080,mitoma85,lucon_stannat_2016} for similar techniques) is postponed to Appendix~\ref{sec:uniqueness}.

\subsection{Particular cases}
\label{sec:proofs_examples}
The point of this paragraph is to prove the results of Section~\ref{sec:examples}.
\subsubsection{Proof of Proposition~\ref{prop:limit_hateta0_indep}}
Recall that we suppose Assumption~\ref{ass:theta0_indep_xi} and that for any test function $g$
\begin{align*}
\left\langle \hat{ \eta}_{ 0}^{ n}\, ,\, g\right\rangle&= \frac{ 1}{ n^{ 3/2}} \sum_{ i,j=1}^{ n} \hat{ \xi}_{ij}^{ (n)} g(\theta_{ 0}^{ i,n}, \theta_{ 0}^{ j, n}).
\end{align*}
Considering $(e_p)_{p\in \bbZ}$ the canonical orthonormal basis of $L^2(\bbT)$, the family defined by $\left(\psi_{p,q}=(1+p^2+q^2)^{-\frac{r}{2}}e_p\otimes e_p\right)_{p,q\in \bbZ}$ constitutes an orthonormal basis of $H^{-r}(\bbT^2)$. Then
\begin{align*}
\mathbb{ E}_{ 0} \left[\left\Vert \hat{ \eta}_{ 0}^{ n} \right\Vert^{ 2}_{ -r}\right]&= \sum_{ p,q\in \bbZ}  \mathbb{ E}_{ 0} \left[ \left\vert \left\langle \hat{ \eta}_{ 0}^{ n}\, ,\, \psi_{ p,q}\right\rangle \right\vert^{ 2}\right]\nonumber\\
&=  \frac{ 1}{ n^{ 3}} \sum_{ i,j, k,l} \hat{ \xi}_{ij}^{ (n)} \hat{ \xi}_{kl}^{ (n)} \mathbb{ E}_{ 0} \left[\sum_{ p,q\in \bbZ}\psi_{ p,q}(\theta_{ 0}^{ i,n}, \theta_{ 0}^{ j, n})\bar\psi_{ p,q}(\theta_{ 0}^{ k,n}, \theta_{ 0}^{ l, n})\right].
\end{align*}
The term
\begin{equation*}
H_{i,j,k,l} :=  \mathbb{ E}_{ 0} \left[\sum_{ p,q\in \bbZ}\psi_{ p,q}(\theta_{ 0}^{ i,n}, \theta_{ 0}^{ j, n})\bar\psi_{ p,q}(\theta_{ 0}^{ k,n}, \theta_{ 0}^{ l, n})\right]
\end{equation*}
satisfies (recall $r>1$), 
\begin{equation*}
\left|H_{i,j,k,l} \right|\leq \sum_{p,q\in \bbZ} (1+p^2+q^2)^{-r}=:C_r.
\end{equation*}
Defining
\begin{equation*}
S_n = \frac{ 1}{ n^{ 3}} \sum_{ i,j, k,l \in \{1,\ldots, n\}, (i,j)\neq (k,l)} \hat{ \xi}_{ij}^{ (n)} \hat{ \xi}_{kl}^{ (n)}  H_{i,j,i,j},
\end{equation*}
we get the decomposition
\begin{equation*}
\bbE_0\left[\left\Vert \hat{ \eta}_{ 0}^{ n} \right\Vert^{ 2}_{ -r}\right] - S_n = \frac{ 1}{ n^{ 3}} \sum_{ i,j=1}^n \left(\hat{ \xi}_{ij}^{ (n)}\right)^2  H_{i,j,i,j}.
\end{equation*}
Since the random variables $V_{ij}=\frac{1}{n^3}\left(\left(\hat{ \xi}_{ij}^{ (n)}\right)^2  H_{i,j,i,j} - \frac{1-p_n}{p_n} H_{i,j,i,j}\right)$ satisfy $\left|V_{ij}\right|\leq \frac{C'_r}{n^3 p_n 2}$ and $\mathrm{Var}\left(V_{ij}\right)\leq \frac{C'_r}{n p_n^3}$, applying Bernstein's inequality leads to the bound
\begin{equation*}
\bbP\left(\left|\sum_{i,j=1}^n V_{ij}\right|>t\right)\leq 2\exp \left(-\frac12\frac{-t^2}{\frac{C'_r}{np_n^3}+\frac{t C'_r}{3 n^3p_n^2}}\right)=2\exp \left(-\frac12\frac{-np_n^3 t^2}{C'_r+\frac{t C'_r p_n}{3n^2}}\right).
\end{equation*}
So, choosing $t=p_n^{-\frac32}n^{-\frac12+\gd}$ for $\gd$ small and remarking that  $\sum_{i,j=1}^n \frac{1-p_n}{n^3 p_n} H_{i,j,i,j}\leq p_n^{-\frac32}n^{-\frac12+\gd}$ for $n$ large enough, we obtain, for some positive constant $c$,
\begin{equation*}
\bbP\left(\left|\bbE_0\left[\left\Vert \hat{ \eta}_{ 0}^{ n}\right\Vert^{ 2}_{ -r} \right] - S_n\right|> 2p_n^{-\frac32}n^{-\frac12+\gd}\right)\leq e^{-c n^{2\gd}}.
\end{equation*}
Now, relying on the decoupling inequality provided by \cite{de1995decoupling}, we get that for some constant $C$ and for $\hat{ \xi}_{ij}^{ (n,1)},\hat{ \xi}_{ij}^{ (n,2)}$ two independent copies of $\hat{ \xi}_{ij}^{ (n)}$,
\begin{equation*}
\bbP\left(\left|S_n\right|\geq t \right)\leq
C  \bbP\left(\sum_{i=1}^n \left| \frac{1}{n^3} \sum_{ i,j, k,l \in \{1,\ldots, n\}, (i,j)\neq (k,l)} \hat{ \xi}_{ij}^{ (n,1)} \hat{ \xi}_{kl}^{ (n,2)}  H_{i,j,k,l}\right|\geq \frac{t}{C} \right).
\end{equation*}
Consider the filtration $\cF_{l,m}=\gs\left(\hat{ \xi}^{ (n,1)},\hat{ \xi}_{1,1}^{ (n,2)},\ldots,\hat{ \xi}_{lm}^{ (n,2)}\right)$. Denoting
\begin{equation*}
X_{k,l}(x,z) =\frac{1}{n^3} \left(\sum_{i,j\in \{1,\ldots,n\}: (i,j)\neq (k,l)}^n H_{i,j,k,l} x_{ij}\right)z,
\end{equation*}
the process
\begin{equation*}
Y_{m_1,m_2} = \sum_{(k,l):(k,l)\leq (m_1,m_2)} X_{k,l}\left(\hat{ \xi}^{ (n,1)}, \hat{ \xi}_{kl}^{ (n,2)}\right)
\end{equation*}
is a martingale which satisfies $Y_{n,n} = \frac{1}{n^3} \sum_{ i,j, k,l \in \{1,\ldots, n\}, (i,j)\neq (k,l)} \hat{ \xi}_{ij}^{ (n,1)} \hat{ \xi}_{kl}^{ (n,2)}  H_{i,j,k,l}$.
It satisfies moreover
\begin{equation*}
\left|X_{k,l}\left(\hat{ \xi}^{ (n,1)}, \hat{ \xi}_{kl}^{ (n,2)}\right)\right|\leq \frac{C_r}{np_n^2},
\end{equation*}
and
\begin{align*}
\left\langle Y\right\rangle_{n,n}& = \sum_{m_1,m_2=1}^n \bbE\left[\left(X_{k,l}\left(\hat{ \xi}^{ (n,1)}, \hat{ \xi}_{kl}^{ (n,2)}\right)\right)^2|\cF_{m_1,m_2-1}\right]\nonumber\\
&=\frac{1}{n^2}\left(\frac{1}{p_n}-1\right) \sum_{m_1,m_2=1}^n \left(\frac{1}{n^2}\sum_{i,j\in \{1,\ldots,n\}: (i,j)\neq (k,l)}^n H_{i,j,k,l} \hat{ \xi}_{ij}^{ (n,1)}\right)^2.
\end{align*}
Applying Bernstein's inequality we get
\begin{equation*}
\begin{split}
	\bbP\left(\left|\frac{1}{n^2}\sum_{i,j\in \{1,\ldots,n\}: (i,j)\neq (k,l)}^n H_{i,j,k,l} \hat{ \xi}_{ij}^{ (n,1)}\right|>t\right)&\leq 2\exp\left(-\frac12 \frac{t^2}{\frac{C_r}{p_n n^2}+\frac{t C_r}{3n^2 p_n}}\right)\\
	&=2\exp\left(-\frac12 \frac{p_n n^2 t^2}{C_r+ \frac{t C_r}{3}}\right),
\end{split}
\end{equation*}
which means that, taking $t=p_n^{-\frac12}n^{-1+\gd}$ for $\gd$ small, we get, for some positive constant $c$,
\begin{equation*}
\bbP\left(\left\langle Y\right\rangle_{n,n}> p_n^{-\frac32}n^{-1+\gd}\right)\leq n^2e^{-cn^{2\gd}}.
\end{equation*}
We can now apply Bernstein's inequality for martingales (see \cite[Th.~1.6]{Freedman1975} or \cite[Th.~1.1]{Tropp2011}):
\begin{align*}
\mathbb{ P} \left( \left\vert Y_{ n,n} \right\vert>t\right) &\leq \mathbb{ P} \left(  \left\langle Y\right\rangle_{ n,n}> p_n^{-\frac12}n^{-1+\gd}\right)+ \mathbb{ P} \left( \left\vert Y_{ n,n} \right\vert> t,\ \left\langle Y\right\rangle_{n, n}\leq p_n^{-\frac12}n^{-1+\gd}\right) \nonumber\\
&\leq n^2e^{-cn^{2\gd}}+ 2 \exp \left(- \frac{ 1}{ 2} \frac{ t^{ 2} }{ p_n^{-\frac12}n^{-1+\gd}+ \frac{  t C_r p_n^{-2}n^{-1}}{ 3}}\right).
\end{align*}
Choosing $t=p_n^{-\frac14} n^{-\frac12+\gd}$ we obtain
\begin{equation*}
\mathbb{ P} \left( \left\vert Y_{ n,n} \right\vert>t\right) \leq n^2e^{-cn^{2\gd}}+ 2 \exp \left(- \frac{ 1}{ 2} \frac{ n^\gd }{ C_r+ \frac{  C_r}{ 3p_n^\frac74 n^\frac12}}\right),
\end{equation*}
which concludes the proof of Proposition~\ref{prop:limit_hateta0_indep}.
\subsection{A case where the initial condition depends on the graph}
\label{ss:initial_cond_on_graph}

	The aim is to construct an example of initial condition that depends on the graph and for which $\hat \eta^n_0$ has a non trivial limit in distribution. We place ourselves in the case where $\Gamma$ is given by $\Gamma(\theta,\theta')=-\sin(\theta-\theta')$ and the graph sequence is a symmetric Erd\H{o}s-R\'enyi with $p=\frac 12$. In particular, the variables $\xi^{(n)}_{ij}$ are such that $\xi^{(n)}_{ij} = \xi^{(n)}_{ji}$ for every $1\leq i, j \leq n$. We further suppose that they do not depend on $n$, i.e., $\xi^{(n)}_{ij} = \xi_{ij}$ for every $i, j$ and $n$. Let $\left(\cG_n\right)_{n\geq 1}$ be the corresponding filtration with $\cG_n=\gs\left(\left( \xi_{ij}\right)_{1\leq i,j\leq n}\right)$.
	
	Let the sequence $(\theta^i_0)_{i\geq 1}$ be defined by recursion with values in $\{0,\frac{\pi}{2}\}$. We initiate the recursion by choosing $\theta^1_0$ uniformly in $\{0,\frac{\pi}{2}\}$. Then, supposing that $(\theta^1_0,\ldots,\theta^n_0)$ are already defined (and $\cG_n$-measurable), we consider the $\cG_{n+1}$-measurable random variables
	\begin{align*}
		R^n_0&=\sum_{i\leq n:\, \theta^i_0=0} \hat \xi_{i\, n+1},\\
		R^n_{\frac{\pi}{2}}&=\sum_{i\leq n:\, \theta^i_0=\frac{\pi}{2}} \hat \xi_{i\, n+1},
	\end{align*}
	and make the choice $\theta^{n+1}_0=0$ when $R^n_{\frac{\pi}{2}} > R^n_0$, $\theta^{n+1}_0=\frac{\pi}{2}$ when $R^n_{\frac{\pi}{2}} < R^n_0$, and $\theta^{n+1}_0=U^k$ with $U^k$ chosen uniformly in $\left\{0,\frac{\pi}{2}\right\}$ (independently from $\cG_n$) when $R^n_{\frac{\pi}{2}} = R^n_0$.
	
	Remark that by symmetry of the laws of $R^n_0$ and $R^n_{\frac \pi 2}$ (recall that $p=\frac12$), $\theta^{n+1}_0$ is a uniform random variable in $\{0, \frac\pi 2\}$ independent from $\cG_n$. Hence, the sequence $(X_n)_{n\geq 1}$ defined by $X_n=\sum_{i=1}^n \mathrm{1}_{\{\theta^i_0=0\}}$ is a symmetric simple random walk on $\mathbb{Z}$.
	Observe that the laws of $R^n_0$ and $R^n_{\frac \pi 2}$ depend on $\cG_n$ only via $X_n$. Conditionally on $\{X_n = x\}$, $\frac{R^n_0 + x}{2}$ is a Binomial random variable with parameters $x$ and $\frac12$ and $\frac{R^n_{\frac \pi 2} + n-x}{2}$ is a Binomial with parameters $n-x$ and $\frac12$, independent of $R^n_0$.
	
	Since $\mu^n_0=\frac{X_n}{n}\gd_0+\frac{n-X_n}{n}\gd_{\frac\pi2}$, we have $d_{BL}(\mu^n_0,\mu_0)\underset{n\rightarrow\infty}{\rightarrow} 0$ $\bbP-$a.s., where $\mu_0=\frac{1}{2}\gd_0+\frac12\gd_{\frac{\pi}{2}}$, while $\eta_0^n=\sqrt{n}\left(\frac{X_n}{n}-\frac12\right)\gd_0+\sqrt{n}\left(\frac{n-X_n}{n}-\frac12\right)\gd_{\frac\pi2}$ converges in law in $H^{-r_1}$ to $\eta_0=Z_1 \gd_0+Z_2\gd_{\frac{\pi}{2}}$, where $(Z_1,Z_2)$ has centered Gaussian distribution of covariance matrix $\left(\begin{array}{cc}
		\frac14 &-\frac14\\-\frac14 &\frac14
	\end{array}\right)$.
	
	The process $\hat{\eta}^n_0$ can be expressed in terms of $R^n_0$, $R^n_{\frac \pi 2}$ as follows:
	\begin{equation*}
		\begin{split}
			\langle \hat\eta^n_0,g\rangle = &\frac{2}{n^\frac32}\sum_{k=1}^{n-1} R^k_0 \left(\mathrm{1}_{\left\{R^k_{\frac{\pi}{2}}> R^k_0 \right\}}+\mathrm{1}_{\left\{R^k_{\frac{\pi}{2}}= R^k_0, \, U^k=0 \right\}}\right)g(0,0)\\
			&+\frac{1}{n^\frac32}\sum_{k=1}^{n-1}  \max\left(R^k_{\frac{\pi}{2} }, R^k_{ 0 }\right) \left( g\left(\frac{\pi}{2},0\right) + g\left(0, \frac \pi 2\right) \right)\\
			&+\frac{2}{n^\frac32}\sum_{k=1}^{n-1} R^k_{\frac{\pi}{2}} \left(\mathrm{1}_{\left\{R^k_{\frac{\pi}{2}  }< R^k_{ 0 }\right\}}+\mathrm{1}_{\left\{R^k_{\frac{\pi}{2}}= R^k_0, \, U^k=\frac{\pi}{2} \right\}}\right)g\left(\frac{\pi}{2},\frac{\pi}{2}\right),
		\end{split}
	\end{equation*}
	and
	\begin{equation*}
		\langle\Gamma *\hat\eta^n_0,f\rangle = \frac{1}{n^\frac32}\sum_{k=1}^{n-1} \max\left(R^k_{\frac{\pi}{2} }, R^k_{0 }\right) \left( f(0) - f\left(\frac{\pi}{2}\right)\right).
	\end{equation*}
	Let us study the convergence of these different terms. Since we have $\bbE\left[\left(R^k_{0 }\right)^2\right]\leq  \frac{n}{4}$ and $\bbE\left[\left(R^k_{\frac{\pi}{2} }\right)^2\right]\leq  \frac{n}{4}$, we get
	\begin{align*}
		\bbE&\Bigg[\Bigg( \frac{1}{n^\frac32}\sum_{k=1}^{n-1} R^k_{  0 }\left(\mathrm{1}_{\left\{R^k_{\frac{\pi}{2}}> R^k_0 \right\}}+\mathrm{1}_{\left\{R^k_{\frac{\pi}{2}}= R^k_0, \, U^k=0 \right\}}\right)\nonumber\\
		&\qquad \qquad\qquad \qquad \qquad - \frac{1}{n^\frac32}\sum_{k=1}^{n-1}\bbE\left[ R^k_{ 0 }\left(\mathrm{1}_{\left\{R^k_{\frac{\pi}{2}}> R^k_0 \right\}}+\mathrm{1}_{\left\{R^k_{\frac{\pi}{2}}= R^k_0, \, U^k=0 \right\}}\right)|\cG_k\right]\Bigg)^2\Bigg]\nonumber\\
		&\leq \frac{1}{n^3}\sum_{k=1}^{n-1}\bbE\Bigg[\Bigg(R^k_{ 0 }\left(\mathrm{1}_{\left\{R^k_{\frac{\pi}{2}}> R^k_0 \right\}}+\mathrm{1}_{\left\{R^k_{\frac{\pi}{2}}= R^k_0, \, U^k=0 \right\}}\right)\nonumber\\
		&\qquad \qquad\qquad \qquad \qquad -\bbE\left[ R^k_{ 0 }\left(\mathrm{1}_{\left\{R^k_{\frac{\pi}{2}}> R^k_0 \right\}}+\mathrm{1}_{\left\{R^k_{\frac{\pi}{2}}= R^k_0, \, U^k=0 \right\}}\right)|\cG_k\right]\Bigg)^2\Bigg]\nonumber\\
		&\leq \frac{4}{n^3}\sum_{k=1}^{n-1}\bbE\left[\left(R^k_{ 0 }\mathrm{1}_{\left\{R^k_{\frac{\pi}{2}  }\geq R^k_{ 0 }\right\}}\right)^2\right]\leq \frac{1}{n},
	\end{align*}
	and similarly
	\begin{align*}
		\bbE&\left[\left(\frac{1}{n^\frac32}\sum_{k=1}^{n-1}  \max\left(R^k_{\frac{\pi}{2}  }, R^k_{ 0 }\right)- \frac{1}{n^\frac32}\sum_{k=1}^{n-1}\bbE\left[\max\left(R^k_{\frac{\pi}{2}  }, R^k_{ 0 }\right)|\cG_k\right] \right)^2\right]\nonumber\\
		&\leq \frac{4}{n^3}\sum_{k=1}^{n-1}\left(\bbE\left[\left(R^k_{\frac{\pi}{2}  }\right)^2\right]+\bbE\left[\left(R^k_{  0}\right)^2\right] \right)\leq \frac{1}{n},
	\end{align*}
and
\begin{align*}
		\bbE&\Bigg[\Bigg( \frac{1}{n^\frac32}\sum_{k=1}^{n-1} R^k_{  \frac{\pi}{2} }\left(\mathrm{1}_{\left\{R^k_{\frac{\pi}{2}}< R^k_0 \right\}}+\mathrm{1}_{\left\{R^k_{\frac{\pi}{2}}= R^k_0, \, U^k=0 \right\}}\right)\nonumber\\
		&\qquad \qquad\qquad \qquad \qquad - \frac{1}{n^\frac32}\sum_{k=1}^{n-1}\bbE\left[ R^k_{ 0 }\left(\mathrm{1}_{\left\{R^k_{\frac{\pi}{2}}< R^k_0 \right\}}+\mathrm{1}_{\left\{R^k_{\frac{\pi}{2}}= R^k_0, \, U^k=\frac{\pi}{2} \right\}}\right)|\cG_k\right]\Bigg)^2\Bigg]\leq \frac{1}{n}.	
\end{align*}
	Hence, it remains to study the convergence of the terms
	\begin{align*}
		N^1_n &:= \frac{1}{n^\frac32}\sum_{k=1}^{n-1}\bbE\left[ R^k_{ 0 }\left(\mathrm{1}_{\left\{R^k_{\frac{\pi}{2}}> R^k_0 \right\}}+\mathrm{1}_{\left\{R^k_{\frac{\pi}{2}}= R^k_0, \, U^k=0 \right\}}\right)|\cG_k\right],\\
		N^2_n &:= \frac{1}{n^\frac32}\sum_{k=1}^{n-1}\bbE\left[\max\left(R^k_{\frac{\pi}{2} \cdot }, R^k_{\cdot 0 }\right)|\cG_k\right],\\
	N^3_n&:= \frac{1}{n^\frac32}\sum_{k=1}^{n-1}\bbE\left[ R^k_{\frac{\pi}{2} }\left(\mathrm{1}_{\left\{R^k_{\frac{\pi}{2}}< R^k_0 \right\}}+\mathrm{1}_{\left\{R^k_{\frac{\pi}{2}}= R^k_0, \, U^k=\frac{\pi}{2} \right\}}\right)|\cG_k\right].
	\end{align*}
	Recall that, conditionally to $\{X_n=x\}$ and up to a scaling factor $1/2$, $R^n_{0}$ and $R^n_{\frac \pi 2}$ are (centered) Binomial random variables. Since $X_n/n$ converges almost surely to $\frac12$, one can apply the classical Normal approximation of the binomial distribution and approximate $\frac{R^n_{0}}{\sqrt{n/2}}$ by a centered Normal random variable. This yields to
	\begin{equation*}
		N^1_n = \frac{1}{\sqrt{2}} \frac{1}{n}\sum_{k=1}^{n-1}\sqrt{\frac{k}{n}}\bbE\left[ \frac{\sqrt{2} R^k_{ 0 }}{\sqrt{k}}\left(\mathrm{1}_{\left\{R^k_{\frac{\pi}{2}}> R^k_0 \right\}}+\mathrm{1}_{\left\{R^k_{\frac{\pi}{2}}= R^k_0, \, U^k=0 \right\}}\right)|\cG_k\right]\overset{a.s.}{\underset{n\rightarrow \infty}{\longrightarrow}} \frac L{\sqrt{2}} \int_{[0,1]} \sqrt{x} \dd x,
	\end{equation*} 
	where, considering two independent random variables $Y_1,Y_2$ with standard normal distribution,
	\begin{equation*}
		L := \bbE\left[Y_1\mathrm{1}_{\{Y_2\geq Y_1\}}\right]=-\frac{1}{2\sqrt{\pi}}.
	\end{equation*}
	This means that $N^1_n$ converges almost surely to $-\frac{1}{3\sqrt{2\pi}}$. Via similar arguments, one can prove that
	\begin{equation*}
		N^2_n \overset{a.s.}{\underset{n\rightarrow \infty}{\longrightarrow}} \frac{1}{\sqrt{2}} \bbE\left[\max \left(Y_1, Y_2\right)\right]\int_{[0,1]} \sqrt{x} \dd x = \frac{2}{3\sqrt{2\pi}},
	\end{equation*} 
and
\begin{equation*}
N^3_n\overset{a.s.}{\underset{n\rightarrow \infty}{\longrightarrow}}-\frac{1}{3\sqrt{2\pi}}.
\end{equation*}
	Finally, we deduce that in this particular example $\hat \eta^n_0$ convergences law in $H^{-r_1}(\bbT^2)$ to $\hat \eta_0=\frac{2}{3\sqrt{\pi}}\left(-\gd_{(0,0)}+\gd_{\left(0,\frac{\pi}{2}\right)}+\gd_{\left(\frac{\pi}{2},0\right)}-\gd_{\left(\frac{\pi}{2},\frac{\pi}{2}\right)}\right)$, while $\Gamma*\hat \eta^n_0$ converges in law to $\frac{2}{3\sqrt{2\pi}}\left(\gd_{0}-\gd_{\frac{\pi}{2}}\right)$ in $H^{-r_1}(\bbT)$.

\section{ About fluctuations of local empirical measures}
\label{sec:proof_local_fluct}
The purpose of the present section is to prove Theorem~\ref{th:local_fluct}. Recall the definition \eqref{def:mu_nl} of the empirical measure $\mu_{ t}^{ n,l}$ of particles at distance $1$ of vertex $l=1,2$
and the definition \eqref{def:mu_n12}  of the empirical measure $\mu_{ t}^{ n, 1,2}$ of particles that are connected to both vertices $1$ and $2$.
Recall that we are interested in the behavior of the joint fluctuation process  \eqref{def:zetas}
\begin{equation*}
\zeta_{ t}^{ n}= \left(\zeta_{ t}^{ n, 1}, \zeta_{ t}^{ n, 2}\right)= \left( \sqrt{ n p_{ n}} \left( \mu_{ t}^{ n, 1} - \mu_{ t}\right), \sqrt{ n p_{ n}} \left( \mu_{ t}^{ n, 2} - \mu_{ t}\right)\right).
\end{equation*}
As we will see below, we will actually need to incorporate also the dynamics of the global fluctuation process $ \eta_{ t}^{ n}$, that is to look at the joint convergence of $ \left(\zeta_{ t}^{ n,1}, \zeta_{ t}^{ n, 2}, \eta_{ t}^{ n}\right)$ as $n\to\infty$.
\subsection{The initial condition}
We first address the convergence of the initial condition $ \left(\zeta_{ 0}^{ n,1}, \zeta_{ 0}^{ n, 2}, \eta_{ 0}^{ n}\right)$.
\begin{proposition}
\label{prop:ident_cond0}
Suppose that $ (\theta_{ 0}^{ i, n})$ are i.i.d. random variable with law $ \mu_{ 0}$, independent of the graph. Recall that $p:= \lim_{ n\to\infty} p_{ n}\in [0, 1]$. Then, if $n p_{ n}^{ 2}\to \infty$, for all $r> \frac{ 1}{ 2}$,
\begin{equation}
\sup_{ n}\mathbb{ E} \left[ \left\Vert \zeta_{ 0}^{ n, l} \right\Vert_{ -r}^{ 2}\right]<\infty,
\end{equation}
and $ \left(\zeta_{ 0}^{n,1}, \zeta_{ 0}^{n,2}, \eta_{ 0}^{ n}\right)$ converges in law as $n\to\infty$ in $H^{ -(r+1/2)} \left(\mathbb{ T}\right)^{ 3}$ to the Gaussian process $(\zeta_{ 0}^{ 1}, \zeta_{ 0}^{ 2}, \eta_{ 0})$ with covariance
\begin{equation}
\label{eq:cov_zetas_0}
\begin{split}
C \left( \left(\begin{smallmatrix} f_{ 1}\\ g_{ 1}\\ h_{ 1}\end{smallmatrix}\right), \left(\begin{smallmatrix} f_{ 2}\\ g_{ 2}\\ h_{ 2}\end{smallmatrix}\right)\right)&:= C_{ \zeta_{ 0}}( f_{ 1}, f_{ 2}) + C_{ \zeta_{ 0}}( g_{ 1}, g_{ 2}) + C_{ \eta_{ 0}}( h_{ 1}, h_{ 2})\\
&\quad+ C_{ \zeta_{ 0}^{ 1},\zeta_{ 0}^{ 2}}( f_{ 1}, g_{ 2}) + C_{ \zeta_{ 0}^{ 1},\zeta_{ 0}^{ 2}}( f_{ 2}, g_{ 1}) + C_{ \zeta_{ 0}, \eta_{ 0}}( f_{ 1}, h_{ 2}) + C_{ \zeta_{ 0}, \eta_{ 0}}( f_{ 2}, h_{ 1}) \\&\quad + C_{ \zeta_{ 0},\eta_{ 0}}( g_{ 1}, h_{ 2}) + C_{ \zeta_{ 0},\eta_{ 0}}( g_{ 2}, h_{ 1}),
\end{split}
\end{equation}
where for $ {\rm Cov}_{ \mu_{ 0}}(f,g)= \int \left(f - \int f {\rm d}\mu_{ 0}\right)\left( g - \int g {\rm d}\mu_{ 0}\right) {\rm d}\mu_{ 0}$,
\begin{equation}
\label{eq:covariances_0_cross}
\begin{split}
C_{ \zeta_{ 0}}( f, g)&:=  {\rm Cov}_{ \mu_{ 0}}(f,g) +(1-p)\left(\int f {\rm d}\mu_{ 0}\right) \left(\int g {\rm d}\mu_{ 0}\right),\\
C_{ \eta_{ 0}}(f,g)&:= {\rm Cov}_{ \mu_{ 0}}(f,g),\\
C_{ \zeta_{ 0}^{ 1}, \zeta_{ 0}^{ 2}}(f,g)&:= p {\rm Cov}_{ \mu_{ 0}}(f,g),\\
C_{ \zeta_{ 0}, \eta_{ 0}}(f,g)&:= \sqrt{ p} {\rm Cov}_{ \mu_{ 0}}(f,g),
\end{split}
\end{equation}
where $ f_{ i}$, $g_{ i}$, $h_{ i}$, $i=1,2$ are test functions on $ \mathbb{ T}$. In particular, $(\zeta_{ 0}^{ 1}, \zeta_{ 0}^{ 2}, \eta_{ 0})$ are mutually independent in the diluted case $p=0$.
\end{proposition}
\begin{proof}[Proof of Proposition~\ref{prop:ident_cond0}]
We have, setting $ \bar \psi := \psi - \int \psi {\rm d} \mu_{ 0}$,
\begin{align*}
\left\Vert \zeta_{ 0}^{ n, l} \right\Vert_{ -r}^{ 2} &= \sum_{ p\geq1} \left\vert \left\langle \zeta_{ 0}^{ n, l}\, ,\, \psi_{ p}\right\rangle \right\vert^{ 2} = np_{ n} \sum_{ p\geq 1} \left\vert \frac{ 1}{ np_{ n}} \sum_{ i=1}^{ n} \xi_{li}^{ (n)}\bar\psi_{ p} \left(\theta_{ 0}^{ i, n}\right) + \frac{ 1}{ n} \sum_{ i=1}^{ n} \hat{ \xi}_{li}^{ (n)} \int \psi_{ p} {\rm d} \mu_{ 0}\right\vert^{ 2}\\
&\leq 2np_{ n} \sum_{ p\geq 1} \left\vert \frac{ 1}{ np_{ n}} \sum_{ i=1}^{ n} \xi_{li}^{ (n)}\bar\psi_{ p} \left(\theta_{ 0}^{ i, n}\right)\right\vert^{ 2} + 2np_{ n}  \left\vert \frac{ 1}{ n} \sum_{ i=1}^{ n} \hat{ \xi}_{li}^{ (n)} \right\vert^{ 2} \left\Vert \mu_{ 0} \right\Vert_{ -r}^{ 2}.
\end{align*}
Taking the expectation (w.r.t. both graph and initial condition), we obtain
\begin{align*}
\mathbb{ E} \left[\left\Vert \zeta_{ 0}^{ n, l} \right\Vert_{ -r}^{ 2}\right] & \leq 2(1-p_{ n}) \left(\sum_{ p\geq1} \int \left(\bar\psi_{ p}(\theta)\right)^{ 2} \mu_{ 0} \left({\rm d}\theta\right) + \left\Vert \mu_{ 0} \right\Vert_{ -r}^{ 2}\right).
\end{align*}
Recalling that $\psi_{p}=(1+p^2)^{-\frac{r}{2}}e_p$, this last quantity is bounded provided $r> \frac{ 1}{ 2}$. Hence $ ( \zeta_{ 0}^{ n ,l})$ is tight in $ H^{ -(r+1/2)} \left(\mathbb{ T}\right)$ and the triplet $ \left(\zeta_{ 0}^{ n, 1}, \zeta_{ 0}^{ n, 2}, \eta_{ 0}^{ n}\right)$ has convergent subsequences in $ H^{ -(r+1/2)} \left(\mathbb{ T}\right)^{ 3}$. It suffices to identify its finite dimensional marginals: let $u,v,w\in \mathbb{ R}$ and $f,g,h$ test functions. Define
\begin{align*}
\varphi_{ n}(u,v,w)&:= \mathbb{ E} \left[ e^{ i u \left\langle \zeta_{ 0}^{ n , 1}\, ,\, f\right\rangle + i v \left\langle \zeta_{ 0}^{ n , 2}\, ,\, g\right\rangle + i w \left\langle \eta_{ 0}^{ n}\, ,\, h\right\rangle}\right],\\
X_{ n}^{ (l)} (f)&:=\frac{ 1}{ \sqrt{ np_{ n}}} \sum_{ j=1}^{ n} \xi_{lj}^{ (n)} \bar f \left( \theta_{ j, 0}\right),\\
Y_{ n}^{ (l)}(f)&:= \left(\frac{\sqrt{ np_{ n}}}{ n} \sum_{ j=1}^{ n} \hat{ \xi}_{lj}^{ (n)}\right)\int f {\rm d}\mu_{ 0},
\end{align*}
so that
\begin{align}
\varphi_{ n}&(u,v,w) \nonumber\\
&= \mathbb{ E} \left[ \exp \left(iu X_{ n}^{ (1)} (f) + iv X_{ n}^{ (2)} (g) + iw \left\langle \eta_{ 0}^{ n}\, ,\, h\right\rangle\right)\exp \left(iu Y_{ n}^{ (1)}(f) + iv Y_{ n}^{ (2)}(g)\right) \right] \nonumber\\
&= \mathbb{ E} \left[ \exp \left(iu Y_{ n}^{ (1)}(f) + iv Y_{ n}^{ (2)}(g)\right)\mathbb{ E} \left[\exp \left(iu X_{ n}^{ (1)} (f) + iv X_{ n}^{ (2)} (g) + iw \left\langle \eta_{ 0}^{ n}\, ,\, h\right\rangle\right) \Bigg \vert \mathcal{ F}_{ \xi}\right]\right],\label{eq:char_funct}
\end{align}
where $ \mathcal{ F}_{ \xi}$ is the $ \sigma$-field generated by the variables $ \left(\xi_{ij}^{ (n)}\right)$. For fixed $ \left(\xi_{lj}^{ (n)}\right)_{ j=1, \ldots, n}$, denote by $U_{ j}^{ (n)}:= \frac{ 1}{ \sqrt{ n}} \left(u \frac{ \xi_{1j}^{ (n)}}{ \sqrt{ p_{ n}}}\bar f \left(\theta_{ j, 0}\right) + v\frac{ \xi_{2j}^{ (n)} }{ \sqrt{p_{ n}}}\bar g \left(\theta_{ j, 0}\right)+ w \bar h \left(\theta_{ j,0}\right)\right):= \frac{ 1}{ \sqrt{ n}} \left(A_{ j}^{ (n)}+B_{ j}^{ (n)} + C_{ j}^{ (n)}\right)$. Then $ \mathbb{ E} \left[ U_{ j}^{ (n)}\vert \mathcal{ F}_{ \xi}\right]=0$ for all $j$ and define $s_{ n, U}^{ 2}:= \sum_{ j=1}^{ n} \mathbb{ E} \left[ \left(U_{ j}^{ (n)}\right)^{ 2} \vert \mathcal{ F}_{ \xi}\right]$. Then we have
\begin{align*}
s_{ n, U}^{ 2}&= \frac{ 1}{ n}\sum_{j=1}^{ n} \mathbb{ E} \left[ \left(A_{ j}^{ (n)} + B_{ j}^{ (n)} + C_{ j}^{ (n)}\right)^{ 2} \vert \mathcal{ F}_{ \xi}\right]\\
&=\frac{ 1}{ n} \sum_{ j=1}^{ n} \left(u^{ 2} \frac{ \xi_{1j}^{ (n)}}{ p_{ n}} {\rm Var}_{ \mu_{ 0}} f + v^{ 2}\frac{ \xi_{2j}^{ (n)} }{p_{ n}} {\rm Var}_{ \mu_{ 0}} g+ w^{ 2} {\rm Var}_{ \mu_{ 0}}h\right)\\
&\quad + \frac{ 2}{ n} \sum_{ j=1}^{ n}\left(uv \frac{ \xi_{1j}^{ (n)} \xi_{2 j}^{ (n)}}{ p_{ n}} {\rm Cov}_{ \mu_{ 0}}(f,g) + uw \frac{ \xi_{1j}^{ (n)}}{ \sqrt{ p_{ n}}} {\rm Cov}_{ \mu_{ 0}}(f,h) + vw \frac{ \xi_{2j}^{ (n)}}{ \sqrt{ p_{ n}}} {\rm Cov}_{ \mu_{ 0}}(g,h) \right).
\end{align*}
Note that, using \eqref{eq:sum_hat_xi_lj} and \eqref{eq:sum_hat_xi_12j}, we have that $ \mathbb{ P}$-a.s. $ s_{ n, U}^{ 2} \xrightarrow[ n\to\infty]{} s_{ U}^{ 2}:=s_{ U}^{ 2}(u,v,w)$, where
\begin{align}
s_{ U}^{ 2}:=& u^{ 2} {\rm Var}_{ \mu_{ 0}} f + v^{ 2} {\rm Var}_{ \mu_{ 0}} g+ w^{ 2} {\rm Var}_{ \mu_{ 0}}h\nonumber\\
&+ 2 \sqrt{ p}\big(uv \sqrt{ p}{\rm Cov}_{ \mu_{ 0}}(f,g) + uw {\rm Cov}_{ \mu_{ 0}}(f,h) + vw {\rm Cov}_{ \mu_{ 0}}(g,h) \big).\label{eq:s2U}
\end{align}
Fix some $ \delta>0$ and compute
\begin{align*}
\frac{ 1}{ s_{ n, U}^{ 2+ \delta}}\sum_{ j=1}^{ n} \mathbb{ E}\left[  \left\vert U_{j}^{ (n)} \right\vert^{ 2+ \delta} \vert \mathcal{ F}_{ \xi}\right]&= \frac{ 1}{ n^{ \frac{  \delta}{ 2}}}  \frac{ \frac{ 1}{ n}\sum_{ j=1}^{ n} \mathbb{ E} \left[ \left(A_{ j}^{ (n)} + B_{ j}^{ (n)} + C_{ j}^{ (n)}\right)^{ 2+ \delta} \vert \mathcal{ F}_{ \xi}\right]}{ \left(  \frac{ 1}{ n}\sum_{j=1}^{ n} \mathbb{ E} \left[ \left(A_{ j}^{ (n)} + B_{ j}^{ (n)} + C_{ j}^{ (n)}\right)^{ 2} \vert \mathcal{ F}_{ \xi}\right]\right)^{ \frac{ 2+  \delta}{ 2}}}.
\end{align*}
Applying H\"{o}lder inequality twice (to the expectation $ \mathbb{ E} \left[\cdot \vert \mathcal{ F}_{ \xi}\right]$ and to the discrete mean $ \frac{ 1}{ n} \sum_{ j=1}^{ n}$), we have,
\begin{align*}
\left(\frac{ 1}{ n}\sum_{ j=1}^{ n} \mathbb{ E} \left[ \left(A_{ j}^{ (n)} + B_{ j}^{ (n)} + C_{ j}^{ (n)}\right)^{ 2+ \delta} \vert \mathcal{ F}_{ \xi}\right]\right)^{ \frac{ 1}{ 2+ \delta}}&\leq \left(\frac{ 1}{ n}\sum_{ j=1}^{ n} \mathbb{ E} \left[ \left(A_{ j}^{ (n)} + B_{ j}^{ (n)} + C_{ j}^{ (n)}\right)^{ 2+ \delta} \vert \mathcal{ F}_{ \xi}\right]^{ \frac{ 2}{ 2+ \delta}}\right)^{ \frac{ 1}{ 2}}\\
& \leq \left(\frac{ 1}{ n}\sum_{ j=1}^{ n} \mathbb{ E} \left[ \left(A_{ j}^{ (n)} + B_{ j}^{ (n)} + C_{ j}^{ (n)}\right)^{ 2} \vert \mathcal{ F}_{ \xi}\right]\right)^{ \frac{ 1}{ 2}}.
\end{align*}
This means that $ \mathbb{ P}$-a.s., $\frac{ 1}{ s_{ n, U}^{ 2+ \delta}}\sum_{ j=1}^{ n} \mathbb{ E}\left[  \left\vert U_{j}^{ (n)} \right\vert^{ 2+ \delta} \vert \mathcal{ F}_{ \xi}\right]\leq n^{ - \frac{ \delta}{ 2}}$ which goes to $0$ as $n\to\infty$. The Lyapounov's condition for CLT is satisfied (see \cite[eq. (27.16), p.~385]{Billingsley1995}). We are in position to apply Th.~27.3, p.~385 of \cite{Billingsley1995}: $ \mathbb{ P}$-a.s., $ \frac{ 1}{ s_{ n, U}}\sum_{ j=1}^{ n} U_{ j}^{ (n)}$ converge in law to a standard Gaussian $ \mathcal{ N}(0,1)$, which gives that $u X_{ n}^{ (1)} (f) + v X_{ n}^{ (2)} (g) + w \left\langle \eta_{ 0}^{ n}\, ,\, h\right\rangle$ converges in law to $ \mathcal{ N}(0, s_{ U}^{ 2}(u,v,w))$ where $s_{ U}^{ 2}(u,v,w)$ is given by \eqref{eq:s2U}.

We use the same argument for the term $u Y_{ n}^{ (1)}(f) + v Y_{ n}^{ (2)}(g)$: if we denote $V_{ j}^{ (n)}:= u\frac{ \sqrt{ np_{ n}}}{ n} \hat{ \xi}_{1j}^{ (n)} \int f {\rm d}\mu_{ 0}+v\frac{ \sqrt{ np_{ n}}}{ n} \hat{ \xi}_{2j}^{ (n)} \int g {\rm d}\mu_{ 0}$, we have $ \mathbb{ E} \left[ V_{ j}^{ (n)}\right]=0$ and
\begin{equation*}
s_{ n, V}^{ 2}:= \sum_{ j=1}^{ n} \mathbb{ E} \left[\left(V_{ j}^{ (n)}\right)^{ 2}\right]= (1-p_{ n}) \left(u^{ 2} \left(\int f {\rm d}\mu_{ 0}\right)^{ 2} + v^{ 2} \left(\int g {\rm d}\mu_{ 0}\right)^{ 2}\right),
\end{equation*}
which goes as $n\to\infty$ to
\begin{equation}
s_{ V}^{ 2}:= (1-p) \left(u^{ 2} \left(\int f {\rm d}\mu_{ 0}\right)^{ 2} + v^{ 2} \left(\int g {\rm d}\mu_{ 0}\right)^{ 2}\right) \label{eq:s2V}.
\end{equation}
The Lyapounov condition is also verified: $ \frac{ 1}{ s_{ n, V}^{ 2+ \delta}}\sum_{ j=1}^{ n} \mathbb{ E} \left[ \left\vert V_{ j}^{ (n)} \right\vert^{ 2+ \delta}\right] $ is of order $ \frac{ c}{ \left(np_{ n}\right)^{ \frac{\delta}{ 2}}} \xrightarrow[ n\to\infty]{}0$. Hence, applying the same result, we have that $u Y_{ n}^{ (1)}(f) + v Y_{ n}^{ (2)}(g)$ converges in law to some Gaussian $ \mathcal{ N} \left(0, s_{ V}^{ 2}\right)$, where $s_{ V}^{ 2}$ is given by \eqref{eq:s2V}.

With these two convergence results at hand, we can now go back to \eqref{eq:char_funct}: since that $ \mathbb{ P}$-a.s., $u X_{ n}^{ (1)} (f) + v X_{ n}^{ (2)} (g) + w \left\langle \eta_{ 0}^{ n}\, ,\, h\right\rangle$ converges in law to $ \mathcal{ N}(0, s_{ U}^{ 2})$, we have that $ \mathbb{ P}$-a.s., $\mathbb{ E} \left[\exp \left(iu X_{ n}^{ (1)} (f) + iv X_{ n}^{ (2)} (g) + iw \left\langle \eta_{ 0}^{ n}\, ,\, h\right\rangle\right) \Bigg \vert \mathcal{ F}_{ \xi}\right]$ converges to $\exp \left(- \frac{ s_{ U}^{ 2}}{ 2}\right)$. Then one can write from \eqref{eq:char_funct},
\begin{align*}
\varphi_{ n}(u,v,w)=& \exp \left(- \frac{ s_{ U}^{ 2}}{ 2}\right) \mathbb{ E} \Bigg[ \exp \left(iu Y_{ n}^{ (1)}(f) + iv Y_{ n}^{ (2)}(g)\right)\Bigg] \\
&+ \mathbb{ E} \Bigg[ \exp \left(iu Y_{ n}^{ (1)}(f) + iv Y_{ n}^{ (2)}(g)\right)\\
&\qquad \qquad \qquad \times \mathbb{ E} \left[\exp \left(iu X_{ n}^{ (1)} (f) + iv X_{ n}^{ (2)} (g) + iw \left\langle \eta_{ 0}^{ n}\, ,\, h\right\rangle\right) \Bigg \vert \mathcal{ F}_{ \xi}\right] - e^{ - \frac{ s_{ U}^{ 2}}{ 2}}\Bigg].
\end{align*}
The first term above converges to $\exp \left(- \frac{ s_{ U}^{ 2}+ s_{ V}^{ 2}}{ 2}\right)$ and by dominated convergence theorem, we see that the second term above converges to $0$.
\end{proof}

\subsection{Semimartingale decompositions}
We follow here the same approach as for the global fluctuation process, that is to apply Ito's formula so as to derive a proper semimartingale decomposition for $ \zeta_{ t}^{ n, l}$, the key step being to identify the vanishing terms as $n\to\infty$ in this decomposition (which correspond here to terms in the asymptotic development of $ \mu^{ n, l}$ that are of order lower than $ \frac{ 1}{ \sqrt{ np_{ n}}}$). Since the approach is highly similar to the one followed for global fluctuations, we only give the main lines of proof and leave to the details to the reader.

\begin{proposition}
\label{prop:semimart_zetas} 
For any $r> \frac{ 3}{ 2}$ the joint process $ (\zeta^{ n,1}, \zeta^{ n, 2}, \eta^n)$ belongs $ \mathbb{ P}\otimes \mathbf{ P}$-a.s. to $ \mathcal{ C} \left([0, T], H^{ -r} \left( \mathbb{ T}\right)^{ 3}\right)$. Moreover, $ \zeta^{ n, l}$, $l=1, 2$ satisfy the following semimartingale representation in $H^{-r_{ 1}}(\bbT)$
\begin{equation}
\label{eq:semimart_zeta}
\zeta_{ t}^{ n, l} = \zeta_{ 0}^{ n, l} + \int_{ 0}^{t}\mathcal{ U}_{ s}^{ \ast}\zeta_{ s}^{ n, l}{\rm d}s + \sqrt{ p_{ n}}\int_{ 0}^{t}  \mathcal{ V}_{ s}^{ n,l,\ast}\eta^{ n}_{ s} {\rm d}s + \int_{ 0}^{t}  \left\lbrace \sqrt{ p_{ n}}\Theta^{ \ast}\hat{ \eta}_{ s}^{ n} + \Theta^{ \ast} \varpi_{ s}^{ n, l}\right\rbrace {\rm d}s + W_{ t}^{ n,l}
\end{equation}
where $ \mathcal{ U}_{ s}$ is given in \eqref{eq:Us}, $ \mathcal{ V}_{ s}^{ n, l}f(\theta):= \left\langle \mu_{ s}^{ n, l}({\rm d}\theta^{ \prime})\, ,\, \partial_{ \theta}f \left(\theta^{ \prime}\right)\Gamma \left(\theta^{ \prime}, \theta\right)\right\rangle$ is the microscopic equivalent of $ \mathcal{ V}_{ s}$ defined in \eqref{eq:Vs} and $ \Theta$ is defined by \eqref{eq:Theta}. The remaining drift term in \eqref{eq:semimart_zeta} is given by
\begin{equation}
\label{eq:hln}
\varpi_{ t}^{n, l}:= \frac{ \sqrt{ np_{ n}}}{ n^{ 2}} \sum_{i, j=1}^n \hat{ \xi}_{li}^{ (n)}\hat{\xi}^{(n)}_{ij}  \delta_{\theta^{i,n}_t} \otimes \delta_{\theta^{j,n}_t},
\end{equation}
and the noise term is
\begin{equation}
W_{ t}^{ n,l}(f):=\int_{ 0}^{t} \frac{ 1}{ \sqrt{ np_{ n}}} \sum_{ i=1}^{ n} \xi_{l i}^{ (n)} \partial_{ \theta}f(\theta_{ s}^{ i, n}) {\rm d}B^{ i}_{s}. \label{aux:IV}
\end{equation}
The process $ (W_{ t}^{ n, 1}, W_{ t}^{ n, 2}, W_{ t}^{ n})_{t \in [0,T]}$ is a martingale in $ \cC \left([0, T], \left(H^{ -r}\right)^{ 3}\right)$ for  $r> \frac{ 3}{ 2}$, with Doob-Meyer process given for $t \in [0,T]$ and $\varphi, \psi\in H^{ r}$ by
\begin{equation}
\label{eq:DM_Mn12G}
\begin{split}
\llangle W^{ n, l}, W^{ n, l}\rrangle_{ t}\cdot \varphi(\psi)&= \int_{ 0}^{t}\left\langle \mu_{ s}^{ n, l}\, ,\,  \partial_{ \theta} \varphi \partial_{ \theta}\psi\right\rangle {\rm d}s, \ l=1, 2,\\
\llangle W^{ n, 1}, W^{ n, 2}\rrangle_{ t}\cdot \varphi(\psi)&= p_{ n}\int_{ 0}^{t}\left\langle \mu_{ s}^{ n, 1,2}\, ,\, \partial_{ \theta} \varphi \partial_{ \theta}\psi\right\rangle {\rm d}s,\\
\llangle W^{ n, l}, W^{ n}\rrangle_{ t}\cdot \varphi(\psi)&= \sqrt{ p_{ n}}\int_{ 0}^{t}\left\langle \mu_{ s}^{ n, l}\, ,\,  \partial_{ \theta} \varphi \partial_{ \theta}\psi\right\rangle {\rm d}s, \ l=1, 2.
\end{split}
\end{equation}
Finally, for $r>1/2$, the process $ \zeta^{ n, l}$ satisfies the following weak-mild equation: for every $h \in H^r(\bbT)$ and $t \in [0,T]$:
\begin{equation}
\label{eq:zeta_mild_semimart}
\begin{split}
\left\langle \zeta_{ t}^{ n, l}\, ,\, h\right\rangle_{-r,r} =& \left\langle \zeta_{ 0}^{ n, l}\, ,\, S_{t}h\right\rangle_{-r,r} + \int_{ 0}^{t} \left\langle \zeta_{ s}^{ n, l}\, ,\, (\Gamma*\mu_s) \partial_\theta S_{t-s} h  \right\rangle_{-r,r} {\rm d}s \\
&+ \sqrt{ p_{ n}}\int_{ 0}^{t} \left\langle \mu_{ s}^{ n, l}\, ,\, (\Gamma*\eta^n_s) \partial_\theta S_{t-s} h  \right\rangle_{-r,r} {\rm d}s \\
&+ \int_0^t \frac{ \sqrt{ p_{ n}}}{n^{3/2}} \sum_{i,j=1}^n \hat{ \xi}_{ij}^{ (n)} \< \delta_{\theta^{i,n}_s}, (\Gamma * \delta_{\theta^{j,n}_s}) \partial_\theta S_{t-s}h >_{-r,r} \dd s \\ 
&+ \int_{ 0}^{t} \frac{ \sqrt{ np_{ n}}}{ n^{ 2}} \sum_{ i,j=1}^{ n} \hat{ \xi}_{li}^{ (n)} \hat{ \xi}_{ij}^{ (n)} \< \delta_{\theta^{i,n}_s}, (\Gamma * \delta_{\theta^{j,n}_s}) \partial_\theta S_{t-s}h >_{-r,r} {\rm d}s \\
&+ w_{ t}^{ n, l}(h),
\end{split}
\end{equation}
where
\begin{equation*}
w^{ n, l}_t(h) = \frac{1}{\sqrt{n p_{ n}}}\sum_{i=1}^n \xi_{li}^{ (n)}\int_0^t (\partial_\theta S_{t-s} h) \left(\theta^{i,n}_s\right)\dd B^i_s.
\end{equation*}
\end{proposition}
\begin{proof}[Main lines of proof of Proposition~\ref{prop:semimart_zetas}]
First begin by applying Ito's formula to $\mu^{n,l}$, for $l=1,2$: for $f$ regular,
\begin{multline}
\label{eq:nunl}
\left\langle \mu_{ t}^{ n, l}\, ,\, f\right\rangle = \left\langle \mu_{ 0}^{ n, l}\, ,\, f\right\rangle + \int_{ 0}^{t} \left\langle \mu_{ s}^{ n, l}\, ,\, \frac{ 1}{ 2} \partial_{ \theta}^{ 2} f\right\rangle{\rm d}s + \int_{ 0}^{t} \frac{ 1}{ np_{ n}} \sum_{ i=1}^{ n} \xi_{l i}^{ (n)} \partial_{ \theta}f(\theta_{ s}^{ i, n}) {\rm d}B^{ i}{ s} \\+ \int_{ 0}^{t} \frac{ 1}{ n^{ 2}} \sum_{ i,j=1}^{ n} \frac{ \xi_{li}^{ (n)}}{ p_{ n}} \frac{ \xi_{ij}^{ (n)}}{ p_{ n}}\partial_{ \theta}f \left(\theta_{ s}^{ i,n}\right) \Gamma \left(\theta_{ s}^{ i,n}, \theta_{ s}^{ j, n}\right){\rm d}s.
\end{multline}
  Concentrate on the last term of \eqref{eq:nunl}: write $\frac{ \xi_{li}^{ (n)}}{ p_{ n}}= \hat{ \xi}_{li}^{ (n)} +1$ and $\frac{ \xi_{ij}^{ (n)}}{ p_{ n}}= \hat{ \xi}_{ij}^{ (n)} +1$ so that
\begin{align*}
\int_{ 0}^{t} \frac{ 1}{ n^{ 2}} \sum_{ i,j=1}^{ n} \frac{ \xi_{li}^{ (n)}}{ p_{ n}} \frac{ \xi_{ij}^{ (n)}}{ p_{ n}}\partial_{ \theta}f \left(\theta_{ s}^{ i,n}\right)& \Gamma \left(\theta_{ s}^{ i,n}, \theta_{ s}^{ j, n}\right){\rm d}s \\
&= \int_{ 0}^{t} \frac{ 1}{ n^{ 2}} \sum_{ i,j=1}^{ n} \frac{ \xi_{li}^{ (n)}}{ p_{ n}} \partial_{ \theta}f \left(\theta_{ s}^{ i,n}\right) \Gamma \left(\theta_{ s}^{ i,n}, \theta_{ s}^{ j, n}\right){\rm d}s\\
&\quad+ \int_{ 0}^{t} \frac{ 1}{ n^{ 2}} \sum_{ i,j=1}^{ n}  \hat{ \xi}_{ij}^{ (n)}\partial_{ \theta}f \left(\theta_{ s}^{ i,n}\right) \Gamma \left(\theta_{ s}^{ i,n}, \theta_{ s}^{ j, n}\right){\rm d}s\\
&\quad + \int_{ 0}^{t} \frac{ 1}{ n^{ 2}} \sum_{ i,j=1}^{ n} \hat{ \xi}_{li}^{ (n)} \hat{ \xi}_{ij}^{ (n)}\partial_{ \theta}f \left(\theta_{ s}^{ i,n}\right) \Gamma \left(\theta_{ s}^{ i,n}, \theta_{ s}^{ j, n}\right){\rm d}s.
\end{align*}
The first term above is equal to $\int_{ 0}^{t} \left\langle \mu_{ s}^{ n, l}({\rm d}\theta)\, ,\, \partial_{ \theta}f \left(\theta\right)  \left\langle \mu_{ s}^{ n}({\rm d}\theta^{ \prime})\, ,\, \Gamma \left(\theta, \theta^{ \prime}\right)\right\rangle\right\rangle {\rm d}s$ and that the second one is exactly $ \int_{ 0}^{t} \< \hat{\nu}^n_t(\dd \theta_1, \dd \theta_2), \Gamma\left(\theta_1, \theta_2\right)\partial_\theta f\left(\theta_1\right) > {\rm d}s$ (recall \eqref{eq:graph_term_hat_eta_n}). Using now the fact that $\mu$ solves  \eqref{eq:limit PDE}, we obtain (recall the definitions of $ \eta^{ n}$ in \eqref{def:fluctuation_process} and of $ \hat{ \eta}^{ n}$ in \eqref{eq:hat_etan}),
\begin{align*}
\left\langle \zeta_{ t}^{ n, l}\, ,\, f\right\rangle =& \left\langle \zeta_{ 0}^{ n, l}\, ,\, f\right\rangle + \int_{ 0}^{t} \left\langle \zeta_{ s}^{ n, l}\, ,\, \frac{ 1}{ 2} \partial_{ \theta}^{ 2} f + \partial_{ \theta}f \left(\theta\right)  \left\langle \mu_{ s}({\rm d}\theta^{ \prime})\, ,\, \Gamma \left(\theta, \theta^{ \prime}\right)\right\rangle \right\rangle{\rm d}s \\ 
&+ \sqrt{ p_{ n}}\int_{ 0}^{t} \left\langle \eta^{ n}_{ s}\left({\rm d}\theta^{ \prime}\right)\, ,\, \left\langle \mu_{ s}^{ n, l}({\rm d}\theta)\, ,\, \partial_{ \theta}f \left(\theta\right) \Gamma \left(\theta, \theta^{ \prime}\right)\right\rangle\right\rangle {\rm d}s\\
&+ \sqrt{ p_{ n}}\int_{ 0}^{t} \left\langle \hat{ \eta}_{ s}^{ n}(\dd \theta_1, \dd \theta_2)\, ,\, \Gamma\left(\theta_1, \theta_2\right)\partial_\theta f\left(\theta_1\right)\right\rangle {\rm d}s\\
&+ \int_{ 0}^{t} \frac{ \sqrt{ np_{ n}}}{ n^{ 2}} \sum_{ i,j=1}^{ n} \hat{ \xi}_{li}^{ (n)} \hat{ \xi}_{ij}^{ (n)}\partial_{ \theta}f \left(\theta_{ s}^{ i,n}\right) \Gamma \left(\theta_{ s}^{ i,n}, \theta_{ s}^{ j, n}\right){\rm d}s+ W_{ t}^{ n,l}(f).
\end{align*}
All of this gives \eqref{eq:semimart_zeta}, using the definition of the noise in \eqref{aux:IV} and the drift term $ \varpi^{ n, l}$ in \eqref{eq:hln}. The rest of the proof follows from the same arguments as for Lemma~\ref{lem:fluctuation_process_semimart}.
\end{proof}

\subsection{Tightness and convergence}
Recall Proposition~\ref{prop:limit_hateta0_indep}: under our hypotheses $ \hat{ \eta}^{ n}$ converges to $ \hat{ \eta}\equiv 0$ as $n\to\infty$, so that we see that the term $\int_{ 0}^{t} \sqrt{ p_{ n}}\Theta^{ \ast}\hat{ \eta}_{ s}^{ n} {\rm d}s$ in \eqref{eq:semimart_zeta} does not contribute to the limit when $n\to\infty$. It remains to deal with the term $\int_{ 0}^{t} \Theta^{ \ast} \varpi_{ s}^{ n, l} {\rm d}s$ that we want also to prove that it vanishes as $n\to\infty$. This is the purpose of the following proposition:
\begin{proposition}
\label{prop:varpi}
Under the hypotheses of Section~\ref{sec:local_fluct}, the process $ \left(\varpi^{ n, l}\right)$ converges to $0$ as $n\to\infty$, in $ \mathcal{ C} \left([0, T], H^{ - r_{ 1}}(\mathbb{ T}^{ 2})\right)$.
\end{proposition}
\begin{proof}[Main lines of proof of Proposition~\ref{prop:varpi}]
We follow the same strategy as for $ \hat{ \eta}^{ n}$: to write a semimartingale decomposition for $ \varpi^{ n, l}$, to prove tightness of this process and to identify its limit as the unique solution to a linear SPDE with noise and initial condition identically zero, so that, by uniqueness $\lim_{ n\to\infty} \varpi^{ n, l}\equiv 0$. We only draw the main lines of proof here. Recall the definition of $ \varpi^{ n, l}$ in \eqref{eq:hln}. By Ito's formula, for any regular $(\theta_{ 1}, \theta_{ 2}) \mapsto g(\theta_{ 1}, \theta_{ 2})$
\begin{align}
\left\langle \varpi_{ t}^{n, l}\, ,\, g\right\rangle=& \left\langle \varpi_{ 0}^{n, l}\, ,\, g\right\rangle + \int_{ 0}^{t} \left\langle \varpi_{ s}^{n, l}\, ,\, \frac{ 1}{ 2} \Delta g\right\rangle \nonumber\\
&+ \int_{ 0}^{t} \frac{ \sqrt{ np_{ n}}}{ n^{ 3}} \sum_{ i,j,k} \hat{ \xi}_{li}^{ (n)}\hat{\xi}^{(n)}_{ij} \frac{ \xi_{ik}^{ (n)}}{ p_{ n}} \partial_{ \theta_{ 1}} g(\theta_{ s}^{ i, n}, \theta_{ s}^{ j, n}) \Gamma \left( \theta_{ s}^{ i, n}, \theta_{ s}^{ k, n}\right) {\rm d}s\label{aux:hI}\\
&+ \int_{ 0}^{t} \frac{ \sqrt{ np_{ n}}}{ n^{ 3}} \sum_{ i,j,k} \hat{ \xi}_{li}^{ (n)}\hat{\xi}^{(n)}_{ij} \frac{ \xi_{jk}^{ (n)}}{ p_{ n}} \partial_{ \theta_{ 2}} g(\theta_{ s}^{ i, n}, \theta_{ s}^{ j, n}) \Gamma \left( \theta_{ s}^{ j, n}, \theta_{ s}^{ k, n}\right) {\rm d}s\label{aux:hII}\\
&+ \int_{ 0}^{t} \frac{ \sqrt{ np_{ n}}}{ n^{ 2}} \sum_{ i,j}\hat{ \xi}_{li}^{ (n)}\hat{\xi}^{(n)}_{ij} \left( \partial_{ \theta_{ 1}}g(\theta_{ s}^{ i, n}, \theta_{ s}^{ j, n}) {\rm d}B^{ i}{s} +\partial_{ \theta_{ 2}}g(\theta_{ s}^{ i, n}, \theta_{ s}^{ j, n}) {\rm d}B^{ j}{s}\right). \label{aux:noise_h}
\end{align}
Writing again $\frac{ \xi_{ik}^{ (n)}}{ p_{ n}}= 1+\hat{ \xi}_{ik}^{ (n)}$ and $\frac{ \xi_{jk}^{ (n)}}{ p_{ n}}= 1+\hat{ \xi}_{jk}^{ (n)}$, so that 
\begin{align*}
\eqref{aux:hI}=&\int_{ 0}^{t} \left\langle \varpi_{ s}^{n, l} ({\rm d}\theta_{ 1}, {\rm d}\theta_{ 2})\, ,\,  \partial_{ \theta_{ 1}} g(\theta_{ 1}, \theta_{ 2}) \left(\frac{ 1}{ n} \sum_{ k}\Gamma \left( \theta_{ 1}, \theta_{ s}^{ k, n}\right)\right)\right\rangle{\rm d}s\\&+ \int_{ 0}^{t} \frac{ \sqrt{ np_{ n}}}{ n^{ 3}} \sum_{ i,j,k} \hat{ \xi}_{li}^{ (n)}\hat{\xi}^{(n)}_{ij} \hat{ \xi}^{ (n)}_{ i, k}\partial_{ \theta_{ 1}} g(\theta_{ s}^{ i, n}, \theta_{ s}^{ j, n}) \Gamma \left( \theta_{ s}^{ i, n}, \theta_{ s}^{ k, n}\right) {\rm d}s,\\
\eqref{aux:hII}=&\int_{ 0}^{t} \left\langle \varpi_{ s}^{n, l} ({\rm d}\theta_{ 1}, {\rm d}\theta_{ 2})\, ,\,  \partial_{ \theta_{ 2}} g(\theta_{ 1}, \theta_{ 2}) \left(\frac{ 1}{ n} \sum_{ k}\Gamma \left( \theta_{ 2}, \theta_{ s}^{ k, n}\right)\right)\right\rangle{\rm d}s\\&+ \int_{ 0}^{t} \frac{ \sqrt{ np_{ n}}}{ n^{ 3}} \sum_{ i,j,k} \hat{ \xi}_{li}^{ (n)}\hat{\xi}^{(n)}_{ij} \hat{ \xi}^{ (n)}_{ j, k}\partial_{ \theta_{ 2}} g(\theta_{ s}^{ i, n}, \theta_{ s}^{ j, n}) \Gamma \left( \theta_{ s}^{ j, n}, \theta_{ s}^{ k, n}\right) {\rm d}s.
\end{align*}
Another application of Grothendieck inequality gives that the remaining terms in the last two sums are controlled in $H^{ -r_{ 1}} \left(\mathbb{ T}^{ 2}\right)$ by respectively $ \sqrt{ np_{ n}}S_{ n}^{ \lijik}(l)$ defined in \eqref{eq:Slijik} and $ \sqrt{ np_{ n}}S_{ n}^{ \lijk}(l)$ defined in \eqref{eq:Slijk}. Hence, by Proposition~\ref{prop:concentration_SnT}, this term is of order $ \frac{ 1}{ (np_{ n}^{ 5})^{ 1/2}}  $ which goes to $0$ as $n\to\infty$. Secondly, the noise term in \eqref{aux:noise_h}, that we denote by $W_{ t}^{ n, \varpi}(g)$ is again controlled as follows 
\begin{align*}
\mathbf{ E} \left[ W_{ t}^{ n, \varpi}(g)^{ 2}\right] &= \frac{ p_{ n}}{ n^{ 3}} \sum_{ i,j,k} \left(\hat{ \xi}_{li}^{ (n)}\right)^{ 2}\hat{\xi}^{(n)}_{ij} \hat{\xi}^{(n)}_{ik} \mathbf{ E}\left(\partial_{ \theta_{ 1}}g (\theta_{ s}^{ i, n}, \theta_{ s}^{ j, n}) \partial_{ \theta_{ 1}} g(\theta_{ s}^{ i, n}, \theta_{ s}^{ k, n})\right),
\end{align*}
and by the same argument as above, the $H^{ -r} \left(\mathbb{ T}^{ 2}\right)$-norm of $W^{ n \varpi}$ (for $r>3$) is of order $ \frac{ 1}{ np_{ n}^{ 3}}$ which goes to $0$ as $n\to\infty$, so that the noise term \eqref{aux:noise_h} vanishes as $n\to\infty$.
Finally, we turn to the initial condition:
\begin{align}
\label{eq:hnlg}
\left\langle \varpi_{ 0}^{l, n}\, ,\, g\right\rangle= \frac{ \sqrt{ np_{ n}}}{ n^{ 2}} \sum_{i, j=1}^n \hat{ \xi}_{li}^{ (n)}\hat{\xi}^{(n)}_{ij} g \left(\theta^{i,n}_0, \theta^{j,n}_0\right).
\end{align}
Using the same arguments as for the proof of Proposition~\ref{prop:limit_hateta0_indep}, it is straightforward to show that $\lim_{n\rightarrow\infty} \mathbb{ E}_{ 0} \left[\left\Vert \varpi_{ 0}^{n,l} \right\Vert_{-r_1}^{ 2}\right] =0, \quad \bbP_g-\text{a.s.}.$ (recall that the initial condition is supposed to be i.i.d. independent of the graph).

By the same arguments as before, one can prove that the process $( \varpi_{ t}^{ n, l})$ is tight in $H^{ -r_{ 1}} \left(\mathbb{ T}^{ 2}\right)$ and converges as $n\to\infty$ to the unique solution $ \varpi_{ s}$ in $H^{ -r_{ 2}} \left(\mathbb{ T}^{ 2}\right)$ to
\begin{equation}
\label{eq:ht}
\begin{split}
\left\langle \varpi_{ t}^{ l}\, ,\, g\right\rangle=& \int_{ 0}^{t} \left\langle \varpi_{ s}^{ l}\, ,\, \frac{ 1}{ 2} \Delta g\right\rangle {\rm d}s\\
&+ \int_{ 0}^{t} \left\langle \varpi_{ s}^{ l} ({\rm d}\theta_{ 1}, {\rm d}\theta_{ 2})\, ,\,  \partial_{ \theta_{ 1}} g(\theta_{ 1}, \theta_{ 2}) \left\langle \mu_{ s}({\rm d}u)\, ,\, \Gamma \left(\theta_{ 1}, u\right) \right\rangle\right\rangle{\rm d}s\\
&+ \int_{ 0}^{t} \left\langle \varpi_{ s}^{ l} ({\rm d}\theta_{ 1}, {\rm d}\theta_{ 2})\, ,\,  \partial_{ \theta_{ 2}} g(\theta_{ 1}, \theta_{ 2}) \left\langle \mu_{ s}({\rm d}u)\, ,\, \Gamma \left(\theta_{ 2}, u\right) \right\rangle\right\rangle{\rm d}s,
\end{split}
\end{equation}
which, by uniqueness of a solution to \eqref{eq:ht}, is necessarily $ \varpi\equiv 0$. This proves Proposition~\ref{prop:varpi}.
\end{proof}
\subsection{Identification of the noise}
\begin{proposition}
\label{prop:ident_noise_zetas}
Under the hypotheses of Section~\ref{sec:local_fluct}, the joint noise process (considered in Proposition~\ref{prop:semimart_zetas}) $ (W_{ t}^{ n, 1}, W_{ t}^{ n, 2}, W_{ t}^{ n})_{t \in [0,T]}$ converges as $n\to\infty$ in $H^{ -r_{ 1}}(\mathbb{ T})^{ 3}$ to $ (W_{ t}^{1}, W_{ t}^{2}, W_{ t})_{t \in [0,T]}$ Gaussian process with covariance, for $f_{ 1}, f_{ 2} \in H^{ r}, \ 0\leq s \leq t\leq T$,
\begin{equation}
\label{eq:cov_noise_12}
\begin{split}
\mathbf{ E} \left[ W_{ s}^{ l}(f_{ 1}) W_{ t}^{ l}(f_{ 2})\right]=\mathbf{ E} \left[ W_{ s}(f_{ 1}) W_{ t}(f_{ 2})\right]&= \int_{ 0}^{s}\left\langle \mu_{ u}\, ,\,  \partial_{ \theta} f_{ 1} \partial_{ \theta}f_{ 2}\right\rangle {\rm d}u, \ l=1, 2,\\
\mathbf{ E} \left[ W_{ s}^{ 1}(f_{ 1}) W_{ t}^{ 2}(f_{ 2})\right]&= p\int_{ 0}^{s}\left\langle \mu_{ u}\, ,\, \partial_{ \theta} f_{ 1} \partial_{ \theta}f_{ 2}\right\rangle {\rm d}u,\\
\mathbf{ E} \left[ W_{ s}^{ l}(f_{ 1}) W_{ t}(f_{ 2})\right]&= \sqrt{ p}\int_{ 0}^{s}\left\langle \mu_{ u}\, ,\,  \partial_{ \theta} \varphi \partial_{ \theta}\psi\right\rangle {\rm d}u, \ l=1, 2.
\end{split}
\end{equation}
Moreover, $ (W_{ t}^{1}, W_{ t}^{2}, W_{ t})_{t \in [0,T]}$ is independent of the initial condition $ (\zeta_{ 0}^{1}, \zeta_{ 0}^{2}, \eta_{ 0})$ given in Proposition~\ref{prop:ident_cond0}.
\end{proposition}
\begin{rem}
\label{rem:W12}
In the diluted case, $ (W_{ t}^{1}, W_{ t}^{2}, W_{ t})_{t \in [0,T]}$ are mutually independent.
\end{rem}
\begin{proof}[Proof of Proposition~\ref{prop:ident_noise_zetas}]
  This follows directly from \eqref{eq:DM_Mn12G} and the convergence results \eqref{eq:conv_munl} and \eqref{eq:conv_mun12}.
\end{proof}
\subsection{Proof of the main convergence result}
We now turn to the proof of Theorem~\ref{th:local_fluct}:
\begin{proof}[Main lines of proof of Theorem~\ref{th:local_fluct}]
Putting all the previous estimates into the semimartingale decomposition \eqref{eq:semimart_zeta} and applying the same arguments as for the process $ \eta^{ n}$ (using in particular the weak-mild formulation \eqref{eq:zeta_mild_semimart} which lead to similar estimates as for Proposition~\ref{prop:eta^n_tight}), we see that $ \zeta^{ n, l}$ is tight in $H^{ -r_{ 1}} \left(\mathbb{ T}\right)$ and that, almost surely w.r.t. the randomness of the graph, the joint process $\left(\zeta^{ n, 1}, \zeta^{ n, 2}, \eta^{ n}\right)$ converges in $ \mathcal{ C} \left([0, T], \left(H^{ - r_{ 1}}(\mathbb{ T})\right)^{ 3}\right)$ towards $ \left(\zeta^{1}, \zeta^{2}, \eta\right)$ solution to
\begin{equation*}
\begin{cases}
\begin{split}
\left\langle \zeta_{ t}^{l}\, ,\, f\right\rangle &= \left\langle \zeta_{ 0}^{l}\, ,\, f\right\rangle + \int_{ 0}^{t} \left\langle \zeta_{ s}^{l}\, ,\, \frac{ 1}{ 2} \partial_{ \theta}^{ 2} f + \partial_{ \theta}f \left(\theta\right)  \left\langle \mu_{ s}({\rm d}\theta^{ \prime})\, ,\, \Gamma \left(\theta, \theta^{ \prime}\right)\right\rangle \right\rangle{\rm d}s\\
&\quad+ \sqrt{ p}\int_{ 0}^{t} \left\langle \eta_{ s} \left({\rm d}\theta^{ \prime}\right)\, ,\, \left\langle \mu_{ s}({\rm d}\theta)\, ,\, \partial_{ \theta}f \left(\theta\right)  \, ,\, \Gamma \left(\theta, \theta^{ \prime}\right)\right\rangle\right\rangle {\rm d}s
+ W_{ t}^{ l}(f),\\
\left\langle \eta_{ t}\,,\, f\right\rangle &= \left\langle \eta_{ 0}\,,\, f\right\rangle + \int_{ 0}^{t} \left\langle \eta_{ s} \,,\,  \mathcal{ L}_{ \mu_s}^{(1)}(f)\right\rangle {\rm d}s +  W_{ t}(f).
\end{split}
\end{cases}
\end{equation*}
that is nothing else as the weak formulation of \eqref{eq:SDPEs_zetas_eta}. Note that we use here the convergence \eqref{eq:conv_munl} to identify the limit as in the proof of Proposition~\ref{prop:ident_lim}. Uniqueness of a solution to \eqref{eq:semimart_zeta} follows from the same arguments as for Section~\ref{sec:uniqueness}.
\end{proof}

\appendix

\section{Sobolev spaces and inequalities}
\label{sec:sobolev}
\subsection{From mixed $L^\infty$-$L^2$ norms to Sobolev norms}

\begin{lemma}
\label{lem:mixed_sobolev_ineq}
There exists a constant $C>0$ such that, for any regular function $v: \bbT\times \bbT \to \R$, it holds that
\begin{equation*}
\sup_{\theta_1 \in \bbT} \norm{v(\theta_1, \cdot)}^2_{L^2(\dd \theta_2)} \leq C \norm{v(\cdot, \cdot)}^2_{H^1(\dd \theta_1, \dd \theta_2)}.
\end{equation*}
\end{lemma}

\begin{proof}
Let $u:\bbT \to \R$ be defined for $\theta_1 \in \bbT$ by $u(\theta_1) = \norm{v(\theta_1, \cdot)}^2_{L^2(\dd \theta_2)}$. Observe that $\norm{u}_{L^1} = \norm{v(\cdot, \cdot)}^2_{L^2(\dd \theta_1, \dd \theta_2)}$. In particular,
\begin{equation*}
\begin{split}
|\partial_{\theta_1} u(\theta_1)| \leq & 2 \int_\bbT \left|\partial_{\theta_1} v(\theta_1, \theta_2) \right| \left|v(\dd \theta_1, \dd \theta_2)\right| \dd \theta_2 \leq \\
\leq & \int_\bbT \left|\partial_{\theta_1} v(\theta_1, \theta_2) \right|^2 \dd \theta_2 + \int_\bbT \left|v(\dd \theta_1, \dd \theta_2)\right|^2 \dd \theta_2.
\end{split}
\end{equation*}
By integrating the previous expression with respect to $\theta_1$, one obtains
\begin{equation*}
\norm{\partial_{\theta_1} u}_{L^1} \leq \int_\bbT \left|\partial_{\theta_1} v(\theta_1, \theta_2) \right|^2 \dd \theta_1 \dd \theta_2 + \norm{v(\cdot, \cdot)}^2_{L^2(\dd \theta_1, \dd \theta_2)},
\end{equation*}
which implies
\begin{equation*}
\begin{split}
\norm{u}_{W^{1,1}}
&\leq \int_\bbT \left( \left|\partial_{\theta_1} v(\theta_1, \theta_2) \right|^2 +  \left|\partial_{\theta_2} v(\theta_1, \theta_2) \right|^2 \right) \dd \theta_1 \dd \theta_2 + 2 \norm{v(\cdot, \cdot)}^2_{L^2(\dd \theta_1, \dd \theta_2)}\\
&\leq 2 \norm{v(\cdot, \cdot)}^2_{W^{1,2}(\dd \theta_1, \dd \theta_2)}.
\end{split}
\end{equation*}
By using \cite[Theorem 5.4, Part I, Case (B)]{adams_fournier_2003}, there exists an universal constant $C>0$ such that
\begin{equation*}
\sup_{\theta_1 \in \bbT} |u(\theta_1)| = \norm{u}_{C^0} \leq C \norm{u}_{W^{1,1}}.
\end{equation*}
Observe that $\sup_{\theta_1 \in \bbT} |u(\theta_1)| = \sup_{\theta_1 \in \bbT} \norm{v(\theta_1, \cdot)}^2_{L^2(\dd \theta_2)}$, the proof is concluded with the constant given by $2C$.
\end{proof}

\section{More on Grothendieck inequalities and concentration estimates}\label{sec:App concentration}
We gather in this Section the definitions of some auxiliary (possibly weighted) empirical mean values concerning the centered variables $ \hat{ \xi}_{ij}^{ (n)}$ as well as concentration estimates for these quantities.

Define the following quantities (where $u=(u_{ i})_{ i=1, \ldots, n}$ and $(v_{ j})_{ j=1,\ldots, n}$ are fixed sequences such that $\left\vert u_{ i} \right\vert\leq 1,\ \left\vert v_{ j} \right\vert\leq 1$):
\begin{align}
U_{ n, 1}^{\ij}(l, v)&=\frac{ 1}{ n} \sum_{ j=1}^{ n} \hat{ \xi}_{l j}^{ (n)} v_{ j}, \ l=1,2,\label{eq:Uij1}\\
U_{ n, 1}^{\ikjk}(l, v)&=\frac{ 1}{ n} \sum_{ j=1}^{ n} \hat{ \xi}_{1, j}^{ (n)} \hat{ \xi}_{2 j}^{ (n)} v_{ j}, \ l=1,2,\label{eq:Uij12}\\
U_{ n,2}^{\ij}(u,v)&:=\frac{ 1}{ n^{ 2}} \sum_{ i,j=1}^{ n} \hat{\xi}^{ (n)}_{ij} u_{ i}v_{ j},\label{eq:Uij2}\\
V_{ n,2}^{ \ijjk}(u,v, l)&:= \frac{ 1}{ n^{ 2}} \sum_{ i,j=1}^{ n } \hat{ \xi}^{ (n)}_{ li} \hat{ \xi}^{ (n)}_{ ij} u_{ i}v_{ j},\ l=1, 2.\label{eq:Vij2}
\end{align}
In \eqref{eq:Uij1} and \eqref{eq:Uij12} (resp. \eqref{eq:Uij2} and \eqref{eq:Vij2}) the subscript $1$ (resp. $2$) stands for the fact that $U_{ n, 1}^{\ij}$ and $U_{ n, 1}^{\ikjk}$ (resp. $U_{ n,2}^{\ij}$ and $V_{ n,2}^{ \ijjk}$) are of order $1$ (resp. order $2$) in the sense that one sums over $j$ only (resp. over both $i$ and $j$).

\begin{lemma}
\label{lem:concent_U Vn}
Suppose that $np_{ n}^{ 2} \xrightarrow[ n\to\infty]{} +\infty$. Then for all $ \varepsilon\in \left(0, \frac{ 1}{ 2}\right)$,
\begin{align}
\limsup_{ n\to\infty}(np_{ n})^{ \frac{ 1}{ 2}- \varepsilon} \left\vert U_{ n, 1}^{\ij}(l, v)\right\vert &\leq 1, \mathbb{P}-a.s.,\ l=1,2 \label{eq:sum_hat_xi_lj}\\ 
\limsup_{ n\to\infty} (np_{ n}^{ 2})^{ \frac{ 1}{ 2}- \varepsilon} \left\vert U_{ n, 1}^{\ikjk}(l, v)  \right\vert&\leq 1, \mathbb{P}-a.s.\label{eq:sum_hat_xi_12j}
\end{align}
Moreover, under $np_{ n} \xrightarrow[ n\to\infty]{} +\infty$, for all $ \varepsilon\in \left(0, \frac{ 1}{ 2}\right)$,
\begin{equation}\label{eq:bound Un}
\limsup_{n\rightarrow\infty}\, (np_{ n})^{ 1- \varepsilon} \left\vert U_{ n,2}^{\ij}(u,v) \right\vert \leq 1, \, \mathbb{ P}-a.s.
\end{equation}
Finally, suppose that $ p_{ n} \geq \frac{ c}{ n^{1-\delta}}$ for some $ \delta>0$ (as $n\to\infty$). For $l\in \left\lbrace 1, 2\right\rbrace$, for any $ \varepsilon>  \frac{ 1- \delta}{ \delta}$, if $ \left\vert u_{ i} \right\vert\leq 1$, $ \left\vert v_{ j} \right\vert\leq 1$,
\begin{equation}\label{eq:bound Vn}
\limsup_{n\rightarrow \infty}\,  (np_{ n})^{ 1- \varepsilon}\left\vert V_{ n,2}^{ \ijjk}(u,v, l) \right\vert \leq 1, \, \mathbb{ P}-a.s.
\end{equation}
\end{lemma}
\begin{proof}[Proof of Lemma~\ref{lem:concent_U Vn}]
We first prove \eqref{eq:sum_hat_xi_lj}. By Bernstein inequality, since $ \left\vert \hat{ \xi}_{lj}^{ (n)} v_{ j}\right\vert\leq \frac{ 1}{ p_{ n}}$ and $ \mathbb{ E} \left[ \left( \hat{ \xi}_{ij}^{ (n)}v_{ j}\right)^{ 2}\right]\leq \frac{ 1}{ p_{ n}}$, we have
$\mathbb{ P} \left( \left\vert \sum_{ j=1}^{ n} \hat{ \xi}_{lj}^{ (n)} v_{ j}\right\vert > t \right)\leq 2\exp \left(- \frac{ 1}{ 2} \frac{ t^{ 2}p_{ n}}{ n + \frac{ t}{ 3}}\right)$. Choosing for $ \varepsilon\in \left(0, \frac{ 1}{ 2}\right)$, $t= n^{ \frac{ 1}{ 2} + \varepsilon} p_{ n}^{ \varepsilon- \frac{ 1}{ 2}}$, we obtain that
$\mathbb{ P} \left( \left\vert \frac{ 1}{ n}\sum_{ j=1}^{ n} \hat{ \xi}_{lj}^{ (n)} v_{ j}\right\vert > \frac{ 1}{ (np_{ n})^{\frac{ 1}{ 2}- \varepsilon}} \right)\leq 2\exp \left(- \frac{ 1}{ 2} \frac{ (np_{ n})^{2\varepsilon}}{ 1 + \frac{ 1}{ 3(np_{ n})^{\frac{ 1}{ 2}- \varepsilon}}}\right)$. Since $np_{ n}\to\infty$, for $n$ large we have $\frac{ 1}{ 3(np_{ n})^{\frac{ 1}{ 2}- \varepsilon}}\leq 1$, so that the previous quantity is further bounded by $2\exp \left(- \frac{ 1}{ 4} (np_{ n})^{2\varepsilon}\right)$, which is summable under the assumptions of the present lemma. To prove \eqref{eq:sum_hat_xi_12j}, we apply the same Bernstein inequality to the sequence of independent $\hat{ \xi}_{1, j}^{ (n)} \hat{ \xi}_{2 j}^{ (n)}v_{ j}$, $j=1, \ldots, n$: since $ \left\vert \hat{ \xi}_{1, j}^{ (n)} \hat{ \xi}_{2 j}^{ (n)}v_{ j} \right\vert \leq \frac{ 1}{ p_{ n}^{ 2}}$ and $ \mathbb{ E} \left[ \left(\hat{ \xi}_{1, j}^{ (n)} \hat{ \xi}_{2 j}^{ (n)}v_{ j}\right)^{ 2}\right]\leq \frac{ 1}{ p_{ n}^{ 2}}$, we have
$ \mathbb{ P} \left( \left\vert \sum_{ j=1}^{ n} \hat{ \xi}_{1, j}^{ (n)} \hat{ \xi}_{2 j}^{ (n)}v_{ j} \right\vert > t \right)\leq 2\exp \left(- \frac{ 1}{ 2} \frac{ t^{ 2}p_{ n}^{ 2}}{ n + \frac{ t}{ 3}}\right)$. Hence, the calculation is the same as before, replacing $p_{ n}$ by $p_{ n}^{ 2}$ and the result follows from the same calculations. Estimate \eqref{eq:bound Un} is again a simple consequence of Bernstein inequality: for all $t>0$ we have
$\mathbb{ P} \left( \left\vert \sum_{ i,j} \hat{ \xi}_{ij} \right\vert > t\right) \leq 2\exp \left( - \frac{ 1}{ 2} \frac{ t^{ 2} p_{ n}}{ n^{ 2} + \frac{ t}{ 3}}\right)$.
Choosing $t= n^{ 1+ \varepsilon}p_{ n}^{\varepsilon-1}$, the previous bound becomes $2\exp \left( - \frac{ 1}{ 2} \frac{ n^{ 2 \varepsilon} p_{ n}^{ 2 \varepsilon-1}}{ 1 + \frac{ 1}{ 3 (np_{ n})^{ 1- \varepsilon}}}\right)$. Since $n p_{ n}\to \infty$, this quantity is further bounded, for $n$ large, by  $2\exp \left( - \frac{ 1}{ 4}  n^{ 2 \varepsilon} p_{ n}^{ 2 \varepsilon-1}\right)$. Now note that $n^{ \varepsilon} p_{ n}^{ 2 \varepsilon-1} \geq 1$ when $ \varepsilon\in \left(0, \frac{ 1}{ 2}\right)$, so that the final bound becomes $2\exp \left( - \frac{ 1}{ 4}  n^{\varepsilon} \right)$. Let us now give the proof of \eqref{eq:bound Vn} for $l=1$. Fix $(u_{ i}, v_{ j})$ such that $ \left\vert u_{ i} \right\vert\leq 1$ and $ \left\vert v_{ j} \right\vert\leq 1$ and define
\begin{equation*}
Y_{ n}:= \sum_{i, j=1}^n \hat{ \xi}_{1i}^{ (n)}\hat{\xi}_{ij}^{ (n)} u_{ i} v_{ j}
\end{equation*}
Denote by $ \mathcal{ F}_{ i}= \sigma \left( \hat{ \xi}_{pq}^{ (n)},\ p,q \leq i\right)$. Then $ \left(Y_{ i}\right)_{ i=1, \ldots, n}$ is a $ \left(\mathcal{ F}_{ i}\right)$-martingale and one has, for all $k<n$
\begin{align*}
\mathbb{ E} \left[ \left(Y_{ k+1}-Y_{ k}\right)^{ 2}\vert \mathcal{ F}_{ k}\right] &\leq 2 \mathbb{ E} \left[ \left( \sum_{ q=1}^{ k+1} \hat{ \xi}_{1, k+1}^{ (n)} \hat{ \xi}_{k+1, q}^{ (n)} u_{ k+1} v_{ q}\right)^{ 2} \vert \mathcal{ F}_{ k}\right]\\
& \quad + 2 \mathbb{ E} \left[ \left( \sum_{ p=1}^{ k} \hat{ \xi}_{1, p}^{ (n)} \hat{ \xi}_{p, k+1}^{ (n)} u_{ p} v_{ k+1}\right)^{ 2} \vert \mathcal{ F}_{ k}\right]\\
&= 2 \sum_{ q=1}^{ k+1} \mathbb{ E} \left[ \left(\hat{ \xi}_{1, k+1}^{ (n)}\right)^{ 2}\right] \mathbb{ E} \left[ \left(\hat{ \xi}_{k+1, q}^{ (n)}\right)^{ 2}\right] u_{ k+1}^{ 2}v_{ q}^{ 2} \\\
&\quad+ 2 \sum_{ p=1}^{ k} \left(\hat{ \xi}_{1, p}^{ (n)}\right)^{ 2} \mathbb{ E} \left[ \left(\hat{ \xi}_{p, k+1}^{ (n)}\right)^{ 2}\right]u_{ p}^{ 2} v_{ k+1}^{ 2}.
\end{align*} 
The above quantity is a.s. bounded by $ \frac{ ck}{ p_{ n}^{ 2}}$, for some numerical constant $c$ independent of $k, n, u, v$. Hence, we deduce that, almost surely $\left\langle Y\right\rangle_{ n} \leq \frac{ c^{ \prime} n^{ 2}}{ p_{ n}^{ 2}}$.
Noting that $ \left\vert Y_{ i+1}-Y_{ i} \right\vert \leq \frac{ c^{ \prime\prime} n}{ p_{ n}^{ 2}}$ almost surely, one can apply Bernstein inequality for martingales \cite{Dzhaparidze2001}, we obtain, for all $t>0$
$\mathbb{ P} \left( \left\vert Y_{ n} \right\vert>t\right) = \mathbb{ P} \left( \left\vert Y_{ n} \right\vert> t,\ \left\langle Y\right\rangle_{ n}\leq L_{ n}\right) \leq 2 \exp \left(- \frac{ 1}{ 2} \frac{ t^{ 2} p_{ n}^{ 2}}{ c^{ \prime}n^{ 2} + \frac{ c^{ \prime\prime}n t}{ 3}}\right)$. Choosing $t= n^{ 1+ \varepsilon}p_{ n}^{\varepsilon-1}$, for $ \varepsilon>0$ (to be fixed later) the bound above becomes
$\mathbb{ P} \left( \left\vert Y_{ n} \right\vert>t\right)\leq 2 \exp \left(- \frac{ 1}{ 2} \frac{ n^{2 \varepsilon} p_{ n}^{ 2 \varepsilon}}{ c^{ \prime} + \frac{ c^{ \prime\prime}n^{\varepsilon} }{ 3p_{ n}^{ 1- \varepsilon}}}\right)$.
Since $ \frac{ n^{ \varepsilon}}{ p_{ n}^{ 1- \varepsilon}}\geq 1$, the previous quantity is bounded by $2 \exp \left(- \frac{ 1}{ 2} \frac{ n^{\varepsilon} p_{ n}^{ 1+\varepsilon}}{ c^{ \prime} + \frac{ c^{ \prime\prime}}{ 3}}\right) $. Since $p_{ n}\geq \frac{ 1}{ n^{ 1- \delta}}$ for $n$ large, we have $n^{ \varepsilon} p_{ n}^{ 1+\varepsilon}\geq n^{ \varepsilon -(1+\varepsilon)(1- \delta)}$. Choosing $ \varepsilon>\frac{ 1- \delta}{ \delta}$, we obtain that $\varepsilon -(1- 3 \varepsilon)(1- \delta)>0$, the above quantity is summable in $n$ and we conclude by Borel-Cantelli Lemma. 
\end{proof}
We now turn to quantities similar to the term $C_{ n}$ in \eqref{eq:def Cn}, for which we apply again Grothendieck inequalities. We recall here the definitions of $S_{ n}^{ \mathcal{ T}}$ in \eqref{eq:Sijik}, \eqref{eq:Sijjk} and Definition~\ref{def:Sn}.
\begin{lemma}
\label{lem:Grothendieck_fonct}
Let $k>3$, $r>0$, $ \Phi:= \left( \varphi_{ p}(\theta_{ 1}, \theta_{ 2})\right)_{ p}$ be a complete orthonormal system in $ H^{r}(\mathbb{ T}^{ 2})$, $F$ an bounded linear operator from $H^{r}$ to $H^{k}$, and $u:= (u_{ i})_{ i=1, \ldots, n}$ an arbitrary vector in $\mathbb{T}^n$. Define the quantities
\begin{align*}
&c^{\ijik}_{ n}(\Phi, F, u):= \frac{ 1}{ n^{ 3}} \sum_{ i,j,k=1}^{ n}\hat{ \xi}_{i j}^{ (n)} \hat{ \xi}_{ik}^{ (n)} \sum_{ p\geq 1} \partial_{ \theta_{ 1}} F(\varphi_{ p}) \left(u_{i}, u_{ j}\right)\partial_{ \theta_{ 1}}\bar F( \varphi_{ p}) \left(u_{i}, u_{ k}\right),\\
&c^{\ikjk}_{ n}(\Phi, F, u):= \frac{ 1}{ n^{ 3}} \sum_{ i,j,k=1}^{ n}\hat{ \xi}_{i k}^{ (n)} \hat{ \xi}_{jk}^{ (n)} \sum_{ p\geq 1} \partial_{ \theta_{ 1}}F( \varphi_{ p}) \left(u_{i}, u_{ k}\right)\partial_{ \theta_{ 1}} \bar F(\varphi_{ p}) \left(u_{j}, u_{ k}\right).
\end{align*}
Then, there exists a constant $C>0$, which is not depending on $n$, such that, for $\cT \in \{\ijik, \ikjk \}$,
\begin{equation}
\label{eq:Grothendieck_fonct}
\sup_{ u\in \mathbb{ T}^{ n}} \left\vert c^{\cT}_{ n}(\Phi, u) \right\vert \leq C S_{ n}^\cT \Vert F\Vert_{\cL(H^r,H^k)}^2.
\end{equation}
\end{lemma}
\begin{proof}[Proof of Lemma~\ref{lem:Grothendieck_fonct}]
We focus on the term with $\cT = \ijik$, but the other terms can be dealt in a similar manner. Write the same decomposition as for Proposition~\ref{prop:bound Cn}: if again $(e_a)_{a \in \Z}$ is the canonical basis of $L^2(\bbT)$ then, for fixed $p\geq 1$, write $\partial_{\theta_1}F( \varphi_{ p})(\theta_1, \theta_2) = \sum_{a \in \Z} e_a(\theta_2) \int_\bbT \partial_{\theta_1} F(\varphi_{ p})(\theta_1, \theta) \bar e_a(\theta) \dd \theta$
as well as, for $a, b\in \mathbb{ Z}$,
\begin{equation*}
\begin{split}
x_{1,i}(a,b) &= \, C^2_\delta \, \big((1+a^2)(1+b^2)\big)^{1/4 + \delta} \\
&\qquad \qquad \times \sum_{ p\geq1}\left( \int_\bbT \partial_{\theta_1}  F(\varphi_{ p})(u_{ i}, \theta) \bar e_a(\theta) \dd \theta \right)  
\left(\int_\bbT \partial_{ \theta_{ 1}} F(\varphi_{ p})(u_{ i}, \theta) \bar e_b(\theta) \dd \theta\right), \\
x_{2,j}(a) &= \, C^{-1}_\delta \, (1+a^2)^{-1/4-\delta} e_a(u_{ j}),\\
x_{3,k}(b) &= \, C^{-1}_\delta \, (1+b^2)^{-1/4 -\delta} e_b(u_{ k}),
\end{split}
\end{equation*}
where $C_\delta = (\sum_{a\in \Z} (1+a^2)^{-1/2-2\delta})^{1/2}$ for some $\delta > 0$. We then have, using Cauchy-Schwarz inequality
\begin{equation}
\label{eq:z1i_gen}
\begin{split}
\norm{x_{1,i}}^2_{\ell^2} &= C^4_\delta \sum_{a,b \in \bbZ}  \big((1+a^2)(1+b^2)\big)^{1/2 + 2\delta}\\
&\qquad \quad \times\left( \sum_{ p\geq1}\left( \int_\bbT \partial_{\theta_1}  F(\varphi_{ p})(u_{ i}, \theta) \bar e_a(\theta) \dd \theta \right) \left(\int_\bbT \partial_{ \theta_{ 1}} F(\varphi_{ p})(u_{ i}, \theta) \bar e_b(\theta) \dd \theta\right) \right)^{ 2}\\
&\leq C^4_\delta \left(\sum_{a\in \bbZ}(1+a^2)^{1/2 + 2\delta} \sum_{ p\geq1}\left\vert \int_\bbT \partial_{\theta_1}  F(\varphi_{ p})(u_{ i}, \theta) \bar e_a(\theta) \dd \theta \right\vert^{ 2} \right)^{ 2}\\
&= C^4_\delta \left(\sum_{a\in \bbZ} (1+a^2)^{1/2 + 2\delta} \left\Vert \mathcal{ I}_{F,a, u_{ i}} \right\Vert_{-r}^{ 2} \right)^2,
\end{split}
\end{equation}
for the linear form
\begin{equation}
\label{eq:Iau}
\mathcal{ I}_{F, a, u}(\varphi):= \int_\bbT \partial_{\theta_1} F( \varphi)(u, \theta) \bar e_a(\theta) \dd \theta.
\end{equation}
Observe that for fixed $a$, $ \left\vert a \right\vert\geq1$, $u\in \mathbb{ T}$ and any $l\geq1$
\begin{equation*}
\begin{split}
0 &= \int_\bbT \partial^l_{\theta} \left[\partial_{\theta_1} F( \varphi)(u, \theta) \bar e_a (\theta)\right] \dd \theta \\
&= \int_{ \mathbb{ T}} [\partial_{ \theta_{ 1}} [\partial_{ \theta_2}^{ l} F(\varphi)(u, \theta)]] \, \bar e_{ a}(\theta) {\rm d}\theta + (ia)^l \int_{ \mathbb{ T}} \partial_{ \theta_{ 1}}F( \varphi)(u, \theta) \, \bar e_{ a}(\theta) {\rm d}\theta\\
&=\int_{ \mathbb{ T}} [\partial_{ \theta_{ 1}} [\partial_{ \theta_2}^{ l} F(\varphi)(u, \theta)]] \, \bar e_{ a}(\theta) {\rm d}\theta + (ia)^l \, \mathcal{ I}_{F, a, u}(\varphi).
\end{split}
\end{equation*}
Taking the absolute values in the previous expression and using Lemma \ref{lem:mixed_sobolev_ineq}, one obtains that there exists a positive constant $C$, independent of $a$ and $u$, such that
\begin{align*}
\left\vert \mathcal{ I}_{ F,a, u} (\varphi)\right\vert &= \frac{ 1}{ \left\vert a \right\vert^{ l}} \left\vert \int_{ \mathbb{ T}} \partial_{ \theta_{ 1}} \partial_{ \theta_{ 2}}^{ l}F( \varphi)(u, \theta) \bar e_{ a}(\theta) {\rm d}\theta\right\vert \leq \frac{ 1}{ \left\vert a \right\vert^{ l}} \sup_{ u\in \mathbb{ T}} \left\Vert  \partial_{ \theta_{ 1}} \partial_{ \theta_{ 2}}^{ l} \varphi \left(u, \cdot\right) \right\Vert_{ L^{ 2}({\rm d}\theta_{ 2})}\\ &\leq \frac{ C}{ \left\vert a \right\vert^{ l}} \left\Vert  \partial_{ \theta_{ 1}} \partial_{ \theta_{ 2}}^{ l} F(\varphi) \left(\cdot, \cdot\right) \right\Vert_{ H^{ 1}({\rm d}\theta_{ 1},{\rm d}\theta_{ 2})} \leq \frac{ C}{ \left\vert a \right\vert^{ l}} \left\Vert F(\varphi) \right\Vert_{ H^{ l+2}({\rm d}\theta_{ 1},{\rm d}\theta_{ 2})}.
\end{align*}
This means that, taking $l=\lfloor k\rfloor -2$,
\begin{equation*}
\sup_{u\in \bbT^n} \left\Vert \mathcal{ I}_{ a, u} \right\Vert_{ -r} \leq \frac{ C}{ \left\vert a \right\vert^{ l}}\Vert F\Vert_{\cL(H^r,H^k)}, \quad \text{for any } a \in \bbZ,\ \left\vert a \right\vert\geq1.
\end{equation*}
Going back to \eqref{eq:z1i_gen}, choosing $\gd$ small enough, this implies that 
\begin{equation*}
\sup_{ i} \norm{x_{1,i}}^2_{\ell^2} \leq C^4_\delta \, C^2 \left( \sum_{a \in \bbZ} |a|^{1+2\delta - 2l} \Vert F\Vert_{\cL(H^r,H^k)}^2\right)^2 < C(l, \delta)\Vert F\Vert_{\cL(H^r,H^k)}^4 <\infty.
\end{equation*}
We are finally in position to apply Grothendieck inequality (Theorem~\ref{th:Grothendieck}), which gives the result. 
\end{proof}

\section{Proof of Theorem~\ref{th:conv_empmeas}}
\label{sec:conv_empmeas}
The proof of convergence follows the same structure for $ \mu^{ n}$, $ \mu^{ n, l}$ and $ \mu^{ n, 1,2}$, since all of them may be written as
\begin{equation}
\label{eq:emp_m}
m_{t}^{ n}:= \frac{ 1}{ n} \sum_{ i=1}^{ n} \Xi_{ i}^{ (n)} \delta_{ \theta_{ t}^{ i, n}},\ t\in[0, T].
\end{equation}
Indeed, $ \mu^{ n}$ corresponds to \eqref{eq:emp_m} for the choice $ \Xi_{ i}^{ (n)}:=1$, $ \mu^{ n, l}$ for the choice $ \Xi_{ i}^{ (n)}:= \frac{ \xi_{li}^{ (n)}}{ p_{ n}}$ whereas $ \mu^{ n, 1, 2}$ satisfies \eqref{eq:emp_m} for $ \Xi_{ i}^{ (n)}:= \frac{ \xi_{1i}^{ (n)} \xi_{2i}^{ (n)}}{ p_{ n}^{ 2}}$. In the following we use the notation
\begin{equation}
\label{eq:Sn_Xi}
S_{ n}\left(\Xi\right):= \frac{ 1}{ n} \sum_{ k=1}^{ n} \Xi_{ k}^{ (n)}.
\end{equation}
Hence, we proceed with the calculations with a general $m^{ n}$ and detail the appropriate changes when required.
Let $ \theta_{ t}$ solution to $ {\rm d}\theta_{ t}= \int \Gamma \left(\theta_{ t}, \theta\right) \mu_{ t} \left({\rm d} \theta\right) {\rm d}t+ {\rm d}B_{ t}$, where $B_{ t}$ is a standard Brownian motion. Define then for $s\leq t$ and any test function $f$, $P_{ s,t}f(\theta):= \mathbf{ E}_{ B} \left[f \left(\Phi_{ s}^{t}(\theta)\right)\right]$, where $t \mapsto \Phi_{ s}^{ t}(\theta)$ is the solution to the previous equation with $ \Phi_{ s}^{ s}(\theta)= \theta$. Straightforward calculations (using Ito's formula and the fact that $s \mapsto P_{ s, t}f$ satisfies a Backward Kolmogorov equation, see e.g. \cite[Lemma~4.3]{lucon_stannat_2014} for more details) show that, for all $f$ regular
\begin{align*}
\big\langle & m_{ T}^{ n}- \mu_{ T}\, ,\, f\big\rangle \\
&= \left\langle m_{ 0}^{ n}- \mu_{ 0}\, ,\, P_{ 0, T}f\right\rangle + \frac{ 1}{ n}\sum_{ k=1}^{ n} \int_{ 0}^{T} \Xi_{k}^{ (n)}\partial_{ \theta}P_{ t, T}f(\theta_{ t}^{ k, n}) {\rm d}B_{t}^{ k}\\ &\quad + \int_{ 0}^{T}\frac{ 1}{ n}\sum_{ i=1}^{ n}  \Xi_{i}^{ (n)}\partial_{ \theta}P_{ t, T}f(\theta_{ t}^{ i, n}) \left( \frac{ 1}{ np_{ n}} \sum_{ j=1}^{ n} \xi_{ij}^{ (n)} \Gamma \left(\theta_{ t}^{ i, n}, \theta_{ t}^{ j, n}\right) - \int \Gamma \left( \theta_{ t}^{ i, n}, \theta\right) \mu_{ t} \left({\rm d}\theta\right)\right) {\rm d}t\\
&=\left\langle m_{ 0}^{ n}- \mu_{ 0}\, ,\, P_{ 0, T}f\right\rangle + \frac{ 1}{ n}\sum_{ k=1}^{ n} \int_{ 0}^{T} \Xi_{k}^{ (n)}\partial_{ \theta}P_{ t, T}f(\theta_{ t}^{ k, n}) {\rm d}B_{t}^{ k}\\ &\quad + \int_{ 0}^{T}\frac{ 1}{ n^{ 2}}\sum_{ i,j=1}^{ n}\Xi_{i}^{ (n)}\hat{ \xi}_{ij}^{ (n)}\partial_{ \theta}P_{ t, T}f(\theta_{ t}^{ i, n})  \Gamma \left(\theta_{ t}^{ i, n}, \theta_{ t}^{ j, n}\right) {\rm d}t\\ &\quad + \int_{ 0}^{T}\frac{ 1}{ n}\sum_{ i=1}^{ n}  \Xi_{i}^{ (n)}\partial_{ \theta}P_{ t, T}f(\theta_{ t}^{ i, n}) \left\langle \Gamma \left(\theta_{ t}^{ i, n},\cdot\right)\, ,\, \mu_{ t}^{ n}- \mu_{ t}\right\rangle {\rm d}t,
\end{align*}
so that
\begin{align}
\frac{ 1}{ 4^{ q-1}} \left\vert \left\langle  m_{ T}^{ n}- \mu_{ T}\, ,\, f\right\rangle \right\vert^{ q}&\leq \left\vert \left\langle m_{ 0}^{ n}- \mu_{ 0}\, ,\, P_{ 0, T}f\right\rangle \right\vert^{ q} + \left\vert \frac{ 1}{ n}\sum_{ k=1}^{ n} \int_{ 0}^{T} \Xi_{k}^{ (n)}\partial_{ \theta}P_{ t, T}f(\theta_{ t}^{ k, n}) {\rm d}B_{t}^{ k} \right\vert^{ q}\nonumber\\ &\quad+T^{ q-1} \int_{ 0}^{T} \left\vert \frac{ 1}{ n^{ 2}}\sum_{ i,j=1}^{ n} \Xi_{i}^{ (n)}\hat{ \xi}_{ij}^{ (n)}\partial_{ \theta}P_{ t, T}f(\theta_{ t}^{ i, n})  \Gamma \left(\theta_{ t}^{ i, n}, \theta_{ t}^{ j, n}\right) \right\vert^{ q} {\rm d}t\nonumber\\&\quad  + T^{ q-1}\int_{ 0}^{T} \left\vert \frac{ 1}{ n}\sum_{ i=1}^{ n}  \Xi_{i}^{ (n)}\partial_{ \theta}P_{ t, T}f(\theta_{ t}^{ i, n}) \left\langle \Gamma \left(\theta_{ t}^{ i, n},\cdot\right)\, ,\, \mu_{ t}^{ n}- \mu_{ t}\right\rangle \right\vert^{ q} {\rm d}t\nonumber\\&:=(A)+(B)+(C)+(D).\label{eq:dec_munl}
\end{align}
Consider the initial condition $(A)$: writing $ \frac{ \xi_{li}}{ p_{ n}}= \hat{ \xi}_{l i} +1$, we see that $\mu_{ 0}^{ n, l}= \mu_{ 0}^{ n} + \hat{ \mu}_{ 0}^{ n ,l} $ (recall the definition of $\hat{\mu}_{ t}^{ n,l}$ in \eqref{def:hat_mu_i_n}) and that $ \mu_{ 0}^{ n, 1, 2}= \mu_{ 0}^{ n} + \hat{ \mu}_{ 0}^{ n ,1}+\hat{ \mu}_{ 0}^{ n ,2} + \hat{ \mu}_{ 0}^{ n ,1,2}$ where
\begin{equation*}
\hat{ \mu}_{ 0}^{ n,1,2}= \frac{ 1}{ n} \sum_{i=1}^{ n} \hat{ \xi}_{1,i}^{ (n)} \hat{ \xi}_{2i}^{ (n)} \delta_{ \theta_{ 0}^{ i,n}}.
\end{equation*}
Thus, setting 
\begin{equation*}
\hat{ m}_{ 0}^{ n}:=\begin{cases}
0 & \text{ if } m^{ n}= \mu^{ n},\\
 \hat{ \mu}_{ 0}^{ n ,l} & \text{ if } m^{ n}= \mu^{ n, l},\\
\hat{ \mu}_{ 0}^{ n ,1}+\hat{ \mu}_{ 0}^{ n ,2} + \hat{ \mu}_{ 0}^{ n ,1,2} & \text{ if } m^{ n}= \mu^{ n,1,2}
\end{cases}
\end{equation*}
we have in all cases that (noting that if $f\in BL= BL( \mathbb{ T})$ (recall \eqref{eq:dBL}), $P_{ 0, T}f$ is also in $BL$)
\begin{align*}
(A)=\left\vert \left\langle m_{ 0}^{ n}- \mu_{ 0}\, ,\, P_{ 0, T}f\right\rangle \right\vert^{ q} & \leq 2^{ q-1} d_{ BL} \left(\mu_{ 0}^{ n}, \mu_{ 0}\right)^{ q} + 2^{ q-1} \left\vert \left\langle \hat{m}_{ 0}^{ n}\, ,\, P_{ 0, T}f\right\rangle \right\vert^{ q}
\end{align*}
Consider now the term $(B)$: note that (\cite[Lemma~4.4]{lucon_stannat_2014}), there exists a constant $C_{ 0}>0$ such that uniformly in $ \theta$, $t\leq T$ and in $f$ such that $ \left\Vert f \right\Vert_{ Lip}\leq 1$, $ \left\vert \partial_{ \theta}P_{ t, T}f(\theta) \right\vert \leq C_{ 0}$. Hence, we see that $\mathbf{ E} \left[(B)\right] \leq C \left( \frac{ 1}{ n^{ 2}} \sum_{ k=1}^{ n} \left(\Xi_{k}^{ (n)}\right)^{ 2}\right)^{ \frac{ q}{ 2}}$, which is, $ \mathbb{ P}$-a.s., uniformly in $f\in BL$, smaller than $ C\beta_{ n}^{ q}$ with 
\begin{equation*}
\beta_{ n}:= \begin{cases}
 n^{ -\frac{ 1}{ 2}} & \text{ if } m^{ n}= \mu^{ n},\\
\left(np_{ n}^{ 2}\right)^{ -\frac{ 1}{ 2}} & \text{ if } m^{ n}= \mu^{ n, l},\\
\left(np_{ n}^{ 4}\right)^{ -\frac{ 1}{ 2}} & \text{ if } m^{ n}= \mu^{ n,1,2}
\end{cases}
\end{equation*} 
Consider now the term $(C)$ in \eqref{eq:dec_munl}: using again that $ \left\vert \partial_{ \theta}P_{ t, T}f(\theta) \right\vert \leq C_{ 0}$ and since $ \frac{ \xi_{li}^{ (n)}}{ p_{ n}}\leq \frac{ 1}{ p_{ n}}$, we see by another application of Grothendieck inequality that, $ \mathbb{ P}\otimes \mathbf{ P}$-a.s., $(C)$ is uniformly controlled by $C \gamma_{ n}^{ q} (S_{ n}^{ \ij})^{ q}$, where 
\begin{equation*}
\gamma_{ n}:= \begin{cases}
1 & \text{ if } m^{ n}= \mu^{ n},\\
 \frac{ 1}{ p_{ n}} & \text{ if } m^{ n}= \mu^{ n, l},\\
\frac{ 1}{ p_{ n}^{ 2}} & \text{ if } m^{ n}= \mu^{ n,1,2}.
\end{cases}
\end{equation*} 
Concentrate now on the last term $(D)$ in \eqref{eq:dec_munl}.
Developing into Fourier series gives, for $e_{ a}, a\in \mathbb{ Z}$ the standard Fourier basis, for $ \theta_{ 1}, \theta_{ 2}\in \mathbb{ T}$, $ \Gamma \left(\theta_{ 1}, \theta_{ 2}\right) = \sum_{ a\in \mathbb{ Z}} e_{ a} \left(\theta_{ 2}\right) \int \Gamma \left(\theta_{ 1}, \theta\right) \bar e_{ a}(\theta) {\rm d}\theta$. Note that $ \left\vert  \int \Gamma \left(\theta_{ 1}, \theta\right) \bar e_{ a}(\theta) {\rm d}\theta \right\vert \leq  \frac{ C}{ (1+ \left\vert a \right\vert)^{ r}}$ for some constant $C>0$ independent of $ \theta_{ 1}$, as $ \left\Vert \partial_{ \theta_{ 2}}^{ r} \Gamma(\theta_{ 1}, \theta_{ 2}) \right\Vert_{ \infty}<+\infty$, which means that
\begin{equation*}
\left\vert \left\langle \Gamma \left(\theta_{ t}^{ i, n},\cdot\right)\, ,\, \mu_{ t}^{ n}- \mu_{ t}\right\rangle \right\vert \leq C\sum_{ a\in \mathbb{ Z}} (1+ \left\vert a \right\vert)^{ -r}\left\vert \left\langle e_{ a}  \, ,\, \mu_{ t}^{ n}- \mu_{ t}\right\rangle \right\vert.
\end{equation*}
So we obtain, by Jensen inequality, recalling that $ \left\vert \partial_{ \theta}P_{ t, T}f(\theta) \right\vert \leq C_{ 0}$ and the definition of $S_{ n} \left(\Xi\right)$ in \eqref{eq:Sn_Xi},
\begin{equation*}
(D) \leq C\left\vert S_{ n} \left(\Xi\right)\right\vert^{ q}\left(\sum_{ a\in \mathbb{ Z}} (1+ \left\vert a \right\vert)^{ -r}\right)^{ q-1} \int_{ 0}^{T} \left(\sum_{ a\in \mathbb{ Z}} (1+ \left\vert a \right\vert)^{ -r}\left\vert \left\langle e_{ a}  \, ,\, \mu_{ t}^{ n}- \mu_{ t}\right\rangle \right\vert^{ q}\right){\rm d}t.
\end{equation*}
Taking the expectation on both sides and noting that for any fixed $a\in \mathbb{ Z}$, $ e_{ a}$ is bounded and Lipschitz with constant equal to $ \left\vert a \right\vert$, we obtain
\begin{multline*}
\mathbf{ E}\left[(D)\right]\leq  C\left\vert S_{ n} \left(\Xi\right)\right\vert^{ q} \left(\sum_{ a\in \mathbb{ Z}} (1+ \left\vert a \right\vert)^{ -r}\right)^{ q-1}\left(\sum_{ a\in \mathbb{ Z}} (1+ \left\vert a \right\vert)^{ -r} \left\vert a \right\vert^{ q} \right)  \\
\times\int_{ 0}^{T} \sup_{ f\in BL}\mathbf{ E} \left[ \sup_{ s\leq t}\left\vert \left\langle f  \, ,\, \mu_{ s}^{ n}- \mu_{ s}\right\rangle \right\vert^{ q}\right] {\rm d}t.
\end{multline*}
Choosing $r= q+2$, we deduce finally that there is another constant $C>0$ such that  
\begin{equation*}
\mathbf{ E} \left[(D)\right] \leq C \left\vert S_{ n} \left(\Xi\right)\right\vert^{ q}\int_{ 0}^{T}\mathbf{ E} \left[ \sup_{ s\leq t}d_{ BL} \left(\mu_{ s}^{ n}, \mu_{ s}\right)^{ q}\right] {\rm d}t. 
\end{equation*}
Taking expectation $ \mathbf{ E} \left[\cdot\right]$ in \eqref{eq:dec_munl}, we obtain, for all $f\in BL$, for some constant $C>0$
\begin{equation}
\label{eq:munl_mu_f}
\begin{split}
\mathbf{ E} \left[ \sup_{ s\leq T}\left\vert \left\langle m^n_s- \mu_{ s}\, ,\, f\right\rangle \right\vert^{ q}\right]  \leq C \Bigg( 2^{ q-1} d_{ BL} \left(\mu_{ 0}^{ n}, \mu_{ 0}\right)^{ q} + 2^{ q-1} \left\vert \left\langle \hat{m}_{ 0}^{ n}\, ,\, P_{ 0, T}f\right\rangle \right\vert^{ q}+  \beta_{ n}^{ q}+ (\gamma_{ n}S_{ n}^{ \ij})^{ q} \\+  \left\vert S_{ n} \left(\Xi\right)\right\vert^{ q} \int_{ 0}^{T} \mathbf{ E} \left[ \sup_{ s\leq t}d_{ BL} \left(\mu_{ s}^{ n}, \mu_{ s}\right)^{ q}\right] {\rm d}t\Bigg).
\end{split}
\end{equation}
Specify first the analysis to the case $m^{ n}= \mu^{ n}$: recalling that $ \hat{ m}_{ 0}^{ n}\equiv 0$ and $S_{ n} \left(\Xi\right)=1$ in this case, one obtains
\begin{multline}\label{eq:munl_mu_f_mu}
\mathbf{ E} \left[ \sup_{ s\leq T}\left\vert \left\langle \mu^n_s- \mu_{ s}\, ,\, f\right\rangle \right\vert^{ q}\right]\\  \leq C \Bigg( 2^{ q-1} d_{ BL} \left(\mu_{ 0}^{ n}, \mu_{ 0}\right)^{ q} +  \beta_{ n}^{ q}+ (\gamma_{ n}S_{ n}^{ \ij})^{ q} +  \int_{ 0}^{T} \mathbf{ E} \left[ \sup_{ s\leq t}d_{ BL} \left(\mu_{ s}^{ n}, \mu_{ s}\right)^{ q}\right] {\rm d}t\Bigg).
\end{multline}
If we would have been able to put the supremum in $f\in BL$ inside the expectation in the lefthand side of \eqref{eq:munl_mu_f_mu}, the result would follow simply by a Gr\"onwall argument. To bypass this difficulty, we proceed by a compactness argument, which is will be useful not only to $ m^{ n}= \mu^{ n}$ but to the other cases too, so that we write it with a general $ m^{ n}$:
the set $ BL$ is compact, by Ascoli-Arzel\`a theorem. Thus, for all $ \varepsilon>0$, there exists $f_{ 1}, \ldots, f_{ k}\in BL$ such that for all $f\in BL$, there exists $j=1, \ldots, k$ such that $\sup_{ \theta\in \mathbb{ T}} \left\vert f(\theta)- f_{ j}(\theta) \right\vert\leq \varepsilon$. 
Take now $f\in BL$, $n\geq1$, we have,
\begin{align*}
\frac{ 1}{ 3^{ q-1}} \sup_{ s\leq T}\left\vert \left\langle m_{ s}^{ n}- \mu_{ s}\, ,\, f\right\rangle \right\vert^{ q} \leq & \sup_{ s\leq T}\left\vert \left\langle m_{ s}^{ n}- \mu_{ s}\, ,\, f_{ j}\right\rangle \right\vert^{ q} + \sup_{ s\leq T}\left\vert \left\langle m_{ s}^{ n}\, ,\, f-f_{ j}\right\rangle \right\vert^{ q}\\&+  \sup_{ s\leq T}\left\vert \left\langle \mu_{ s}\, ,\, f-f_{ j}\right\rangle \right\vert^{ q}.
\end{align*}
Note that since $ \mu_{ s}$ is a probability measure, $\sup_{ s\leq T}\left\vert \left\langle \mu_{ s}\, ,\, f-f_{ j}\right\rangle \right\vert^{ q} \leq \varepsilon^{ q}\leq \varepsilon$ for $ \varepsilon\leq 1$. In a same way, $\sup_{ s\leq T}\left\vert \left\langle m_{ s}^{ n}\, ,\,f- f_{ j}\right\rangle \right\vert^{ q} \leq  S_{ n} \left(\Xi\right)^{ q} \varepsilon$, $ \mathbb{ P}_{ g}$-a.s. (and this $ \mathbb{ P}_{ g}$-a.s. does not depend on $f$, $g_{ j}$, nor $ \varepsilon$). Hence, there is a universal constant $C>0$ such that, for any $t\in [0, T]$
\begin{equation}
\label{eq:dBL_vs_f}
\mathbf{ E} \left[ \sup_{ s\leq t} d_{ BL} \left( m_{ s}^{ n}, \mu_{ s}\right)^{ q}\right] \leq C \left(1+ S_{ n} \left(\Xi\right)^{ q}\right) \varepsilon + \max_{ j=1, \ldots, k} \mathbf{ E} \left[ \sup_{ s\leq t}\left\vert \left\langle m_{ s}^{ n}- \mu_{ s}\, ,\, f_{ j}\right\rangle \right\vert^{ q} \right]
\end{equation}
Apply once \eqref{eq:dBL_vs_f} to the righthand side of \eqref{eq:munl_mu_f_mu} (with $m^{ n}= \mu^{ n}$), then take $f=f_{ j}$ and finally the maximum over $j=1,\ldots, k$ in \eqref{eq:munl_mu_f} we obtain, 
\begin{equation*}
\begin{split}
\max_{ j=1, \ldots, k}\mathbf{ E} \left[ \sup_{ s\leq T}\left\vert \left\langle \mu^n_s- \mu_{ s}\, ,\, f_{ j}\right\rangle \right\vert^{ q}\right]  \leq C \Bigg( 2^{ q-1} d_{ BL} \left(\mu_{ 0}^{ n}, \mu_{ 0}\right)^{ q} +  \beta_{ n}^{ q}+ (\gamma_{ n}S_{ n}^{ \ij})^{ q} +  2CT \varepsilon \\ +  \int_{ 0}^{T} \max_{ j=1, \ldots, k} \mathbf{ E} \left[ \sup_{ s\leq t}\left\vert \left\langle \mu_{ s}^{ n}- \mu_{ s}\, ,\, f_{ j}\right\rangle \right\vert^{ q} \right]
{\rm d}t\Bigg).
\end{split}
\end{equation*}
Taking $\limsup_{ n\to\infty}$ on both sides, using \eqref{eq:conv_mu0}, the fact that both $ \beta_{ n}$ and $ \gamma_{ n}S_{ n}^{ \ij}$ go to $0$ as $n\to\infty$ under the present assumptions, we obtain, setting
\begin{equation*}
v_{ t}:= \limsup_{ n\to\infty} \max_{ j=1, \ldots, k}\mathbf{ E} \left[ \sup_{ s\leq t}\left\vert \left\langle \mu^n_s- \mu_{ s}\, ,\, f_{ j}\right\rangle \right\vert^{ q}\right]
\end{equation*} 
that  $v_{ T}  \leq 2C^{ 2}T \varepsilon  + C \int_{ 0}^{T} v_{ t}{\rm d}t$, so that by Gr\"onwall Lemma, $v_{ T}\leq C^{ \prime} \varepsilon$, for a constant $C^{ \prime}>0$ that only depends on $( \Gamma, T)$. Inserting this estimate into \eqref{eq:dBL_vs_f} gives finally that $ \limsup_{ n\to\infty}\mathbf{ E} \left[ \sup_{ s\leq t} d_{ BL} \left( \mu_{ s}^{ n}, \mu_{ s}\right)^{ q}\right] \leq (C+ C^{ \prime}) \varepsilon $ for all $ \varepsilon$, which gives the desired convergence \eqref{eq:conv_empmeas}. 

We now turn to the proof of the convergence of the local empirical measures: combining \eqref{eq:dBL_vs_f} with \eqref{eq:munl_mu_f} applied to $f=f_{ j}$ and taking now advantage that we know that \eqref{eq:conv_empmeas} is true, we obtain
\begin{align*}
\limsup_{ n\to\infty}\mathbf{ E} \left[ \sup_{ s\leq T} d_{ BL} \left( m_{ s}^{ n}, \mu_{ s}\right)^{ q}\right] &\leq C \varepsilon + \limsup_{ n\to\infty}\max_{ j=1, \ldots, k} \left\vert \left\langle \hat{ m}_{ 0}^{ n}\, ,\, P_{ 0, T}f_{ j}\right\rangle \right\vert^{ q}.
\end{align*}
Note here that it is not sufficient for us to directly apply Grothendieck inequality to the remaining term (since it is only saying that this term is bounded, not that it goes to $0$). 
The point here is to note that $ f_{ j}\in BL$ so that $ g_{ j}:=P_{ 0, T} f_{ j}\in BL$ too. Hence, there exists some $j^{ \prime}$, such that $ \left\Vert g_{ j} - f_{ j^{ \prime}}\right\Vert_{ \infty} \leq \varepsilon$. Writing again $ g_{ j}= g_{ j}- f_{ j^{ \prime}}+ f_{ j^{ \prime}}$, we obtain further that 
\begin{align}
\label{aux:dist_munl}
\limsup_{ n\to\infty}\mathbf{ E} \left[ \sup_{ s\leq T} d_{ BL} \left( m_{ s}^{ n}, \mu_{ s}\right)^{ q}\right] &\leq C \varepsilon + \limsup_{ n\to\infty}\max_{ j^{ \prime}=1, \ldots, k} \left\vert \left\langle \hat{ m}_{ 0}^{ n}\, ,\,f_{ j^{ \prime}}\right\rangle \right\vert^{ q}.
\end{align}
for another constant $C>0$.
Recall that for any test function $f$ (recall \eqref{eq:Uij1} and \eqref{eq:Uij12}), \begin{equation*}
\left\langle \hat{ \mu}_{ 0}^{ n, l}\, ,\, f\right\rangle= \frac{ 1}{ n} \sum_{ j=1}^{ n} \hat{ \xi}_{lj}^{ (n)} f(\theta_{ 0}^{ j,n})= U_{ n, 1}^{\ij}(l, v),
\end{equation*}
and
\begin{equation*}
\left\langle \hat{ \mu}_{ 0}^{ n, 1,2}\, ,\, f\right\rangle= \frac{ 1}{ n} \sum_{ j=1}^{ n} \hat{ \xi}_{1,j}^{ (n)} \hat{ \xi}_{2j}^{ (n)}f(\theta_{ 0}^{ j,n})= U_{ n, 1}^{\ikjk}(l, v),
\end{equation*}
for the choice of $v_{ j}= f \left( \theta^n_{ j,0}\right)$. Using \eqref{eq:sum_hat_xi_lj} and \eqref{eq:sum_hat_xi_12j}, we obtain that in any case $ \mathbb{P}_{ g}$-a.s., $\max_{ j^{ \prime}=1, \ldots, k} \left\vert \left\langle \hat{ m}_{ 0}^{ n }\, ,\,f_{ j^{ \prime}}\right\rangle \right\vert^{ q}\to 0$ as $n\to\infty$. Note however that this $ \mathbb{ P}_{ g}$-a.s. depends on the choice of the functions $f_{ j}$ and thus on $ \varepsilon$. Taking now $ \varepsilon$ of the form $ \varepsilon= \frac{ 1}{ p}$ with $p\geq1$, we have from \eqref{aux:dist_munl} and the previous argument that $ \mathbb{ P}_{ g}$-a.s., \begin{equation*}
\limsup_{ n\to\infty}\mathbf{ E} \left[ \sup_{ s\leq T} d_{ BL} \left( m_{ s}^{ n}, \mu_{ s}\right)^{ q}\right]\leq \frac{ C}{ p},
\end{equation*}
for any $p\geq1$, which concludes the proof.
\qedhere

\section{Uniqueness results}
\label{sec:uniqueness}
Let us introduce some notation: define 
\begin{equation*}
\begin{split}
    \Lambda_{ s}(\theta_{ 1}, \theta_{ 2}) &:= \left(\left\langle \mu_{ s}({\rm d}\theta^{ \prime})\, ,\, \Gamma (\theta_{ 1}, \theta^{ \prime})\right\rangle , \left\langle \mu_{ s}({\rm d}\theta^{ \prime})\, ,\, \Gamma (\theta_{ 2}, \theta^{ \prime})\right\rangle \right),\\
    \cL^{(1)}_s &:= \cL^{(1)}_{\mu_s}, \\
    \cL^{(2)}_s &:= \cL^{(2)}_{\mu_s},
\end{split}
\end{equation*}
so that the operator $ \mathcal{ L}_{ s}^{(2)}$ (recall its definition in \eqref{eq:L2}) may be written as
\begin{equation*}
\mathcal{ L}_{ \mu_s}^{(2)}(g)= \frac{ 1}{ 2} \Delta g + \nabla g \cdot \Lambda_{ s} (\theta_{ 1}, \theta_{ 2}).
\end{equation*}
Since $ \Gamma$ is regular with bounded derivatives, so is $ (\theta_{ 1}, \theta_{ 2}) \mapsto \Lambda_{ s}(\theta_{ 1}, \theta_{ 2})$ with derivatives that are bounded uniformly in $s\in[0, T]$. Applying \cite{MR876080}, p.227, the flow $(X_{ s,t}( \theta))_{ 0\leq s\leq t \leq T}$ is a $ \mathcal{ C}^{r_{ 2}+2}$ diffeomorphism, where $X_{ s,t}(\theta)$ is the unique solution to the Itô SDE in $ \mathbb{ T}^{ 2}$
\begin{equation*}
X_{ s,t}(\theta)= \theta + \int_{ s}^{t} \Lambda_{ r}(X_{ s,t}(\theta)) {\rm d}r + B_{ t}- B_{ s},
\end{equation*}
where $ \theta:=(\theta_{ 1}, \theta_{ 2})\in \mathbb{ T}^{ 2}$ and $B:=(B^{1}, B^{ 2})$ is a standard Brownian motion on $ \mathbb{ T}^{ 2}$. Using this and backward Itô formula \cite{MR876080}, p~256, it is possible to prove (see \cite{Jourdain1998}, p.~760 for further details) that setting
\begin{equation}
\label{eq:Uts}
U(t,s)g(\theta_{ 1}, \theta_{ 2}):= \mathbf{ E}_{ B} \left[ g(X_{ s,t}(\theta))\right],
\end{equation}
one obtains, for all $g\in \mathcal{ C}^{ 2}$, all $0\leq s \leq t \leq T$, $ \theta\in \mathbb{ T}^{ 2}$
\begin{equation*}
U(t,s)g(\theta)- g(\theta) = \int_{ s}^{t} \mathcal{ L}_{ r}^{ (2)}U(t,r)g(\theta) {\rm d}r.
\end{equation*}
The next step is to prove that the previous equality is also valid in the space $ \mathcal{ C}^{ r_{ 2}}$. This relies on the following lemma (see \cite{Jourdain1998}, Lemma~3.11 for a proof of this result)
\begin{lemma}
\label{lem:Uts}
Under the assumptions of Section~\ref{sec:hypotheses}, for any probability measure $\nu$ the operator $ \mathcal{ L}_{ \nu}^{ (2)}$ is continuous from $ \mathcal{ C}^{ r_{ 2}+2}$ to $ \mathcal{C}^{ r_{ 2}}$ and
\begin{align*}
\left\Vert \mathcal{ L}_{ \nu}^{ (2)} g \right\Vert_{ \mathcal{ C}^{ r_{ 2}}}&\leq C \left\Vert g \right\Vert_{ \mathcal{ C}^{ r_{ 2}+2}},\\
\left\Vert \mathcal{ L}_{ s}^{ (2)} g - \mathcal{ L}_{ t}^{ (2)}g\right\Vert_{ \mathcal{ C}^{ r_{ 2}}} &\leq C \left\Vert g \right\Vert_{ \mathcal{ C}^{ r_{ 2}+2}} \left\vert t-s \right\vert.
\end{align*}
as well as, for any $j\leq r_{ 2}+2$ the operator $U(t,s)$ is a linear operator from $ \mathcal{ C}^{ j}$ to $ \mathcal{ C}^{ j}$ such that
\begin{align*}
\left\Vert U(t, s) g \right\Vert_{ \mathcal{ C}^{ j}} &\leq C \left\Vert g \right\Vert_{ \mathcal{ C}^{ j}},\ 0\leq s \leq t \leq T,\\
\left\Vert U(t,s) g - U(t, s^{ \prime})g \right\Vert_{  \mathcal{ C}^{ j}} &\leq C \left\Vert g \right\Vert_{ \mathcal{ C}^{ j+1}} \sqrt{ s^{ \prime}-s},\ 0\leq s \leq s^{ \prime}\leq t \leq T.
\end{align*}
\end{lemma}
In particular, we know that for $g\in \mathcal{ C}^{r_{ 2}+3}$, $ s \mapsto \mathcal{ L}_{ s}^{ (2)} \left(U(t,s) g\right)$ is continuous in $ \mathcal{ C}^{r_{ 2}}$ and hence that $ \int_{ 0}^{t} \mathcal{ L}_{ s}^{ (2)} (U(t,s)g) {\rm d}s$ makes sense as a Bochner integral in $ \mathcal{ C}^{ r_{ 2}}$. In particular, we obtain, for every $g\in \mathcal{ C}^{ r_{ 2}+3}$ that 
\begin{equation}
\label{eq:backward_U}
U(t,s)(g)- g = \int_{ s}^{t} \mathcal{ L}_{ r}^{ (2)}U(t,r)g {\rm d}r,\ \text{ in } \mathcal{ C}^{ r_{ 2}}.
\end{equation}
With this at hand, we are ready to state the first uniqueness result:
\begin{proposition}
\label{prop:SPDE_L2}
Under the assumptions of Section~\ref{sec:hypotheses}, for any functional $R$ belonging to $\mathcal{ C} \left([0, T], H^{ -r_{ 2}}\left(\mathbb{ T}^{ 2}\right)\right)$, there is at most one solution in $ \mathcal{ C} \left([0, T], H^{ -r_{ 1}}\left(\mathbb{ T}^{ 2}\right)\right)$ to the equation
\begin{equation}
\label{eq:SPDE_L2}
\mathcal{ E}_{ t}= \int_{ 0}^{t} \mathcal{ L}_{ s}^{ (2), \ast} \mathcal{ E}_{ s} {\rm d}s + \int_{ 0}^{t} R_{ s} {\rm d}s,\ t\in [0, T], \text{ in }H^{ -r_{ 2}}\left(\mathbb{ T}^{ 2}\right).
\end{equation}
Moreover, one has the representation
\begin{equation*}
\mathcal{ E}_{ t}= \int_{ 0}^{t} U(t,s)^{ \ast} R_{ s} {\rm d}s, \text{ in } \mathcal{ C}^{ -r_{ 2}},
\end{equation*}
where $U$ is given in \eqref{eq:Uts}.
\end{proposition}
\begin{proof}
Let $ \mathcal{ E}$ a solution to \eqref{eq:SPDE_L2} in $ \mathcal{ C} \left([0, T], H^{ -r_{ 1}}\left(\mathbb{ T}^{ 2}\right)\right)$. Since $H^{ -r_{ 1}} \hookrightarrow H^{ -r_{ 2}}$, we have for all $g\in H^{ r_{ 2}} \left(\mathbb{ T}^{ 2}\right)$
\begin{equation*}
\left\langle \mathcal{ E}_{ t}\, ,\, g\right\rangle = \int_{ 0}^{t} \left\langle \mathcal{ E}_{ s}\, ,\, \mathcal{ L}_{ s}^{ (2)} g\right\rangle {\rm d}s + \int_{ 0}^{t} \left\langle R_{ s}\, ,\, g\right\rangle {\rm d}s.
\end{equation*}
Both relations are in particular true for every $g\in \mathcal{ C}^{r_{ 2}+3} \hookrightarrow H^{ r_{ 2}}$. Combining \eqref{eq:backward_U} and \eqref{eq:SPDE_L2}, we obtain for $g\in \mathcal{ C}^{r_{ 2}+3}$,
\begin{align*}
\left\langle \mathcal{ E}_{ t}\, ,\, g\right\rangle &=\int_{ 0}^{t} \left\lbrace \left\langle \mathcal{ E}_{ s}\, ,\, \mathcal{ L}_{ s}^{ (2)} U(t,s)g\right\rangle + \left\langle R_{ s}\, ,\, U(t,s)g \right\rangle\right\rbrace {\rm d}s \\&\quad - \int_{ 0}^{t} \int_{ s}^{t} \left\langle \mathcal{ E}_{ s}\, ,\, \mathcal{ L}_{ s}^{ (2)} \mathcal{ L}_{ r}^{ (2)}U(t,r)g\right\rangle {\rm d}r {\rm d}s - \int_{ 0}^{t} \int_{ s}^{t} \left\langle R_{ s}\, ,\, \mathcal{ L}_{ r}^{ (2)}U(t,r)g \right\rangle {\rm d}r{\rm d}s\\
&=\int_{ 0}^{t} \left\lbrace \left\langle \mathcal{ E}_{ s}\, ,\, \mathcal{ L}_{ s}^{ (2)} U(t,s)g\right\rangle + \left\langle R_{ s}\, ,\, U(t,s)g \right\rangle\right\rbrace {\rm d}s \\&\quad - \int_{ 0}^{t} \left\lbrace\int_{ 0}^{r} \left\langle \mathcal{ E}_{ s}\, ,\, \mathcal{ L}_{ s}^{ (2)} \mathcal{ L}_{ r}^{ (2)}U(t,r)g\right\rangle {\rm d}s  + \int_{ 0}^{r} \left\langle R_{ s}\, ,\, \mathcal{ L}_{ r}^{ (2)}U(t,r)g \right\rangle {\rm d}s\right\rbrace {\rm d}r\\
&=\int_{ 0}^{t} \left\lbrace \left\langle \mathcal{ E}_{ s}\, ,\, \mathcal{ L}_{ s}^{ (2)} U(t,s)g\right\rangle + \left\langle R_{ s}\, ,\, U(t,s)g \right\rangle\right\rbrace {\rm d}s - \int_{ 0}^{t} \left\langle \mathcal{ E}_{ r}\, ,\, \mathcal{ L}_{ r}^{ (2)}U(t,r)g\right\rangle {\rm d}r\\
&=\int_{ 0}^{t}  \left\langle R_{ s}\, ,\, U(t,s)g \right\rangle {\rm d}s. 
\end{align*}
Since $ \mathcal{ C}^{r_{ 2}+3}$ is dense in $ \mathcal{ C}^{ r_{ 2}}$, the identity $ \mathcal{ E}_{ t}= \int_{ 0}^{t} U(t,s)^{ \ast} R_{ s} {\rm d}s$ holds in $ \mathcal{ C}^{ -r_{ 2}}$. Since $ \mathcal{ C}^{ r_{ 2}}$ is dense in $ H^{ r_{ 1}}$, uniqueness for \eqref{eq:SPDE_L2} holds in $ \mathcal{ C} \left([0, T], H^{ -r_{ 1}}(\mathbb{ T}^{ 2})\right)$.
\end{proof}
In an identical way, one can state a similar result concerning $ \mathcal{ L}_{ s}^{ (1)}$, the only difference being that $ \mathcal{ L}_{ s}^{ (1)}$ is not the generator of a diffusion, due to the nonstandard nonlocal last term in \eqref{eq:L1}. Hence, we decompose $ \mathcal{ L}_{ s}^{ (1)}$ into
\begin{equation}
\label{eq:decomp_L1}
\mathcal{ L}_{ s}^{ (1)}:= L_{ s}^{(1)} + K_{ s}^{ (1)},
\end{equation}
where
\begin{align*}
L_{ s}^{ (1)}f&:=  \frac{ 1}{ 2} \partial_{ \theta}^{ 2} f(\theta) + \partial_{ \theta}f(\theta)\left\langle \mu_{ s}({\rm d} \theta^{ \prime})\, ,\, \Gamma(\theta, \theta^{ \prime}) \right\rangle,\\
K_{ s}^{ (1)}f&:= \left\langle \mu_{ s}({\rm d}\theta^{ \prime})\, ,\, \Gamma \left(\theta^{ \prime}, \theta\right)\partial_{ \theta}f(\theta^{ \prime})\right\rangle .
\end{align*}
Setting $v_{ s}(\theta):= \left\langle \mu_{ s}({\rm d} \theta^{ \prime})\, ,\, \Gamma(\theta, \theta^{ \prime}) \right\rangle$ define in a similar way the flow $(X_{ s,t}( \theta))_{ 0\leq s\leq t \leq T}$ as the unique solution to the Itô SDE in $ \mathbb{ T}$
\begin{equation*}
X_{ s,t}(\theta)= \theta + \int_{ s}^{t} v_{ r}(X_{ s,t}(\theta)) {\rm d}r + B_{ t}- B_{ s},
\end{equation*}
where $ \theta\in \mathbb{ T}$ and $B$ is a standard Brownian motion on $ \mathbb{ T}$. Define also
\begin{equation}
\label{eq:Vts}
V(t,s)f(\theta):= \mathbf{ E}_{ B} \left[ f(X_{ s,t}(\theta))\right],
\end{equation}
one obtains, in a same way as before the following result (whose proof is left to the reader)
\begin{proposition}
\label{prop:SPDE_L1}
Under the assumptions of Section~\ref{sec:hypotheses}, for any functional $R$ element of $ \mathcal{ C} \left([0, T], H^{ -r_{ 1}}\left(\mathbb{ T}\right)\right)$, there is at most one solution in $ \mathcal{ C} \left([0, T], H^{ -r_{ 1}}\left(\mathbb{ T}\right)\right)$ to the equation
\begin{equation}
\label{eq:SPDE_L1}
\mathcal{ E}_{ t}= \int_{ 0}^{t}L_{ s}^{ (1), \ast} \mathcal{ E}_{ s} {\rm d}s + \int_{ 0}^{t} R_{ s} {\rm d}s,\ t\in [0, T], \text{ in }H^{ -r_{ 2}}\left(\mathbb{ T}\right).
\end{equation}
Moreover, one has the representation
\begin{equation*}
\mathcal{ E}_{ t}= \int_{ 0}^{t} V(t,s)^{ \ast} R_{ s} {\rm d}s, \text{ in } \mathcal{ C}^{ -r_{ 1}},
\end{equation*}
where $V$ is given in \eqref{eq:Vts}.
\end{proposition}
We are now in position to prove Proposition~\ref{prop:unique_SPDE}:
\begin{proof}[Proof of Proposition~\ref{prop:unique_SPDE}]
The proof follows arguments similar to \cite{mitoma85} (see also \cite{lucon_stannat_2016}). Pathwise uniqueness of a solution to the second equation of \eqref{eq:limit_etas} is easy: let $\hat{ \eta}_{ 1},\hat{ \eta}_{ 2}\in \mathcal{ C} \left([0, T], H^{-r_{ 1}}(\mathbb{ T}^{ 2})\right)$ be two solutions in $H^{-r_{ 2}}(\mathbb{ T}^{ 2})$. Then, the difference $ \hat{ \eta}:= \hat{ \eta}_{ 1}- \hat{ \eta}_{ 2}$ satisfies \eqref{eq:SPDE_L2} in the case $R\equiv 0$, so that Proposition~\ref{prop:SPDE_L2} gives $ \hat{ \eta}\equiv 0$. Now turn to the pathwise uniqueness of a solution to the first equation of \eqref{eq:limit_etas}: let $ \eta\in \mathcal{ C} \left([0, T], H^{ -r_{ 1}} \left(\mathbb{ T}\right)\right)$ solution in $H^{ -r_{ 2}}(\mathbb{ T})$. Setting $h(t):= \int_{ 0}^{t} \mathcal{ L}_{ s}^{ (1), \ast} \eta_{ s} {\rm d}s= \eta_{ t}- \eta_{ 0} - \int_{ 0}^{t} \Theta^{ \ast} \hat{ \eta}_{ s} {\rm d}s- W_{ t}\in \mathcal{ C} \left([0, T], H^{ -r_{ 1}}\left(\mathbb{ T}\right)\right)$ and differentiating this quantity w.r.t. $t$, one obtains that (almost surely w.r.t. the randomness)
\begin{equation*}
\frac{ {\rm d}}{ {\rm d}t} h(t) = L_{ t}^{ (1), \ast} h(t) + K_{ t}^{ (1), \ast} h(t) + \mathcal{ L}_{ t}^{ (1), \ast} \left(\eta_{ 0} + \int_{ 0}^{t} \Theta^{ \ast} \hat{ \eta}_{ s}{\rm d}s+ W_{ t}\right),
\end{equation*}
where we recall the decomposition \eqref{eq:decomp_L1}. 

Set $R_{ t}:= K_{ t}^{ (1), \ast} h(t) + \mathcal{ L}_{ t}^{ (1), \ast} \left(\eta_{ 0} + \int_{ 0}^{t}  \Theta^{ \ast} \hat{ \eta}_{ s}{\rm d}s+W_{ t}\right)$. Since $\mathcal{ L}_{ t}^{ (1), \ast} \left(\eta_{ 0} + \int_{ 0}^{t} \Theta^{ \ast} \hat{ \eta}_{ s} {\rm d}s+ W_{ t}\right)$ belongs to $\mathcal{ C}([0, T], H^{ -r_{ 1}}(\mathbb{ T}))$ we focus on the regularity of the first term: we have for $0<s<t$, 
\begin{equation*}
\left\vert \left\langle h(t)-h(s)\, ,\, (K_{ t}^{ (1)}-  K_{ s}^{ (1)})f\right\rangle \right\vert \leq \left\Vert h(t)-h(s) \right\Vert_{ -r_{ 1}} \left\Vert (K_{ t}^{ (1)}-  K_{ s}^{ (1)})f \right\Vert_{ r_{ 1}},
\end{equation*}
and using the fact that $ \mu_{ t}$ has for $t>0$ a smooth density $ \mu_{ t}({\rm d}\theta)= p_{ t}(\theta) {\rm d}\theta$ (see e.g. \cite{daiPra96} or \cite[Prop.~7.1]{MR3207725})
\begin{align}
\left\Vert (K_{ t}^{ (1)}-  K_{ s}^{ (1)})f \right\Vert_{ r_{ 1}}^{ 2}&= \sum_{ k\leq r_{ 1}} \int_{ \mathbb{ T}} \left\vert \int_{ \mathbb{ T}} \partial_{ \theta}^{ k}\Gamma \left(\theta^{ \prime}, \theta\right)\partial_{ \theta}f(\theta^{ \prime}) (p_{ t}(\theta^{ \prime})-p_{ s}(\theta^{ \prime})) {\rm d}\theta^{ \prime}\right\vert^{ 2}{\rm d}\theta \nonumber\\
& \leq \left\Vert \partial_{ \theta}f \right\Vert_{ \infty} \sum_{ k\leq r_{ 1}} \int_{ \mathbb{ T}} \left( \int_{ \mathbb{ T}} \left\vert \partial_{ \theta}^{ k}\Gamma \left(\theta^{ \prime}, \theta\right) \right\vert \left\vert p_{ t}(\theta^{ \prime})-p_{ s}(\theta^{ \prime})) \right\vert {\rm d}\theta^{ \prime}\right)^{ 2}{\rm d}\theta. \label{aux:delta_p}
\end{align}
Since $ \left\Vert \partial_{ \theta}f \right\Vert_{ \infty}\leq \left\Vert f \right\Vert_{ \mathcal{ C}^{ 1}} \leq C \left\Vert f \right\Vert_{ 2}\leq C \left\Vert f \right\Vert_{ r_{ 1}}$ and since the quantity in \eqref{aux:delta_p} goes to $0$ as $t-s\to 0$, we conclude from this that $R\in \mathcal{ C}\left([0, T], H^{ -r_{ 1}}\left(\mathbb{ T}\right)\right)$. Remark also that for all $j\leq r_{ 1}$
\begin{align*}
\left\Vert \partial_{ \theta}^{ j}K_{ t}^{ (1)}f \right\Vert_{ \infty}&= \left\Vert \int_{ \mathbb{ T}} \partial_{ \theta}^{ j}\Gamma \left(\theta^{ \prime}, \theta\right)\partial_{ \theta}f(\theta^{ \prime}) \mu_{ t}({\rm d}\theta^{ \prime})\right\Vert_{ \infty}\leq C\left\Vert \partial_{ \theta}f \right\Vert_{ \infty},
\end{align*}
for a constant $C$ that only depends on $ \Gamma$. In particular, $ \left\Vert K_{ t}^{ (1)}f \right\Vert_{ \mathcal{ C}^{ r_{ 1}}} \leq C \left\Vert f \right\Vert_{ \mathcal{ C}^{ 1}} \leq C \left\Vert f \right\Vert_{ \mathcal{ C}^{ r_{ 1}}}$ and hence 
\begin{equation}
\label{eq:Kbound}
\left\Vert K_{ t}^{ (1)} \right\Vert_{ \mathcal{ C}^{ -r_{ 1}}}\leq C.
\end{equation}
We are now in position to apply Proposition~\ref{prop:SPDE_L1} to the case $ \mathcal{ E}=h$: $h$ is solution to
\begin{equation}
\label{eq:eq_ht}
h(t) = \int_{ 0}^{t} V^{ \ast}(t,u) \left( K_{ u}^{ (1), \ast} h(u) + \mathcal{ L}_{ u}^{ (1), \ast} \left(\eta_{ 0} + \int_{ 0}^{u} \Theta^{ \ast} \hat{ \eta}_{ v} {\rm d}v+ W_{ u}\right)\right) {\rm d}u,\ \text{ in } \mathcal{ C}^{ -r_{ 1}}.
\end{equation}
The main point of the proof is to see that $h$ solution of \eqref{eq:eq_ht} can be approximated by the converging sequence $(h_{ n})_{ n\geq1}$ defined recursively as follows
\begin{equation}
\label{eq:recurs_hn}
\begin{cases}
h_{ 1}(t)&= \int_{0}^{t} V^{ \ast}(t, u) \mathcal{ L}_{ u}^{ (1), \ast} ( \eta_{ 0} + \int_{0}^{u} \Theta^{ \ast} \hat{ \eta}_{ v} {\rm d}v + W_{ u})\dd u,\\
h_{ n}(t)&= \int_{0}^{ t}V^{ \ast}(t, u)( K^{ (1),\ast}_{ u} h_{ n-1}(u) + \mathcal{ L}_{ u}^{ (1), \ast} \left(\eta_{ 0} + \int_{ 0}^{u} \Theta^{ \ast} \hat{ \eta}_{ v} {\rm d}v+W_{ u}\right))\dd u,\ n\geq 2.
\end{cases} 
\end{equation}
Indeed, by the boundedness of the semigroup $V(t, u)$ and by \eqref{eq:Kbound}, we obtain that for all $0<u<t<T$, for all $f\in \mathcal{C}^{ r_{ 1}}$, for all $h\in \mathcal{C}^{ -r_{ 1}}$
\begin{align*}
\left\vert \left\langle h\, ,\, K_{ u}^{ (1)} V(t, u) f\right\rangle  \right\vert &\leq \left\Vert h \right\Vert_{ \mathcal{C}^{ -r_{ 1}}} \left\Vert {K_{ u}^{ (1)} V(t, u) f} \right\Vert_{  \mathcal{C}^{ r_{ 1}}}\leq C \left\Vert h \right\Vert_{  \mathcal{C}^{ -r_{ 1}}} \left\Vert {V(t, u) f} \right\Vert_{  \mathcal{C}^{ r_{ 1}}}\\
&\leq C \left\Vert h \right\Vert_{  \mathcal{C}^{ -r_{ 1}}} \left\Vert {f} \right\Vert_{  \mathcal{C}^{ r_{ 1}}}.
\end{align*}
Thus, the sequence $(h_{ n})_{ n\geq1}$ defined in \eqref{eq:recurs_hn} satisfies, for all $n\geq 2$
\begin{align*}
\left\Vert {h_{ n+1}(t) - h_{ n}(t)} \right\Vert_{ \mathcal{C}^{ -r_{ 1}}}&\leq C \int_{0}^{t} \left\Vert {h_{ n}(u)- h_{ n-1}(u)} \right\Vert_{ \mathcal{C}^{ -r_{ 1}}}\dd u.
\end{align*}
By an immediate recursion, for all $k\geq 1$,  $ \left\Vert {h_{ n+1+k}(t) - h_{ n+k}(t)} \right\Vert_{ \mathcal{C}^{ -r_{ 1}}}\leq C^{ k} \frac{ T^{ k}}{ k!}$, so that $(h_{ n})_{ n\geq1}$ is a Cauchy sequence in $\mathcal{C}([0, T], \mathcal{C}^{ -r_{ 1}})$ and thus converges to $h$, solution of \eqref{eq:eq_ht}. Turning back to $ \eta$ and writing $S(t):= \eta_{ 0} + \int_{ 0}^{t} \Theta^{ \ast} \hat{ \eta}_{ s} {\rm d}s+W_{ t}$, we obtain that $ \eta$ is uniquely written as
\begin{equation}
\label{eq:repr_eta}
\begin{split}
\eta_{ t} = \lim_{ n\to \infty} \Bigg\lbrace \eta_{ 0} + &\int_{ 0}^{t} \Theta^{ \ast} \hat{ \eta}_{ s} {\rm d}s+ W_{ t} + \int_{0}^{t} V^{ \ast}(t, t_{ 1})\mathcal{L}^{ (1),\ast}_{ t_{ 1}} S(t_{ 1}) \dd t_{ 1}\\ &+ \int_{0}^{t} \int_{0}^{t_{ 1}} V^{ \ast}(t, t_{ 1}) K_{ t_{ 1}}^{ (1), \ast} V^{ \ast}(t, t_{ 2}) \mathcal{L}^{ (1),\ast}_{ t_{ 2}} S(t_{ 2}) \dd t_{ 2} \dd t_{ 1} + \ldots \\
&+\int_{0}^{t} \int_{0}^{t_{ 1}} \cdots \int_{0}^{t_{ n-1}} V^{ \ast}(t, t_{ 1}) K_{ t_{ 1}}^{ (1), \ast}\cdots V^{ \ast}(t, t_{ n}) \mathcal{L}^{ (1),\ast}_{ t_{ n}} S(t_{ n})\dd t_{ n}\ldots\dd t_{ 2}\dd t_{ 1}\Bigg\rbrace.
\end{split}
\end{equation}
This proves pathwise uniqueness. But if one chooses another solution $ \tilde \eta$ defined on another probability space, with initial condition $\tilde \eta_{ 0}$ and noise $\tilde{W}$ with the same law as $(\eta_{ 0}, W)$, we obtain the same expression as above with $S$ replaced by $\tilde S(t) = \tilde\eta_{ 0} + \int_{ 0}^{t} \Theta^{ \ast} \hat{ \eta}_{ s}{\rm d}s+\tilde{W}_{ t}$. Since $S$ and $\tilde S$ have then the same law, uniqueness in law follows from \eqref{eq:repr_eta}. Proposition~\ref{prop:unique_SPDE} is proven.
\end{proof}

\section{On the choice of renormalisation}
\label{sec:renorm}
The question we ask here concerns the influence of a different choice of renormalisation in \eqref{eq:wips}: namely, it would also make sense to renormalise the interaction in \eqref{eq:wips} by the exact degree
\begin{equation}
\label{eq:dni}
d_{ n,i}:= \sum_{ j=1}^{ n} \xi_{ij}^{ (n)},\ i=1,\ldots, n,
\end{equation}
rather than the expected degree $ \mathbb{ E} \left[d_{ n,i}\right]=np_{ n}$ as we have done in \eqref{eq:wips}. Hence, define
\begin{equation}
\label{eq:wips_d}
\dd \theta^{i,n}_{ d,t} = \frac{1}{d_{ n,i}}\sum_{j=1}^n \xi_{ij}^{(n)}\Gamma \left(\theta^{i,n}_{ d,t}, \theta^{j,n}_{ d,t} \right)\dd t+\dd B^i_t, \quad 0 < t \leq T, \quad i= 1, \dots, n,
\end{equation}
The subscript $d$ is here to specify this choice of renormalisation by the degree. Define accordingly $ \mu^{ n}_{ d}$, $ \eta^{ n}_{ d}$ and $ \hat{ \eta}^{ n}_{ d}$ the respective global empirical measure and fluctuation processes. The corresponding local empirical and fluctuation processes become naturally
\begin{align}
\mu^{ n,l}_{ d}&:= \frac{ 1}{ d_{ n, l}} \sum_{ k=1}^{ n} \xi_{lk}^{ (n)} \delta_{ \theta^{ k,n}_{ d}},\ l=1,\ldots, n \label{def:munl_d}\\
\zeta^{ n,l}_{ d}&:= \sqrt{ d_{ n,l}}\left( \mu^{ n,l}_{ d} - \mu\right). \label{eq:zetas_d}
\end{align}
It is quite clear that this different choice of renormalisation does not change anything concerning the Law of Large Number on the global and local empirical measures, as $ \mathbb{ P}$-a.s. $ \frac{ d_{ n,i}}{ np_{ n}} \to 1$ as $n\to\infty$, for all $i=1,\ldots, n$: Theorem~\ref{th:conv_empmeas}  remains identical for $ \mu^{ n}_{ d}$ and $ \mu^{ n,l}_{ d}$. It is a priori not clear if this change in renormalisation would influence the fluctuation results. We suppose here for simplicity that the initial condition is chosen to be independent from the graph. The main result of this paragraph is the following.
\begin{proposition}
\label{prop:renorm}
Theorem~\ref{th: limit eta independent} remains unchanged when one replaces \eqref{eq:wips} by \eqref{eq:wips_d}: under the same hypotheses, the process $ \eta^{ n}_{ d}$ converges to the same limit $ \eta$ given by \eqref{eq:limit_eta_only}. Moreover, under the same hypotheses as Theorem~\ref{th:local_fluct}, the joint process $(\zeta^{ n,1}_{ d}, \zeta^{ n,2}_{ d})$ converges in law w.r.t. $ \mathbf{ P} \otimes \mathbb{ P}$ as. $n\to\infty$ towards $ \left(\zeta^{ 1}+ D^{ 1} \mu, \zeta^{ 2}+D^{ 2}\mu\right)$, where $ \left(\zeta^{ 1}, \zeta^{ 2}\right)$ are given in \eqref{eq:SDPEs_zetas_eta} and $(D^{ 1}, D^{ 2})$ are independent real Gaussian variables $D^{ l}\sim \mathcal{ N}(0, 1-p)$, with $\left(\zeta^{ 1}, \zeta^{ 2}\right)$ and $(D^{ 1}, D^{ 2})$ independent.
\end{proposition}
\begin{proof}[Proof of Proposition~\ref{prop:renorm}]
Let us first consider the convergence of the global fluctuation process $ \eta^{ n}_{ d}$. Perform the same Ito decomposition for $ \mu^{ n}_{ d}$ as in the proof of Lemma~\ref{lem:mu^n_t}: for $f$ regular, one obtains
\begin{align*}
\left\langle \mu^n_{ d,t}\, ,\, f\right\rangle =& \left\langle \mu^n_{ d,0}\, ,\, f\right\rangle + \int_0^t \left\langle \mu^{ n}_{ d,s}\, ,\, \frac 12 \partial_\theta^2 f + (\mu^{ n}_{ d,s} * \Gamma)  \partial_\theta f \right\rangle  \dd s + \left\langle M^{ n}_{ d,t}\, ,\, f\right\rangle \\
&+ \int_0^t \frac{1}{n^2}\sum_{i,j=1}^n \hat{ \xi}_{ij}^{ (n)} \Gamma\left(\theta^{i,n}_{ d,s}, \theta^{j,n}_{ d,s}\right)\partial_\theta f\left(\theta^{i,n}_{ d,s}\right) \dd s + D_{ n,t}(f).
\end{align*}
The only thing that changes from Lemma~\ref{lem:mu^n_t} is the apparition of the supplementary drift term
\begin{equation}
\label{eq:Dnt}
D_{ n,t}(f):= \int_0^t \frac{1}{n^{ 2}}\sum_{i,j=1}^n \xi_{ij}^{ (n)}\left(\frac{ n }{ d_{ n,i}}- \frac{ 1}{ p_{ n}}\right)\Gamma\left(\theta^{i,n}_s, \theta^{j,n}_s\right)\partial_\theta f\left(\theta^{i,n}_s\right) \dd s .
\end{equation}
All that remains is to show that $ \sqrt{ n} D_{ n,t}$ goes to $0$ as $n\to\infty$.
Some auxiliary results first: introduce the random variables $ \chi_{n}:= \sup_{ i=1,\ldots, n} \frac{ np_{ n}}{ d_{ n,i}}$ and $\rho_{ n}:= \sup_{ i=1, \ldots, n} \left\vert \frac{ 1}{ n} \sum_{ k=1}^{ n} \hat{ \xi}_{ik}^{ (n)} \right\vert$. For all $i=1, \ldots, n$, an application of Bernstein inequality leads to 
\begin{align*}
 \mathbb{ P} \left( \frac{ np_{ n}}{ d_{ n,i}} \geq 2\right)&= \mathbb{ P} \left(\frac{ 1}{ n} \sum_{ j=1}^{ n} \hat{ \xi}_{ij}^{ (n)}\leq - \frac{ 1}{ 2}\right)\leq \exp \left( - \frac{ 1}{ 8} \frac{ n p_{ n}}{7/6}\right)\leq \exp( - n^{ 3/4}),
\end{align*}
under the dilution condition \eqref{eq:dilution_cond}. A union bound on $i=1,\ldots, n$ together with Borel Cantelli Lemma gives that
\begin{equation}
\label{eq:bound_Chin}
\limsup_{ n\to\infty} \chi_{ n} \leq 2, \ \mathbb{ P}-\text{ a.s.}
\end{equation}
 In a same way, it is easy to prove that under \eqref{eq:dilution_cond} that, $ \mathbb{ P}$-a.s., for all $ \varepsilon\in \left(0, \frac{ 1}{ 2}\right)$,
\begin{equation}
\label{eq:bound_rhon}
\limsup_{ n\to\infty} (np_{ n})^{ \frac{ 1}{ 2}- \varepsilon} \rho_{ n} \leq 1.
\end{equation}
Notice that
\begin{equation}
\label{eq:ident_dni_hatxi}
\left(\frac{ n}{ d_{ n,i}}- \frac{ 1}{ p_{ n}}\right)= -\frac{ n}{ d_{ n,i}} \left(\frac{1}{ n} \sum_{ k=1}^{ n} \hat{ \xi}_{ik}^{ (n)}\right),
\end{equation}
so that
\begin{equation}
\label{eq:Dnt_expand1}
D_{ n,t}= -\int_0^t \frac{1}{n^{ 3}}\sum_{i,j,k=1}^n \xi_{ij}^{ (n)} \frac{ n}{ d_{ n,i}} \hat{ \xi}_{ik}^{ (n)}\Gamma\left(\theta^{i,n}_s, \theta^{j,n}_s\right)\partial_\theta f\left(\theta^{i,n}_s\right) \dd s .
\end{equation}
Write $ \frac{ \xi_{ij}^{ (n)}}{ p_{ n}}= \hat{ \xi}_{ij}^{ (n)} +1$, so that one obtains
\begin{multline*}
D_{ n,t}= -\int_0^t \frac{1}{n^{ 3}}\sum_{i,j,k=1}^n  \frac{ np_{ n}}{ d_{ n,i}} \hat{ \xi}_{ij}^{ (n)} \hat{ \xi}_{ik}^{ (n)}\Gamma\left(\theta^{i,n}_s, \theta^{j,n}_s\right)\partial_\theta f\left(\theta^{i,n}_s\right) \dd s \\
-\int_0^t \frac{1}{n^{ 3}}\sum_{i,j,k=1}^n \frac{ np_{ n}}{ d_{ n,i}} \hat{ \xi}_{ik}^{ (n)}\Gamma\left(\theta^{i,n}_s, \theta^{j,n}_s\right)\partial_\theta f\left(\theta^{i,n}_s\right) \dd s.
\end{multline*}
Using again \eqref{eq:ident_dni_hatxi} in the second term above, we obtain
\begin{align*}
D_{ n,t}&= -\int_0^t \frac{1}{n^{ 3}}\sum_{i,j,k=1}^n  \frac{ np_{ n}}{ d_{ n,i}} \hat{ \xi}_{ij}^{ (n)} \hat{ \xi}_{ik}^{ (n)}\Gamma\left(\theta^{i,n}_s, \theta^{j,n}_s\right)\partial_\theta f\left(\theta^{i,n}_s\right) \dd s \nonumber\\
&\quad +\int_0^t \frac{1}{n}\sum_{i=1}^n \frac{ np_{ n}}{ d_{ n,i}} \left( \frac{ 1}{ n} \sum_{ k=1}^{ n}\hat{ \xi}_{ik}^{ (n)} \right)^{ 2}\left( \frac{ 1}{ n} \sum_{ j=1}^{ n} \Gamma\left(\theta^{i,n}_s, \theta^{j,n}_s\right)\right)\partial_\theta f\left(\theta^{i,n}_s\right) \dd s \nonumber\\
&\quad -\int_0^t \frac{1}{n^{ 2}}\sum_{i,k=1}^n \hat{ \xi}_{ik}^{ (n)} \left( \frac{ 1}{ n} \sum_{ j=1}^{ n} \Gamma\left(\theta^{i,n}_s, \theta^{j,n}_s\right) \right)\partial_\theta f\left(\theta^{i,n}_s\right) \dd s, \nonumber\\
&:= D_{ n,t}^{ (1)}+ D_{ n,t}^{ (2)}+ D_{ n,t}^{ (3)}.
\end{align*}
Concentrate on the first term above:
\begin{align*}
D_{ n,t}^{ (1)}=-\int_0^t \frac{1}{n^{ 2}}\sum_{i,j=1}^n  \hat{ \xi}_{ij}^{ (n)} \left( \frac{ 1}{ n}\sum_{ k=1}^{ n}\hat{ \xi}_{ik}^{ (n)}\right) \frac{ np_{ n}}{ d_{ n,i}}\Gamma\left(\theta^{i,n}_s, \theta^{j,n}_s\right)\partial_\theta f\left(\theta^{i,n}_s\right) \dd s.
\end{align*}
Informally, one expects that $\frac{ np_{ n}}{ d_{ n,i}} \approx 1$. Nonetheless, this term is still random and one cannot rule out the unlikely possibility that it takes extreme values. The point is to include this term within the Grothendieck estimate: applying \eqref{eq:classical_grothendieck} to $a_{ij}:= \hat{ \xi}_{ij}^{ (n)} \left( \frac{ 1}{ n}\sum_{ k=1}^{ n}\hat{ \xi}_{ik}^{ (n)}\right)$, we obtain that, for some constant that depends on $(\Gamma, f)$, $ \mathbf{ P}\otimes \mathbb{ P}$-a.s.
\begin{align*}
\left\vert \frac{ D_{ n,t}^{ (1)}}{ \chi_{ n}} \right\vert \leq K \sup_{ s_{ i}, t_{ j}} \frac{ 1}{ n^{ 2}} \sum_{ i,j=1}^{ n} \hat{ \xi}_{ij}^{ (n)} \left( \frac{ 1}{ n}\sum_{ k=1}^{ n}\hat{ \xi}_{ik}^{ (n)}\right) s_{ i} t_{ j} \leq K \sup_{ s_{ i}, t_{ j}, r_{ k}} \frac{ 1}{ n^{ 3}} \sum_{ i,j,k=1}^{ n} \hat{ \xi}_{ij}^{ (n)} \hat{ \xi}_{ik}^{ (n)} s_{ i} t_{ j} r_{ k}.
\end{align*}
The main conclusion of this calculation is that one has $ \mathbf{ P}\otimes \mathbb{ P}$-a.s.
\begin{equation}
\label{eq:Dn_first_bound}
\left\vert D_{ n,t}^{ (1)} \right\vert \leq K S_{ n}^{ \ijik} \chi_{ n}.
\end{equation}
Combining \eqref{eq:Dn_first_bound}, \eqref{eq:bound SnT} and \eqref{eq:bound_Chin}, we see that $D_{ n, t}^{ (1)}$ is $ \mathbf{ P}\otimes \mathbb{ P}$-a.s. of order $ \frac{ C}{ np_{ n}^{ 2}}$, so that $ \sqrt{ n} D_{ n, t}^{ (1)} \to 0$ as $n\to\infty$, under \eqref{eq:dilution_cond}. For the term $D_{ n,t}^{ (2)}$, we have the rough bound $\left\vert D_{ n, t}^{ (2)} \right\vert \leq \chi_{ n} \rho_{ n}^{ 2} \left\Vert \Gamma \right\Vert_{ \infty} \left\Vert \partial_{ \theta}f \right\Vert_{ \infty}$ which, by \eqref{eq:bound_Chin} and \eqref{eq:bound_rhon}, is $ \mathbf{ P}\otimes \mathbb{ P}$-a.s. of order $(np_{n})^{ -1+2 \varepsilon}$, so that $ \sqrt{ n} D_{ n, t}^{ (2)}$ is of order $ n^{ -1/4 + 3 \varepsilon/2}$ under \eqref{eq:dilution_cond}. Choosing $ \varepsilon< 1/6$ gives that $\sqrt{ n} D_{ n, t}^{ (2)}\to 0$ as $n\to\infty$. The last term gives nothing else than 
\begin{equation*}
\sqrt{ n} D_{ n, t}^{ (3)}= - \int_{ 0}^{t} \left\langle \hat{ \eta}_{ s}^{ n} \left({\rm d}\theta_{ 1}, {\rm d}\theta_{ 2}\right)\, ,\, \partial_{ \theta}f(\theta_{ 1}) \left(\Gamma \ast \mu_{ s}^{ n}\right)(\theta_{ 2})\right\rangle {\rm d}s,
\end{equation*}
which converges to $0$ under the present hypotheses that the initial condition is chosen to be independent on the graph (as one can prove along the same lines as what has been done before that the limit $ \hat{ \eta}\equiv 0$ in this case). Note that all the previous calculations depend on the specific choice of test function $f$. It suffices to take a dense set of such test functions to realise that $D_{ n,t}\equiv 0$, $ \mathbf{ P}\otimes \mathbb{ P}$-a.s.

\medskip

Turn now to the case of local fluctuations: we only sketch the proof and leave the main details to the reader. Define first some auxiliary local processes, for $l=1, \ldots, n$
\begin{align}
\label{def:zetas_mix}
\tilde{ \mu}^{ n, l}_{ d}:= \frac{ 1}{ np_{ n}} \sum_{ k=1}^{ n} \xi_{ik}^{ (n)} \delta_{ \theta^{ k, n}_{ d}},\\
\tilde{ \zeta}^{ n,l}_{ d}:= \sqrt{ n p_{ n}} \left( \tilde{\mu}^{ n, l}_{ d} - \mu \right).
\end{align}
Think of $(\tilde{ \mu}^{ n, l}_{ d}, \tilde{ \zeta}^{ n,l}_{ d})$ as intermediate steps between \eqref{def:mu_nl} and \eqref{def:munl_d} (resp. \eqref{def:zetas} and \eqref{eq:zetas_d}): they are empirical measures on the system \eqref{eq:wips_d} but we are still conserving the same expected renormalisation $np_{ n}$. The same analysis as for $ \eta^{ n}_{ d}$ (based on Ito decomposition and proof that the remaining terms vanish as $n\to\infty$) gives that, under the hypotheses of Theorem~\ref{th:local_fluct}, the limit of $ \left(\tilde{ \zeta}^{ n,1}_{ d}, \tilde{ \zeta}^{ n,2}_{ d}, \eta^{ n}_{ d}\right)$ remains the same, i.e. is given by \eqref{eq:SDPEs_zetas_eta}. An easy calculation relates the two processes $ \zeta^{ n,l}_{ d}$ and $ \tilde{ \zeta}^{ n,l}_{ d}$: for $l=1,2$ and $t\in [0, T]$,
\begin{equation}
\label{eq:zetatilde_zeta}
\zeta^{ n, l}_{ d}= \sqrt{ \frac{ np_{ n}}{ d_{ n, l}}} \left( \tilde{ \zeta}^{ n,l}_{ d} - \sqrt{ np_{ n}}\left( \frac{ 1}{ n}\sum_{ j=1}^{ n} \hat{\xi}_{lj}^{ (n)}\right) \mu\right):=\sqrt{ \frac{ np_{ n}}{ d_{ n, l}}} \left( \tilde{ \zeta}^{ n,l}_{ d} - D^{ n, l} \mu\right).
\end{equation}
One can interpret the decomposition \eqref{eq:zetatilde_zeta} as the sum of a fluctuation term coming from the dynamics and a term coming from the fluctuation of the degree itself. From the easy convergence of the joint renormalised degrees $(D^{ n,1}, D^{ n, 2})$ (in law w.r.t. $ \mathbb{ P}_{ g}$) towards some independent $(D^{ 1}, D^{ 2})$ with $D^{ l}\sim \mathcal{ N}(0, 1-p)$, the fact that $ \sqrt{ \frac{ np_{ n}}{ d_{ n, l}}} \to 1$ as $n\to\infty$, $ \mathbb{ P}_{ g}$-a.s., and the convergence of $(\tilde{ \zeta}^{ n,1}_{ d}, \tilde{ \zeta}^{ n,2}_{ d})$ (in law w.r.t. $ \mathbf{ P}$, almost-surely w.r.t. $ \mathbb{ P}$), we deduce immediately the result of Proposition~\ref{prop:renorm}.
\end{proof}

\section*{Acknowledgements}
C.P. thanks Guillaume Aubrun for fruitful discussions and in particular for his help for the proof of Proposition~\ref{prop:concentration_SnT}. E.L. would like to thank warmly J\'er\^ome Dedecker and Nathael Gozlan for useful discussions and acknowledges the support of ANR-19-CE40-002 (ChaMaNe). C.P acknowledges the support of ANR-17-CE40-0030 (EFI). E.L and C.P. both acknowledge the support of ANR–19–CE40–0023 (PERISTOCH).

\end{document}